\documentclass[english,10pt]{amsart}
\usepackage[english]{babel}

\usepackage[matrix,arrow]{xy}
\xyoption{all}
\usepackage{amscd,amssymb,amsfonts,amsmath}

\usepackage{graphics}
\usepackage{epsfig}
\usepackage{mathrsfs}

\newcommand\mypagesizel{
\textwidth= 6.5in
\textheight=9in
\voffset-.55in
\hoffset -0.75in
\marginparwidth=56pt
}

\newcommand{\Pic}{\textup{Pic}}

\newcommand{\Supp}{\textup{Supp}}

\newcommand{\RatCurves}{\textup{RatCurves}}

\newcommand{\Chow}{\textup{Chow}}
\newcommand{\codim}{\textup{codim}}

\newcommand{\p}[0]{{\mathbb P}}
\newcommand{\rat}[0]{\operatorname{RatCurves}^n}

\newcommand{\NE}{\overline{\textup{NE}}}

\newcommand{\epsi}{\varepsilon}
\renewcommand{\phi}{\varphi}
\newcommand{\into}{\hookrightarrow}
\newcommand{\map}{\dashrightarrow}

\renewcommand{\le}{\leqslant}
\renewcommand{\ge}{\geqslant}

\newcommand{\bN}{\textbf{N}}

\newcommand{\bQ}{\textup{\textbf{Q}}}

\newcommand{\bZ}{\textup{\textbf{Z}}}

\newcommand{\cC}{\mathcal{C}}

\newcommand{\cO}{\mathcal{O}}

\newcommand{\cQ}{\mathcal{Q}}

\newcommand{\sA}{\mathscr{A}}

\newcommand{\sE}{\mathscr{E}}
\newcommand{\sF}{\mathscr{F}}
\newcommand{\sG}{\mathscr{G}}
\newcommand{\sH}{\mathscr{H}}
\newcommand{\sI}{\mathscr{I}}

\newcommand{\sK}{\mathscr{K}}
\newcommand{\sL}{\mathscr{L}}
\newcommand{\sM}{\mathscr{M}}
\newcommand{\sN}{\mathscr{N}}
\newcommand{\sO}{\mathscr{O}}

\newcommand{\sQ}{\mathscr{Q}}
\newcommand{\sR}{\mathscr{R}}

\newcommand{\sT}{\mathscr{T}}

\newcommand{\sV}{\mathscr{V}}
\newcommand{\sW}{\mathscr{W}}

\mypagesizel

\newtheorem{thm}{Theorem}[section]
\newtheorem*{thm*}{Theorem}

\newtheorem{question}[thm]{Question}
\newtheorem{lemma}[thm]{Lemma}
\newtheorem{cor}[thm]{Corollary}

\newtheorem{prop}[thm]{Proposition}

\theoremstyle{definition}
\newtheorem{defn}[thm]{Definition}

\newtheorem{say}[thm]{}

\newtheorem{const}[thm]{Construction}

\newtheorem{notation}[thm]{Notation}

\newtheorem{defn-thm}[thm]{Definition-Theorem}

\newtheorem{rem}[thm]{Remark}

\theoremstyle{remark}

\newtheorem{step}{Step}

\newtheorem*{not-and-def}{Notation and definitions}

\numberwithin{equation}{section}

\begin{document}

\title[On Fano foliations]{On Fano foliations}

\author{Carolina \textsc{Araujo}} 

\address{\noindent Carolina Araujo: IMPA, Estrada Dona Castorina 110, Rio de
  Janeiro, 22460-320, Brazil} 

\email{caraujo@impa.br}

\author{St\'ephane \textsc{Druel}}

\address{St\'ephane Druel: Institut Fourier, UMR 5582 du
  CNRS, Universit\'e Grenoble 1, BP 74, 38402 Saint Martin
  d'H\`eres, France} 

\email{druel@ujf-grenoble.fr}

\thanks{The first named author was partially supported by CNPq and Faperj Research 
  Fellowships}

\thanks{The second named author was partially supported by the CLASS project of the 
A.N.R}


\begin{abstract}
In this paper we address Fano foliations on complex projective varieties.
These are foliations $\sF$ whose anti-canonical class $-K_{\sF}$ is ample. 
We focus our attention on a special class of Fano foliations, namely \emph{del Pezzo} foliations on complex projective manifolds.
We show that these foliations are algebraically integrable, with one exceptional case when the ambient space 
is $\p^n$. 
We also provide a classification of del Pezzo foliations with mild singularities.
\end{abstract}

\maketitle

\tableofcontents


\section{Introduction}

In the last few decades, much progress has been made in the classification of complex projective varieties.
The general viewpoint is that complex projective manifolds $X$ should be classified according to the behavior of their canonical class $K_X$.
As a result of the minimal model program, we know that every complex projective manifold can be build up from
3 classes of (possibly singular) projective varieties, namely,  varieties $X$ for which  $K_X$ is $\bQ$-Cartier, and satisfies
$K_X<0$, $K_X\equiv 0$ or $K_X>0$. 
Projective manifolds $X$ whose anti-canonical class $-K_X$ is ample are called \emph{Fano manifolds}, and are quite special.
For instance, Fano manifolds are known to be rationally connected (see \cite{campana} and  \cite{kmm3}).

One defines the index  $\iota_X$ of a Fano manifold $X$ to be the largest integer dividing $-K_X$ in $\Pic(X)$. 
A classical result of Kobayachi-Ochiai's asserts that $\iota_X\leq \dim X+1$, and equality 
holds if and only if $X\simeq \p^n$. Moreover, $\iota_X= \dim X$ if and only if $X$ is a 
quadric hypersurface (\cite{kobayashi_ochiai}). 
Fano manifolds whose index satisfies $\iota_X= \dim X-1$ were classified by Fujita in \cite{fujita82a} and  \cite{fujita82b}. 
These are called \emph{del Pezzo} manifolds.  
The philosophy behind these results is that 
Fano manifolds with high index are the simplest projective manifolds.

Similar ideas can be applied in the context of \emph{foliations} on complex projective manifolds.
If $\sF\subsetneq T_X$ is a foliation on a  complex projective manifold $X$, we define its canonical class to be 
$K_{\sF}=-c_1(\sF)$. In analogy with the case of projective manifolds, one expects the numerical properties of $K_{\sF}$
to reflect geometric aspects of $\sF$. 
In fact, ideas from the minimal model program have been successfully applied to the theory of  foliations 
(see for instance \cite{brunella} and \cite{mcquillan08}), and led to a birational classification in the case of rank one foliations on surfaces
(\cite{brunella}).  
More recently, Loray, Pereira and Touzet have investigated the structure of codimension $1$ foliations with 
 $K_{\sF}\equiv 0$  in \cite{loray_pereira_touzet}.

In this paper we propose to investigate \emph{Fano foliations} on complex projective manifolds.
These are foliations $\sF\subsetneq T_X$ whose anti-canonical class $-K_{\sF}$ is ample 
(see Section~\ref{section:foliations} for details). 
As in the case of Fano manifolds, we expect Fano foliations to present very special behavior.
This is the case for instance if the rank of $\sF$ is $1$, i.e., $\sF$ is an ample invertible subsheaf
of $T_X$. By Wahl's Theorem \cite{wahl83}, this can only happen if $(X,\sF)\simeq \big(\p^n,\sO(1)\big)$.

Guided by the theory of Fano manifolds, we define the index $\iota_{\sF}$ of a foliation $\sF$ on a 
complex projective manifold $X$ to be the largest integer dividing $-K_{\sF}$ in $\Pic(X)$. 
The expected philosophy is that Fano foliations with high index are the simplest ones. 
For instance, when $X=\p^n$,  the index of a foliation $\sF\subsetneq T_{\p^n}$  of rank $r$ 
satisfies $\iota_{\sF}\leq r$. By \cite[Th\'eor\`eme 3.8]{cerveau_deserti}, 
equality holds if and only if $\sF$ is induced by a  linear projection 
$\p^n \dashrightarrow \p^{n-r}$, i.e., it comes from  the family $r$-planes in $\p^n$
containing a fixed $(r-1)$-plane. 
Fano foliations in $\sF\subsetneq T_{\p^n}$ satisfying $\iota_{\sF} = r-1$ were 
classified in 
\cite[Theorem 6.2]{loray_pereira_touzet_2}. They fall into one of the following two classes. 
\begin{enumerate}
\item Either $\sF$ is induced by a dominant rational  map $\p^n\dashrightarrow \p(1^{n-r},2)$,
defined by $n-r$ linear forms and one quadratic form, or 
\item $\sF$ is the linear pullback of a foliation on $\p^{n-r+1}$ induced by a global 
holomorphic vector field.
\end{enumerate}

In analogy with Kobayachi-Ochiai's theorem, we have the following result.

\begin{thm*}[{\cite[Theorem 1.1]{adk08}}] Let $\sF\subsetneq T_X$ be a Fano foliation of rank $r$ on a
complex projective manifold $X$. 
Then $\iota_{\sF}\leq r$, and equality holds only if $X\cong \p^n$.
\end{thm*}

We say that a Fano foliation $\sF\subsetneq T_X$ of rank $r$ on a complex projective manifold $X$
is a  \emph{del Pezzo foliation} if $\iota_{\sF} = r-1$.
Ultimately we would like to classify del Pezzo foliations.
In addition to the above mentioned foliations on $\p^n$, 
we know examples of del Pezzo foliations of any rank on quadric hypersurfaces, 
del Pezzo foliations of rank $2$ on certain Grassmannians, 
and del Pezzo foliations of rank $2$ and $3$ on $\p^m$-bundles over $\p^l$. 
These examples are described in Sections~\ref{section:examples} and 
\ref{section:examples_pn_bundles}.

We note that the generic del Pezzo foliation on $\p^n$ of type (2) above does not have algebraic leaves.
Our first main result says that this is the only del Pezzo foliation that is not algebraically integrable.
We also describe the geometry of the general leaf in all other cases.

\begin{thm}\label{thma}
Let $\sF\subsetneq T_X$ be a del Pezzo foliation on a complex projective manifold $X\not\simeq \p^n$.
Then $\sF$ is algebraically integrable, and its general leaves are rationally connected.
\end{thm}

One of the key ingredients in the proof of Theorem~\ref{thma} is the following
criterion by Bogomolov and McQuillan for a foliation to be algebraically integrable
with rationally connected general leaf.

\begin{thm}[{\cite[Theorem 0.1]{bogomolov_mcquillan01}, \cite[Theorem 1]{kebekus_solaconde_toma07}}] \label{bogomolov_mcquillan}
Let $X$ be a normal complex projective variety, and $\sF$ a  foliation on $X$.
Let $C \subset X$ be a complete curve disjoint from the singular loci of $X$ and $\sF$.
Suppose that the restriction $\sF|_C$ is an ample vector bundle on $C$.
Then the leaf of $\sF$ through any point of $C$ is an algebraic variety, and the 
leaf of $\sF$ through a general point of $C$ is moreover rationally connected.
\end{thm}

Given a del Pezzo foliation $\sF\subsetneq T_X$ on a complex projective manifold $X$, 
it is not clear a priori how to find a curve $C\subset X$ satisfying the hypothesis of 
Theorem~\ref{bogomolov_mcquillan}.
Instead, in order to prove Theorem~\ref{thma}, we will apply Theorem~\ref{bogomolov_mcquillan}
in several steps.
First we construct suitable subfoliations $\sH\subset \sF$
for which we can prove algebraic integrability and rationally connectedness
of general leaves.
Next we consider the the closure $W$ in $\Chow(X)$ of the 
subvariety parametrizing general leaves of $\sH$, as explained in
Section~\ref{section:algebraic_foliations}.
We then apply Theorem~\ref{bogomolov_mcquillan} to the foliation on $W$ induced by 
$\sF$.

In the course of our study of Fano foliations, we were led to deal with singularities of foliations.
We introduce new notions of singularities for foliations,
inspired by the theory of singularities of pairs, developed in the context of the minimal model program. 
In order to explain this, let $\sF\subsetneq T_X$ be an algebraically integrable foliation on a
complex projective manifold $X$, and denote by $i:\tilde F\to  X$ 
the normalization of the closure of a general leaf of $\sF$. 
Then there is an effective Weil divisor $\tilde \Delta$ on $\tilde F$ such that
$-K_{\tilde F}  =i^*(-K_{\sF}) + \tilde \Delta$. 
We call the pair $( \tilde F,  \tilde \Delta)$ a general \emph{log leaf} of $\sF$.  We say that 
$\sF$ has \emph{log canonical singularities along a general leaf} if 
$(\tilde F,\tilde \Delta)$ is log canonical (see Section~\ref{section:algebraic_foliations} for details).
Algebraically integrable Fano foliations having log canonical singularities along a general leaf
have a very special property: there is a common point contained in the closure of a general leaf
(see  Proposition~\ref{lemma:common_point}). 
This property is useful to derive classification results under some restrictions on the
singularities of $\sF$, such as the following (see also Theorem~\ref{thmb'}).

\begin{thm}\label{thmb}
Let $\sF\subsetneq T_X$ be a del Pezzo foliation of rank $r$ on a complex projective manifold $X\not\simeq \p^n$.
Suppose that $\sF$  has log canonical singularities and is locally free along a general leaf. 
Then either $\rho(X)=1$, or $r\leq 3$, $X$ is a $\p^m$-bundle over $\p^l$ and $\sF\not\subset T_{X/\p^l}$.
\end{thm}

Notice that a del Pezzo foliation $\sF$ on $X\not\simeq \p^n$ is algebraically integrable by Theorem~\ref{thma}. 
Hence it makes sense to ask that
$\sF$  has log canonical singularities along a general leaf in Theorem~\ref{thmb} above. 
We remark that del Pezzo foliations of codimension $1$ on Fano manifolds with Picard 
number $1$ were classified in \cite[Proposition 3.7]{loray_pereira_touzet_2}.

Theorem~\ref{thmb} raises the problem of classifying del Pezzo foliations on $\p^m$-bundles 
$\pi:X\to \p^l$. If $m=1$, then $X\simeq \p^1\times \p^l$, and $\sF$ is the pullback via $\pi$ of a  
foliation $\sO(1)^{\oplus i}\subset T_{\p^l}$ for some $i\in\{1,2\}$ (see \ref{say:special_cases}).
For $m\geq 2$,  we have the following result (see Theorems~\ref{thm:description} and \ref{thm:classification} for more details). 

\begin{thm}\label{thmc}
Let $\sF\subsetneq T_X$ be a del Pezzo foliation on a $\p^m$-bundle $\pi:X\to \p^l$, with $m\geq 2$.
Suppose that $\sF\not\subset T_{X/\p^l}$. Then there is an exact sequence of vector bundles 
$0\to \sK\to \sE\to \sQ \to 0$ on $\p^l$ such that $X\simeq \p_{\p^l}(\sE)$, and $\sF$ is the pullback via the relative linear projection $X\map Z=\p_{\p^l}(\sK)$ of a foliation $q^*\det(\sQ)\subset T_Z$. 
Here $q:Z\to \p^l$ denotes the natural projection.  
Moreover, one of the following holds.
			\begin{enumerate}
				\item $l=1$, $\sQ\simeq \sO(1)$, $\sK$ is an ample vector bundle
					such that $\sK\not\simeq\sO_{\p^1}(a)^{\oplus m}$ for any integer $a$,
					and $\sE\simeq\sQ \oplus\sK$ ($r_{\sF}=2$).
				\item $l=1$, $\sQ\simeq \sO(2)$, $\sK\simeq\sO_{\p^1}(a)^{\oplus m}$ 
					for some integer $a \ge 1$,
					and $\sE\simeq\sQ \oplus\sK$ ($r_{\sF}=2$).
				\item $l=1$, $\sQ\simeq \sO(1)\oplus \sO(1)$, 
					$\sK\simeq\sO_{\p^1}(a)^{\oplus m-1}$ 
					for some integer $a \ge 1$,
					and $\sE\simeq\sQ \oplus\sK$ ($r_{\sF}=3$).
				\item $l \ge 2$, $\sQ\simeq \sO(1)$, and $\sK$ is $V$-equivariant 	
					for some $V\in H^0\big(\p^l,T_{\p^l}\otimes\sO(-1)\big)\setminus \{0\}$
					($r_{\sF}=2$).
			\end{enumerate}
Conversely, given $\sK$, $\sE$ and $\sQ$ satisfying any of the conditions above, there exists a del Pezzo foliation of that type.
\end{thm}

The paper is organized as follows. In Section~\ref{section:foliations} we introduce the basic notions 
concerning foliations and Pfaff fields on varieties. 
In Section~\ref{section:algebraic_foliations} we focus on algebraically integrable foliations, and
develop notions of singularities for these foliations.
In Section~\ref{section:examples} we describe examples of Fano foliations on Fano manifolds with Picard number $1$.
In Section~\ref{section:relative_-K} we study the relative anti-canonical bundle of a fibration, and provide
applications to the theory of Fano foliations.
In Section~\ref{section:rat_curves} we recall some results from the theory of rational curves on varieties, and explain 
how they apply to foliations. 
In Section~\ref{section:thma} we prove Theorem~\ref{thma}.
In Section~\ref{section:thmb} we address the problem of classifying Fano foliations with mild singularities. 
In particular we prove Theorem~\ref{thmb}.
In Section~\ref{section:examples_pn_bundles} we address del Pezzo foliations on 
projective space bundles.

We plan to address Fano foliations on Fano manifolds with Picard 
number $1$ and related questions in forthcoming works.

\

\noindent {\bf Notation and conventions.}
We always work over the field ${\mathbb C}$ of complex numbers. 
Varieties are always assumed to be irreducible.
We denote by $\textup{Sing}(X)$ the singular locus of a variety $X$.
Given a sheaf $\sF$ of $\sO_X$-modules on a variety $X$, we denote by $\sF^{*}$ the sheaf $\sH\hspace{-0.1cm}\textit{om}_{\sO_X}(\sF,\sO_X)$.
If $r$ is the generic rank of $\sF$, then we denote by $\det (\sF)$ the sheaf $(\wedge^r \sF)^{**}$.
If $\sG$ is another sheaf of $\sO_X$-modules on $X$, then we denote by  $\sF[\otimes]\sG$ the sheaf $(\sF\otimes\sG)^{**}$.
If $\sE$ is a locally free sheaf of $\sO_X$-modules on a variety $X$, we denote by $\p_X(\sE)$ the Grothendieck projectivization $\textup{Proj}_X(\textup{Sym}(\sE))$.
If $X$ is a normal variety and $X\to Y$ is any morphism, we denote by
$T_{X/Y}$ the sheaf $(\Omega_{X/Y}^1)^*$. In particular, $T_X=(\Omega_{X}^1)^*$.
If $X$ is a smooth variety and $D$ is a reduced divisor on $X$ with simple normal crossings support, 
we denote by $\Omega_X^1(\textup{log }D)$ the sheaf of differential $1$-forms with logarithmic poles
along $D$, and by $T_X(-\textup{log }D)$ its dual sheal $\Omega_X^1(\textup{log }D)^*$. Notice that
$\det(\Omega_X^1(\textup{log }D))\simeq\sO_X(K_X+D)$.

\

\noindent {\bf Acknowledgements.}
Much of this work was developed during the authors' visits to IMPA and Institut Fourier.
We would like to thank both institutions for their support and hospitality. 
We also thank our colleagues Julie D\'eserti and Jorge Vit\'orio Pereira for very helpful discussions.


\section{Foliations and Pfaff fields} \label{section:foliations}

\begin{defn}
Let $X$ be normal variety.
A \emph{foliation} on $X$ is a nonzero coherent subsheaf $\sF\subsetneq T_X$ satisfying
\begin{enumerate}
	\item $\sF$ is closed under the Lie bracket, and
	\item $\sF$ is saturated in $T_X$ (i.e., $T_X / \sF$ is torsion free).
\end{enumerate}

The \textit{rank} $r_\sF$ of $\sF$ is the generic rank of $\sF$.

The \textit{canonical class} $K_{\sF}$ of $\sF$ is any Weil divisor on $X$ such that 
$\sO_X(-K_{\sF})\simeq \det(\sF)$. 

A \textit{foliated variety} is a pair $(X,\sF)$ consisting of  a normal variety $X$ 
together with a foliation  $\sF$ on $X$.
\end{defn}

\begin{defn}\label{def:gorenstein}
A foliation $\sF$ on a normal variety is said to be \emph{1-Gorenstein} if 
its canonical class $K_{\sF}$ is a Cartier divisor.
\end{defn}

\begin{rem}\label{lemma:saturated_reflexive}
Condition (2) above implies that $\sF$ is reflexive. Indeed, 
$T_X$ is reflexive by \cite[Corollary 1.2]{hartshorne80}. 
Thus, the inclusion $\sF\subset T_X$ factors through $\sF \subset \sF^{**}$. 
The induced map $\sF^{**}\to T_X / \sF$ is generically zero.
Hence it is identically zero since
$T_X / \sF$ is torsion free by (2). Thus
$\sF=\sF^{**}$.
\end{rem}

\begin{defn}
Let $X$ be a variety, and $r$ a  positive integer.
A  \emph{Pfaff field of rank r} on $X$ is a nonzero map
$\eta : \Omega^r_X\to \sL$, where  $\sL$ is an
invertible sheaf on $X$ (see \cite{esteves_kleiman03}). 
The \textit{singular locus} $S$ of $\eta$
is the closed subscheme of $X$ whose ideal sheaf $\sI_S$ is the image of
the induced map $\Omega^r_X\otimes \sL^*\to \sO_X$.

A closed subscheme $Y$ of $X$ is said to be \textit{invariant} under $\eta$ if 
\begin{enumerate}
	\item no irreducible component of $Y$ is contained in the singular locus of $\eta$, and
	\item the restriction $\eta|_{Y} : {\Omega^r_X}|_{Y}\to\sL|_{Y}$ factors through the natural map
	${\Omega^r_X}|_{Y}\to\Omega^r_Y$, in other words, there is a commutative diagram 

\centerline{
\xymatrix{
\Omega^r_X|_{Y} \ar[r]^{\eta|_{Y}}\ar[d] & \sL|_{Y} \ ,\\
\Omega^r_Y\ar[ru]
}
}
\noindent where the vertical map is the natural one.
\end{enumerate}
\end{defn}

Notice that a 1-Gorenstein foliation $\sF$ of rank $r$  on normal variety $X$ naturaly gives rise to a Pfaff field of rank $r$ on $X$:
$$
\eta:\Omega_X^r=\wedge^r(\Omega_X^1) \to \wedge^r(T_X^*) \to \wedge^r(\sF^*) \to \det(\sF^*)\simeq\det(\sF)^{*}=\sO_X(K_\sF).
$$

\begin{defn}\label{defn:sing_locus} \label{def:regular}
Let  $\sF$ be a 1-Gorenstein foliation on a normal variety $X$.
The \textit{singular locus} of  $\sF$ is defined to be
the singular locus $S$ of the associated Pfaff field. 
We say that $\sF$ or $(X,\sF)$  is \textit{regular at a point} $x\in X$ if $x\not\in S$. 
We say that $\sF$ or $(X,\sF)$  is \textit{regular} if $S=\emptyset$.
\end{defn}

Using Frobenius' theorem, one can prove the following. 

\begin{lemma}[{\cite[Lemma 1.3.2]{bogomolov_mcquillan01}}]\label{lemma:regular_smooth}
Let $(X,\sF)$ be a 1-Gorenstein foliated variety.
Suppose that $\sF$ regular and locally free at a point $x\in X$. 
Then there exists an analytic open neighborhood $U$ of $x$, a  complex analytic space $W$,
and a smooth morphism 
$U \to W$ of relative dimension $r_\sF$ 
such that $\sF_U=T_{U/W}$.
\end{lemma}

\begin{lemma}\label{lemma:singularlocusfoliationveruspfafffield}
Let $X$ be a smooth variety, and $\sF$ a foliation 
of rank $r$
on $X$ with  singular locus $S$.
Let $S_1$ be the set of points $x\in X$ at which 
$\sF$ is not locally free, and $S_2$  the set of points $x\in X$ such that 
$\sF$ is locally free at $x$ and $\sF\otimes k(x) \to T_X\otimes k(x)$ is not injective.
\begin{enumerate}
\item Then $S \subset S_1\cup S_2$ as sets, and $S\setminus S_1=S_2$. 
\item Let $Y\subset X$ be an irreducible subvariety of dimension $r_\sF$
such that $Y \not\subset S_1\cup S_2$. 
Then $Y\setminus S_1\cup S_2$ is a leaf of
$\sF|_{X\setminus S_1\cup S_2}$ if and only if $Y$ is invariant under the associated Pfaff field 
$\eta:\Omega^{r}_X \to \sO_X(K_\sF)$.
\end{enumerate}
\end{lemma}

\begin{proof}
Let $x\in X$ be a point at which  $\sF$ is locally free. 
Then there is an  open neighborhood of $x$ where 
$\det(\sF^*)$ is invertible.
Thus $x\in S$ if and only if $x\in S_2$, proving (1).

Let 
$x\in Y\setminus S_1\cup S_2$ be a smooth point of $Y$ and let $\vec v_1,\ldots,\vec v_r$ be local vector fields
that generate $\sF$ on an affine neighbourhood $U$ of $x$. Observe that 
$\eta|_{U}:{\Omega^{r}_X}|_{U} \to {\sO_X(K_\sF)}|_{U}$ is given by
\begin{eqnarray*}
H^0(U,\Omega^{r}_X|_{U}) & \longrightarrow & H^0(U,\sO_X(K_\sF)|_{U}) \\
 {\alpha}  & \longmapsto & \alpha(\vec v_1,\ldots,\vec v_r)\omega
\end{eqnarray*}
where $\omega\in H^0(U,\sO_X(K_\sF)|_{U})$ is such that 
$\omega(\vec v_1,\ldots,\vec v_r)=1$. It follows that $Y$ is invariant under $\eta$
if and only if, for any local function $f$ on $U$ vanishing
along $Y\cap U$, and any local $(r-1)$-differential form $\beta$ on $U$, we have
$(df\wedge \beta)(\vec v_1,\ldots,\vec v_r)=0$.
This happens if and only if,
for any $i\in \{1,\ldots,r\}$ and any local function $f$ on $U$ vanishing
along $Y\cap U$, we have $df(\vec v_i)=0$ .
This is in turn equivalent to requiring that 
$\vec v_i(x)\in T_{Y,x}$ for any $i\in \{1,\ldots,r\}$, which is saying precisely that  $Y\setminus S_1\cup S_2$ 
is a leaf of $\sF|_{X\setminus S_1\cup S_2}$. This proves $(2)$. 
\end{proof}

Next we define Fano foliations and Fano Pfaff fields.

\begin{defn}\label{def:fano_foliation}\label{def:index_fano_foliation}\label{def:fano_pfaff}\label{def:index_fano_pfaff}
Let $X$ be a normal projective variety.

Let $\sF$ be a 1-Gorenstein foliation on $X$.
We say that $\sF$ is a \emph{Fano foliation} if $-K_{\sF}$ is ample.
In this case, the index $\iota_{\sF}$ of $\sF$ is the largest positive integer such that
$-K_{\sF} \sim \iota_{\sF} H$ for a Cartier divisor $H$ on $X$.

Let $\sL$ be a line bundle on $X$, $r$ a positive integer, and
$\eta : \Omega^r_X\to \sL$ a Pfaff field.
We say that $\eta$ is a \emph{Fano Pfaff field} if $\sL^{-1}$ is ample.
In this case, the index $\iota_{\eta}$ of $\eta$ is the largest positive integer such that
$\sL^{-1} \sim \sA^{\otimes \iota_\eta} $ for a line bundle $\sA$ on $X$.
\end{defn}

\begin{rem}\label{uniruledness}
Let $X$ be a smooth complex projective variety. 
If $X$ admits a Fano foliation or a Fano Pfaff field, then $X$ is uniruled by 
\cite[Corollary 8.6]{miyaoka87}.
\end{rem}

In analogy with Kobayachi-Ochiai's theorem, we have the following.

\begin{thm}[{\cite[Theorem 1.1]{adk08}}] 
Let $X$ be a smooth complex projective variety,
$\sL$ a line bundle on $X$, $r$ a positive integer, and
$\eta : \Omega^r_X\to \sL$ a Fano Pfaff field. Then:
\begin{enumerate}
\item $\iota_{\eta} \le r+1$;
\item $\iota_{\eta} = r+1$ if and only if $r=\dim(X)$ and $(X,\sL)\simeq
(\p^{r},\sO_{\p^{r}}(-1))$;
\item $\iota_{\eta} = r$ if and only if 
either $(X,\sL)\simeq (\p^n,\sO_{\p^{n}}(-1))$ for some $n\geq r$,
or $r=\dim(X)$ and $(X,\sL)\simeq (Q_{r},\sO_{Q_{r}}(-1))$, where $Q_{r}$ denotes a smooth quadric 
hypersurface in $\p^{r+1}$ and $\sO_{Q_{r}}(-1)$ denotes the restriction of $\sO_{\p^{r+1}}(-1)$
to $Q_{r}$.
\end{enumerate}
\end{thm}

\begin{defn}
Let $X$ be a smooth projective variety, and $\sF$ a Fano
foliation on $X$ of rank $r_\sF$ and index $\iota_{\sF}$.
We say that $\sF$ is a \textit{del Pezzo foliation} if $r_\sF\ge 2$ and $\iota_{\sF} = r_\sF-1$.
\end{defn}


\section{Algebraically integrable foliations}\label{section:algebraic_foliations}

\begin{defn}
Let $X$ be normal variety. A foliation $\sF$ on $X$ is said to be
\emph{algebraically integrable} if  the leaf of $\sF$ through a general point of $X$ is an algebraic variety. 
In this situation, by abuse of 
notation we often use the word ``leaf" to mean the closure in $X$ of a leaf of $\sF$. 
\end{defn}

\begin{lemma}\label{lemma:leaffoliation}
Let $X$ be normal projective variety, and $\sF$ an algebraically integrable foliation on $X$.
There is a unique irreducible closed subvariety $W$ of $\Chow(X)$ 
whose general point parametrizes the closure of a general leaf of $\sF$
(viewed as a reduced and irreducible cycle in $X$). In other words, if 
$U \subset W\times X$ is the universal cycle, with universal morphisms
$\pi:U\to W$ and $e:U\to X$,
then $e$ is birational, and, for a general point $w\in W$, 
$e\big(\pi^{-1}(w)\big) \subset X$ is the closure of a leaf of $\sF$.
\end{lemma}

\begin{notation}
We say that the subvariety $W$  provided by Lemma~\ref{lemma:leaffoliation} is 
\emph{the closure in $\Chow(X)$ of the subvariety parametrizing general leaves of $\sF$}.
\end{notation}

\begin{proof}[Proof of Lemma~\ref{lemma:leaffoliation}]
First of all, recall that $\Chow(X)$ has countably many irreducible components.
On the other hand, since we are working over $\mathbb C$, $\sF$ has uncountably many leaves.
Therefore, there is a closed subvariety $W$ of $\Chow(X)$ such that
\begin{enumerate}
	\item the universal cycle over $W$ dominates $X$, and
	\item the subset of points in $W$  parametrizing leaves of $\sF$ (viewed as reduced and irreducible cycles in $X$) is Zariski dense in $W$.
\end{enumerate}
Let $U \subset W\times X$ be the universal cycle over $W$, 
denote by $p:W\times X \to W$ and $q:W\times X \to X$ the natural projections, and by 
$\pi=p|_U:U\to W$ and $e=q|_U:U\to X$ their restrictions to $U$.
We need to show that, for a general point $w\in W$, 
$e\big(\pi^{-1}(w)\big) \subset X$ is the closure of a leaf of $\sF$.

To simplify notation, we suppose that $X$ is smooth. In the general case, in what follows one should replace 
$X$ with its smooth locus $X_0$, $W$ with a dense open subset 
$W_0\subset q(p^{-1}(X_0))$ and
$U$ with $U_0=q^{-1}(X_0)\cap p^{-1}(W_0)\cap U$.

Let $\eta_X:\Omega^{r}_X \to \sO_X(K_\sF)$ be the Pfaff field associated to $\sF$.
It induces a Pfaff field of rank $r$ on $W\times X$:
$$
\eta_{W\times X}:\Omega_{W\times X}^{r}
=\wedge^{r}(p^{*}\Omega_{W}^{1}\oplus q^{*}\Omega_{X}^{1})
\to \wedge^{r}(q^{*}\Omega_{X}^{1})\simeq q^{*}\Omega_{X}^{r}
\to q^*\sO_X(K_\sF).
$$ 
We claim that $U$ is invariant under $\eta_{W\times X}$. 
Indeed, let $K$ be the kernel of the natural morphism 
${\Omega_{W\times X}^{r}}|_{U}\twoheadrightarrow \Omega_U^{r}$.
The composite map 
$K\to {\Omega_{W\times X}^{r}}|_{U} \to e^*\sO_X(K_\sF)$
vanishes on a Zariski dense subset of $U$ by Lemma \ref{lemma:singularlocusfoliationveruspfafffield}.
Since $e^*\sO_X(K_\sF)$ is torsion-free, it vanishes identically, and thus the restriction 
$\eta_{W\times X}|_U:{\Omega_{W\times X}^{r}}|_{U} \to e^*\sO_X(K_\sF)$ factors through 
${\Omega_{W\times X}^{r}}|_{U}\twoheadrightarrow \Omega_U^{r}$.
Similarly, the morphism $\eta_U: \Omega_U^{r} \to e^*\sO_X(K_\sF)$ factors through 
the natural morphism $\Omega_U^{r} \twoheadrightarrow \Omega_{U/W}^{r}$.
Lemma \ref{lemma:singularlocusfoliationveruspfafffield} then implies that, for a general point $w\in W$,
$e\big(\pi^{-1}(w)\big) \subset X$ is the closure of a leaf of $\sF$.
\end{proof}

Next we come to the definition of a general log leaf of an  algebraically integrable foliation.

\begin{defn}\label{defn:log_leaf}
Let $X$ be normal projective variety, $\sF$ a $1$-Gorenstein algebraically integrable foliation of rank $r$ on $X$, and 
$\eta_\sF:\Omega_X^{r} \to \sO_X(K_\sF)$ the corresponding Pfaff field.
Let $F$ be the closure of a general leaf of $\sF$, and $n:\tilde F\to F\subset X$ its normalization.
By Lemma \ref{lemma:singularlocusfoliationveruspfafffield}, $F$ is invariant under $\eta_\sF$, i.e., the restriction
$\eta_\sF|_{F} : {\Omega^r_X}|_{F}\to\sO_X(K_\sF)|_{F}$ factors through the natural map
${\Omega^r_X}|_{F}\to\Omega^r_F$.
By Lemma~\ref{lemma:extensionpfafffields} below, the induced map
$\eta: \Omega^r_F\to\sO_X(K_\sF)|_{F}$ extends uniquely to a generically surjective map
$\tilde \eta: \Omega^r_{\tilde F}\to n^*\sO_X(K_\sF)$.
Hence there is a canonically defined effective Weil divisor $\tilde \Delta$ on $\tilde F$ such that 
$\sO_{\tilde F}(K_{\tilde F}+\tilde \Delta)\simeq n^*\sO_X(K_\sF)$. Namely, $\tilde \Delta$ is the divisor of zeroes of $\tilde\eta$.

We call the pair $(\tilde F,\tilde \Delta)$ a \emph{general log leaf} of $\sF$.
\end{defn}

\begin{lemma}[{\cite[Proposition 4.5]{adk08}}]\label{lemma:extensionpfafffields}
Let $X$ be a variety and $n:\widetilde X\to X$ its normalization.  Let
$\sL$ be a line bundle on $X$, $r$ a positive integer, and
$\eta : \Omega^r_X\to \sL$ a Pfaff field.
Then $\eta$ extends uniquely  to a Pfaff field  
$\tilde\eta : \Omega^r_{\tilde X}\to n^*\sL$ of rank $r$.
\end{lemma}

Next we define notions of singularity for $1$-Gorenstein algebraically integrable foliations
according to the singularity type of their general log leaf. First we
recall some definitions of singularities of pairs, developed in the context of the minimal model program.
We refer to  \cite[section 2.3]{kollar_mori} for details.

\begin{say}[Singularities of pairs.]
Let $X$ be a normal projective variety, and
$\Delta=\sum a_i\Delta_i$ an effective $\bQ$-divisor on $X$, i.e., $\Delta$ is  a nonnegative $\bQ$-linear combination 
of distinct prime Weil divisors $\Delta_i$'s on $X$. 
Suppose that $K_X+\Delta$ is $\bQ$-Cartier, i.e.,  some nonzero multiple of it is a Cartier divisor on $X$. 

Let $f:\tilde X\to X$ be a log resolution of the pair $(X,\Delta)$. 
This means that $\tilde X$ is a smooth projective
variety, $f$ is a birational projective morphism whose exceptional locus is the union of prime divisors $E_i$'s, 
and the divisor $\sum E_i+f^{-1}_*\Delta$ has simple normal crossing 
support.  
There are uniquely defined rational numbers $a(E_i,X,\Delta)$'s such that
$$
K_{\tilde X}+f^{-1}_*\Delta = f^*(K_X+\Delta)+\sum_{E_i}a(E_i,X,\Delta)E_i.
$$
The $a(E_i,X,\Delta)$'s do not depend on the log resolution $f$,
but only on the valuations associated to the $E_i$'s. 

We say that $(X,\Delta)$ is \emph{log terminal} (or \emph{klt})  if  all $a_i<1$, and, for some  log resolution 
$f:\tilde X\to X$ of $(X,\Delta)$, $a(E_i,X,\Delta)>-1$ 
for every $f$-exceptional prime divisor $E_i$.
We say that  $(X,\Delta)$ is \emph{log canonical} if  all $a_i\leq 1$, and, for some  log resolution 
$f:\tilde X\to X$ of $(X,\Delta)$, $a(E_i,X,\Delta)\geq -1$ 
for every $f$-exceptional prime divisor $E_i$.
If these conditions hold for some log resolution of $(X,\Delta)$, then they hold for every  
log resolution of $(X,\Delta)$.
\end{say}

\begin{defn}\label{def:sing_fol}
Let $X$ be normal projective variety, $\sF$ a $1$-Gorenstein algebraically integrable foliation  on $X$,
and $(\tilde F,\tilde \Delta)$ its general log leaf. 
We say that $\sF$ has \emph{log terminal (respectively log canonical) singularities along a general leaf} if
$(\tilde F,\tilde \Delta)$ is log terminal (respectively log canonical).
In particular, if $\sF$ has log terminal singularities along a general leaf, then $\tilde \Delta=0$.
\end{defn}

\begin{rem}\label{remark:definition_log_leaf}
Let $X$ be normal projective variety,  and $\sF$ a $1$-Gorenstein algebraically integrable foliation of rank $r$ on $X$.
Let $W$ be the closure in $\Chow(X)$ of the subvariety parametrizing general leaves of $\sF$, and
$U \subset W\times X$ the universal cycle. Denote by $e:U\to X$ the natural morphism.
We saw in the proof of Lemma~\ref{lemma:leaffoliation} that $\sF$ induces a Pfaff field 
$\eta_U: \Omega_U^{r} \to e^*\sO_X(K_\sF)$, which factors through 
the natural morphism $\Omega_U^{r} 
\twoheadrightarrow \Omega_{U/W}^{r}$.

Let $\tilde W$ and $\tilde U$ be the normalizations of $W$ and $U$, respectively. Denote by 
$\tilde \pi:\tilde U\to \tilde W$ and $\tilde e:\tilde U\to X$ the induced morphisms. By Lemma
\ref{lemma:extensionpfafffields},
$\eta_U: \Omega_U^{r} \to e^*\sO_X(K_\sF)$
extends uniquely to a Pfaff field
$\eta_{\tilde U}:\Omega_{\tilde U}^{r} \to {\tilde e}^*\sO_X(K_\sF)$.
As before, this morphism factors through the natural morphism
$\Omega_{\tilde U}^{r} \twoheadrightarrow \Omega_{\tilde U/\tilde W}^{r}$,
yielding a generically surjective map 
$$
\Omega_{\tilde U/\tilde W}^{r}\to {\tilde e}^*\sO_X(K_\sF).
$$
Thus there is a canonically defined effective Weil divisor $\Delta$ on $\tilde U$ such that 
$\det(\Omega_{\tilde U/\tilde W}^1)[\otimes] \sO_{\tilde U}(\Delta) \simeq  {\tilde e}^*\sO_X(K_\sF)$.

Let $w$ be a general point of $\tilde W$, set $\tilde U_w := \tilde \pi^{-1}(w)$ and $\Delta_w:=\Delta|_{\tilde U_w}$.
Then $(\tilde U_w, \Delta_w)$ coincides with the general log leaf $(\tilde F,\tilde \Delta)$ defined above.
In particular, by \cite[Corollary 1.4.5]{BCHM}, $\sF$ has log terminal (respectively log canonical) 
singularities along a general leaf if and only if $(\tilde U,\Delta)$ has log terminal (respectively log canonical) 
singularities over the generic point of $\tilde W$.

The same construction can be carried out by replacing $W$ with a general closed subvariety of it.
\end{rem}

Next we compare the notions of singularities for algebraically integrable foliations introduced in 
Definition~\ref{def:sing_fol} with those introduced earlier
in \cite{mcquillan08}.
We recall McQuillan's definitions, which do not require algebraic integrability.

\begin{say}[{\cite[Definition I.1.2]{mcquillan08}}]
Let $(X,\sF)$ be a  foliated variety.
Given a birational morphism $\varphi:\tilde X \to X$,
there is a unique foliation $\tilde\sF$ on $\tilde X$ that agrees with $\varphi^*\sF$ on the open 
subset of $\tilde X$ where $\varphi$ is an isomorphism. 
We say that $\varphi:(\tilde X, \tilde \sF) \to (X,\sF)$ is a birational morphism of foliated varieties.

From now on assume moreover that $K_\sF$ is $\bQ$-Cartier and $\varphi$ is projective.
Then there are uniquely defined rational numbers $a(E,X,\sF)$'s such that
$$
K_{\tilde{\sF}}=\varphi^*K_{\sF}+ \sum_E a(E,X,\sF)E,
$$
where $E$ runs through all exceptional prime divisors for $\varphi$.
The $a(E,X,\Delta)$'s do not depend on the birational morphism $\varphi$,
but only on the valuations associated to the $E$'s. 

For an exceptional prime divisor $E$ over $X$, define 
$$\epsilon(E):= \left\{
\begin{array}{l}
0 \text{ \ if } E \text{ is invariant by the foliation, } \\
1 \text{ \ if } E \text{ is not invariant by the foliation.}
\end{array}
\right. $$

The foliated variety
$(X,\sF)$ is said to be
\begin{equation*}
\begin{cases}
\text{terminal}\\
\text{canonical}\\
\text{log terminal}\\
\text{log canonical}
\end{cases}
\text{in the sense of McQuillan if, for all $E$ exceptional over $X$, } a(E,X,\sF)
\begin{cases}
> 0,\\
\ge 0,\\
>-\epsilon(E),\\
\geq-\epsilon(E).
\end{cases}
\end{equation*}
\end{say}

\begin{lemma} Let $(X,\sF)$ be a $1$-Gorenstein foliated variety. If $\sF$ is regular, then 
$(X,\sF)$ is canonical in the sense of McQuillan.
\end{lemma}

\begin{proof}
Let $\varphi: (\tilde{X}, \tilde{\sF}) \to (X,\sF)$ be a birational projective morphism of foliated varieties with 
$\tilde X$ smooth. 
Let $\eta_\sF:\Omega^{r_\sF}_X \to \sO_X(K_\sF)$ and 
$\eta_{\tilde\sF}:\Omega^{r_\sF}_{\tilde X} \to \sO_{\tilde X}(K_{\tilde \sF})$ be the
associated Pfaff fields.
Since $\sF$ is regular, $\varphi^*\eta_\sF:\varphi^*\Omega_X^{r_\sF}\to \varphi^*\sO_X(K_\sF)$
is a surjective morphism.

We claim that the the composite map
$\varphi^*\Omega_X^{r_\sF} \to \Omega^{r_\sF}_{\tilde X} \to \sO_{\tilde X}(K_{\tilde \sF})$
factors through $\varphi^*\eta_\sF:\varphi^*\Omega_X^{r_\sF} \to \varphi^*\sO_X(K_\sF)$.
Indeed, denote by $K$ be the kernel of $\varphi^*\eta_\sF:\varphi^*\Omega_X^{r_\sF}\to \varphi^*\sO_X(K_\sF)$.
The composite map
$K \to
\varphi^*\Omega_X^{r_\sF} \to \Omega^{r_\sF}_{\tilde X} \to \sO_{\tilde X}(K_{\tilde \sF})$
vanishes over a dense subset of $\tilde X$.  
Since $\sO_{\tilde X}(K_{\tilde \sF})$ is torsion-free, it vanishes identically on $\tilde X$.
This proves the claim.
So we obtain a nonzero map $\varphi^*\sO_X(K_\sF) \to \sO_{\tilde X}(K_{\tilde \sF})$.
Thus there is an effective divisor $E$ on $\tilde X$ such that $K_{\tilde \sF}=\varphi^*K_{\sF}+E$.
\end{proof}

\begin{prop}
Let $X$ be a normal projective variety,  and $\sF$ a $1$-Gorenstein algebraically integrable foliation on $X$.
Let $W$ be the closure in $\Chow(X)$ of the subvariety parametrizing general leaves of $\sF$. 
If $(X,\sF)$ is log terminal (respectively log canonical) in the sense of McQuillan, then $\sF$
has log terminal (respectively log canonical) singularities along a general leaf. 
\end{prop}

\begin{proof}
We follow the notation in Remark~\ref{remark:definition_log_leaf}.

Let $w\in \tilde W$ be a general point and let $(\tilde U_w,\Delta_w)$ be the corresponding log leaf.
We denote by $\tilde e_w:\tilde U_w\to X$ the natural morphism. 
Recall that 
\begin{equation}\label{eq:can2}
K_{\tilde U_w}+\Delta_w= \tilde e_w^*({K_{\sF}}).
\end{equation}
Suppose that $(X,\sF)$ is log terminal (respectively log canonical) in the sense of McQuillan.
We have to show that the pair $(\tilde U_w,\Delta_w)$ is log terminal (respectively log canonical).

Let $d:Y \to \tilde U$ be a log resolution of singularities, and consider the commutative diagram

\centerline{
\xymatrix{
Y\ar[r]^{d}\ar@/^1.5pc/[rr]^{g} 
& \tilde U \ar[r]^{\tilde e}\ar[d]_{\tilde \pi} & X. \\
& \tilde W &
}
}
\noindent 
Denote by $\sF_Y$  the foliation  induced by $\sF$ on $Y$, and notice that $\sF_Y=T_{Y/\tilde W}$. 
Write 
$$
K_{\sF_Y}=g^*K_{\sF}+ \sum a(E,X,\sF)E,
$$
where $E$ runs through all exceptional prime divisors for $g$. Note that the support of
the divisor  $\Delta$ on $\tilde U$ defined in Remark~\ref{remark:definition_log_leaf} is exceptional over $X$, so 
the strict transforms  of its components in $Y$ appear among the $E$'s.

Set $Y_w:=d^{-1}(\tilde U_w)$,  $d_w:=d|_{Y_w}:Y_w\to\tilde U_w$, and $E_w:=E|_{\tilde U_w}$. 
Since $w\in \tilde W$ is general, ${K_{\sF_Y}}|_{Y_w}=K_{Y_w}$. 
Thus
\begin{equation}\label{eq:can1}
K_{Y_w}=d_w^*\tilde e_w^*({K_{\sF}})+\sum a(E,X,\sF) E_w.
\end{equation}
Notice that $d_w:Y_w \to \tilde U_w$ is a log resolution of singularities.
From (\ref{eq:can1}) and (\ref{eq:can2}) we deduce that
\begin{equation*}\label{eq:can3}
K_{Y_w}=d^*(K_{\tilde U_w}+\Delta_w)+\sum a(E,X,\sF) E_w.
\end{equation*}
This proves the result.
\end{proof}

\begin{rem}
The notions of singularities of foliations discussed above do not say anything about the singularities 
of the ambient space.
For instance, let $Y$ be a smooth variety, $T$ any normal variety, and set $X:=Y\times T$, 
with natural projection $p:X\to Y$.
Set $\sF:= p^*T_Y\subset T_X$.
Then $(X,\sF)$ is a regular 1-Gorenstein foliated variety,  canonical in the sense of McQuillan,
while $X$ may be very singular.
\end{rem}


\section{Examples}
\label{section:examples}

\begin{say}[Foliations of rank $r$ and index $r$ on $\p^n$]\label{example:ample_on_p^n}
Let $\sF\subsetneq T_{\p^n}$ be a Fano foliation of rank $r$ and index $\iota_{\sF}=r$ on $\p^n$.
These are classically known as \emph{degree 0 foliations on } $\p^n$.
By \cite[Th\'eor\`eme 3.8]{cerveau_deserti}, $\sF$ is defined by a linear projection 
$\p^n \dashrightarrow \p^{n-r}$. 
The singular locus of $\sF$ is a linear subspace $S$ of dimension $r-1$.
The closure
of the leaf through a point $p\not\in S$ is the $r$-dimensional linear subspace $L$ of $\p^n$ containing both $p$ and  $S$. 
Let $p_1,\ldots,p_r \in S$ be $r$ linearly
independent points in $S$, and $v_i\in H^0(\p^n,T_{\p^n}(-1))$ a nonzero section 
vanishing at $p_i$. 
Then the $v_i$'s define an injective map $\sO_{\p^n}(1)^{\oplus r}\to T_{\p^n}$
whose image is $\sF$. Thus the restricted map $\sF|_{L}\to T_L$ is induced
by the sections ${v_i}|_{L}\in H^0(L,T_L(-1))\subset H^0(L,{T_{\p^n}(-1)}|_{L})$. In particular,
the zero locus of the map $\det({\sF})|_{L}\to \det(T_L)$ is the codimension one linear 
subspace $S\cap L\subset L$. Thus the log leaf  $(\tilde F, \tilde \Delta)=(L,S\cap L)$ is log canonical, and 
$\sF$ has log canonical singularities along a general leaf.
\end{say}

\begin{say}[Foliations of rank $r$ and index $r-1$ on $\p^n$]\label{example:del_Pezzo_on_p^n}
Let $\sF\subsetneq T_{\p^n}$ be a Fano foliation of rank $r$ and index $\iota_{\sF}=r-1$ on $\p^n$.
By \cite[Theorem 6.2]{loray_pereira_touzet_2},
\begin{itemize}
\item either $\sF$ is defined by a rational dominant map $\p^n\dashrightarrow \p(1^{n-r},2)$, 
defined by $n-r$ linear forms and one quadric form,
where $\p(1^{n-r},2)$ denotes the weighted projective space of type 
$(\underbrace{1,\ldots,1}_{r \text{ times}},2)$.
\item or $\sF$ is the linear pullback of a foliation on $\p^{n-r+1}$ induced by a global holomorphic vector field.
\end{itemize}

Note that a foliation on $\p^{n-r+1}$ induced by a global holomorphic vector field may or may not have algebraic leaves.
Moreover, algebraically integrable foliations of rank $r$ and index $r-1$ on $\p^n$
may or may not have log canonical singularities along a general leaf.
\end{say}

\begin{say}[Fano foliations on Grassmannians] \label{example:grass}
Let $m$ and $n$ be nonnegative integers, and $V$ a complex vector space of dimension $n+1$.
Let $G=G(m+1, V)$ be the Grassmannian of ($m+1$)-dimensional linear subspaces of $V$, with tautological exact sequence
$$0 \to \sK \to V\otimes \sO_G \to \sQ \to 0.$$
Let $k$ be an integer such that $0\le k \le n-m-1$, and 
$W$ a ($k+1$)-dimensional linear subspace of $V$.
Set 
$$\sF:=W\otimes \sK^*\subset V\otimes \sK^*.$$
The map $V\otimes \sK^*\to \sQ\otimes \sK^*$
induced by $V\otimes \sO_G \to \sQ$ yields a map
$\sF \to \sQ\otimes \sK^*\simeq T_G$. 
For a general point $[L]\in G$, $L\cap W=\{0\}$ since $k+m \le n-1$.
Thus the map $\sF\to T_G$ is injective at $[L]$. 
Since $\sF$ is locally free, $\sF\into T_G$ is injective.
Let $P$ be the linear span of $L$ and $W$ in $V$. It has dimension $m+k+2 \le n+1$.
Notice that the Grassmannian 
$G(m+1,P)\subset G$ is tangent to $\sF$ at 
a general point of $G(m+1,P)$. 

Suppose that $k\le n-m-2$ (or equivalently that $\dim(P)<\dim(V)$). 
Then $\sF$ is a subbundle of $T_G$ in codimension one, and thus saturated in $T_G$ by lemma \ref{lemma:saturation}. 
In particular $\sF$ is a Fano foliation on $G$ of rank $r=(m+1)(k+1)$.
Its singular locus $S$ is the set of points $[L]\in G$ such that
$\dim(L\cap W)\ge 1$.

Recall that $\textup{Pic}(G)=\bZ[\sO_G(1)]$ where $\sO_G(1)\simeq \det(\sQ)$ is the pullback of
$\sO_{\p(\wedge^{m+1}V)}(1)$ under the Pl\"ucker embedding. It follows that $\sF$ has index 
$\iota_{\sF}=k+1$. 
In particular, $\iota_{\sF}=r-1$
if and only if $m=1$ and $k=0$.
In this case,
$G=G(2,V)$ and $\sF$ is the rank $2$ foliation on $G$ whose 
general leaf 
is the $\p^2$ of  $2$-dimensional linear subspaces of  a general $3$-plane containing the line $W$.

Finally, observe that $S \cap G(m+1,P)$ is irreducible and has codimension one in $G(m+1,P)$.
Moreover, $\det(T_{G(m+1,P)})\simeq \sO_{G(m+1,P)}(m+k+2)$, and 
$\det(\sF)|_{G(m+1,P)}\simeq \sO_{G(m+1,P)}(k+1)$. It follows that the map
$\det(\sF)|_{G(m+1,P)}\to \det(T_{G(m+1,P)})$ vanishes at order $m+1$ along
$S \cap G(m+1,P)$. 
So the general log leaf of $\sF$ is  
$$
(\tilde F, \tilde \Delta)=\Big(G(m+1,P), (m+1)\cdot \big(S \cap G(m+1,P)\big)\Big).
$$
In particular, $\sF$ has log canonical singularities along a general leaf if and only if $m=0$, i.e., $G=\p^n$, and 
$\sF$ is the foliation described in \ref{example:ample_on_p^n} above.
In all other cases,  the closures of the leaves of $\sF$ do not have a 
common point in $G$.

When $m=1$ and $k=0$, we obtain a rank $2$ del Pezzo foliation on $G=G(2,V)$ with general log leaf $(\tilde F, \tilde \Delta)\simeq (\p^2, 2H)$,
where $H$ is a line in $\p^2$. 
\end{say}

Next we want to discuss Fano foliations on hypersurfaces of projective spaces.
In order to do so, it will be convenient to view foliations as given by differential forms.

\begin{say}[Foliations as $q$-forms] 
Let $X$ be a smooth variety of dimension $n\ge 2$, and 
$\sF\subsetneq T_X$ a foliation of rank $r$ on $X$. 
Set $N^*_\sF:=(T_X/\sF)^*$, and  $N_\sF:=(N^*_\sF)^*$.
These are called the \emph{conormal} and \emph{normal} sheaves of the foliation $\sF$, respectively.
The conormal sheaf $N^*_\sF$ is a saturated subsheaf of $\Omega^1_X$ of rank $q:=n-r$. 
The $q$-th wedge product of the inclusion
$N^*_\sF\subset \Omega^1_X$ gives rise to a nonzero twisted differential $q$-form $\omega$ with coefficients in the
line bundle $\sL:=\det(N_\sF)$, which is \emph{locally decomposable} and \emph{integrable}.
To say that $\omega\in H^0(X,\Omega^q_X\otimes \sL)$ is locally decomposable means that, 
in a neighborhood of a general point of $X$, $\omega$ decomposes as the wedge product of $q$ local $1$-forms 
$\omega=\omega_1\wedge\cdots\wedge\omega_q$.
To say that it is integrable means that for this local decomposition one has 
$d\omega_i\wedge \omega=0$ for $i\in\{1,\ldots,q\}$. 
Conversely, given a twisted $q$-form $\omega\in H^0(X,\Omega^q_X\otimes \sL)\setminus\{0\}$
which is locally decomposable and integrable, we define 
a foliation of rank $r$ on $X$ as the kernel
of the morphism $T_X \to \Omega^{q-1}_X\otimes \sL$ given by the contraction with $\omega$. 
\end{say}

\begin{lemma}\label{lemma:vanishing_hypersurface}
Fix $n \ge 3$, and let $X\subset\p^{n+1}$ be a smooth
hypersurface of degree $d \ge 3$. Let $k$ and $q$ be integers such that
$k \le q \le n-2$ and $q\ge 1$. Then $h^0(X,\Omega^q_X(k))=0$. 
\end{lemma}

Before we prove the lemma, we recall Bott's formulae.

\begin{say}[Bott's Formulae]
Let $n,p,q$ and $k$ be integers, with $n$  positive and $p$ and $q$ 
nonnegative.  Then
\begin{equation*}
h^q(\p^n,\Omega_{\p^n}^p(k)) =
\begin{cases}
\binom{k+n-p}{k}\binom{k-1}{p} & \text{for } q=0, 0\le p\le n \text{ and } k>p,\\
1 & \text{for } k=0 \text{ and } 0\le p=q\le n,\\
\binom{-k+p}{-k}\binom{-k-1}{n-p} & \text{for } q=n, 0\le p\le n \text{ and } k<p-n,\\
0 & \text{otherwise.}
\end{cases}
\end{equation*}
Let $r,s$ and $t$ be integers, with $r$ and $s$ 
nonnegative. 
Observe that the natural pairing 
$\Omega^p_{\p^n}\otimes\Omega^{n-p}_{\p^n} \to \Omega^n_{\p^n}$ is perfect. 
It induces an isomorphism 
$\wedge^{r}T_{\p^n}(t)\simeq \Omega_{\p^n}^{n-r}(t+n+1)$.
So 
the formulae above become
\begin{equation*}
h^s(\p^n,\wedge^{r}T_{\p^n}(t)) =
\begin{cases}
\binom{t+n+1+r}{t+n+1}\binom{t+n}{n-r} & \text{for } s=0, 0\le r \le n \text{ and } t+r\ge 0,\\
1 & \text{for } t=-n-1 \text{ and } 0\le n-r=s\le n,\\
\binom{-t-1+r}{-t-n-1}\binom{-t-n-2}{r} & \text{for } s=n, 0\le r\le n \text{ and } t+n+r+2\le 0,\\
0 & \text{otherwise.}
\end{cases}
\end{equation*}
\end{say}

\begin{proof}[Proof of Lemma~\ref{lemma:vanishing_hypersurface}]
By \cite[Satz 8.11]{flenner81},
\begin{enumerate}
\item $h^0(X,\Omega^r_X(s))=0$ for $s < r\le n-1$,
\item $h^1(X,\Omega^r_X(s))=0$ for $0\le r\le n-2$ and $s\le r-2$.
\end{enumerate}
Thus, it is enough to prove that $h^0(X,\Omega^q_X(q))=0$ for $1\le q \le n-2$. 
Let $q\in\{1,\ldots,n-2\}$. 
By Bott's formulae,
\begin{enumerate}
\item $h^0(\p^{n+1},\Omega^r_{\p^{n+1}}(r))=0$ for $r\ge 1$,
\item $h^1(\p^{n+1},\Omega^r_{\p^{n+1}}(s))=0$ for $s <r-1$.
\end{enumerate}
The cohomology of the exact sequence
of sheaves on $\p^{n+1}$
$$0\to \Omega^{q}_{\p^{n+1}}(q-d) \to \Omega^{q}_{\p^{n+1}}(q)
\to {\Omega^{q}_{\p^{n+1}}(q)}|_{X}
\to 0,$$
and the vanishing of $H^0(\p^{n+1},\Omega^{q}_{\p^{n+1}}(q))$
and $H^1(\p^{n+1},\Omega^{q}_{\p^{n+1}}(q-d))$ imply the vanishing
of $H^0(X,{\Omega^{q}_{\p^{n+1}}(q)}|_{X})$.

The cohomology of the exact sequence
of sheaves on $X$
$$0\to  \Omega^{q-1}_X(q-d) \to {\Omega^{q}_{\p^{n+1}}(q)}|_{X}
\to \Omega^{q}_X(q) \to 0,$$
and the vanishing of 
$H^0(X,{\Omega^{q}_{\p^{n+1}}(q)}|_{X})$
and $H^1(X,\Omega^{q-1}_X(q-d))$
yield the result.
\end{proof}

\begin{prop}[Fano foliations on hypersurfaces] \label{prop:foliation_hypersurface}
Fix $n \ge 3$, and let $X\subset\p^{n+1}$ be a smooth
hypersurface of degree $d \ge 3$. 
Let $r\in\{2, \cdots, n-1\}$, and $\iota$ be a positive integer.
Then there exists a Fano foliation of rank $r$ and index $\iota$ on $X$  if and only if $d+\iota \le r+1$.
\end{prop}

\begin{proof}Let $\sF$ be a Fano foliation on $X$ of rank $r$ and index $\iota$ defined  
by a twisted $(n-r)$-form
$\omega\in H^0(X,\Omega^{n-r}_X(n+2-d-\iota))$.
Notice that $1\le n-r \le n-2$.
By lemma \ref{lemma:vanishing_hypersurface}, we must have 
$$n-r < n+2-d-\iota ,$$ or, equivalently, $$d+\iota \le r+1.$$

Conversely, let $r\in\{2, \cdots, n-1\}$ and $\iota$ be such that $d+\iota \le r+1$.
Let $\omega\in H^0(\p^{n+1},\Omega^{n-r}_{\p^{n+1}}(n+2-d-\iota))$
be a general twisted $(n-r)$-form defining a Fano foliation of rank $r+1$ and index $d+\iota\leq r+1$ on $\p^{n+1}$.
Then $\omega|_{X}\in H^0(X,\Omega^{n-r}_{X}(n+2-d-\iota))$
defines a foliation on $X$ of rank $r$ and index $\iota$. 
\end{proof}

\begin{cor}
Fix $n \ge 3$, and let $X\subset\p^{n+1}$ be a smooth
hypersurface of degree $d \ge 2$. Then there exists a Fano foliation on $X$ of rank $r\in\{2, \cdots, n-1\}$
and index $\iota = r-1$ if and only if $d=2$.
\end{cor}

\begin{proof}
Suppose there exists a  Fano foliation on $X$ of rank $r\in\{2, \cdots, n-1\}$
and index $\iota = r-1$ on $X$.
By Proposition~\ref{prop:foliation_hypersurface}, we must have $d \le 2$. 
Conversely, a foliation of rank $r+1$ and 
and index $\iota = r+1$ on $\p^{n+1}$ induces a foliation
of rank $r$ and index $\iota=r-1$ on $X$.
\end{proof}

\begin{question}
Let $X\subset \p^{n+1}$ be a smooth hypersurface of degree $d\ge 2$ and dimension 
$n\ge 3$.
Let $\sF\subsetneq T_X$ be  a Fano foliation of rank $r$ and index $\iota$ on $X$, with 
$d+\iota = r+1$.
Is $\sF$ induced by a Fano foliation  of rank $r+1$ and index  $r+1$ on $\p^{n+1}$?
\end{question}

In Section~\ref{section:examples_pn_bundles}, we provide several examples of del Pezzo 
foliations on projective space bundles.

\section{The relative anticanonical bundle of a fibration and applications}
\label{section:relative_-K}

In  \cite[Theorem 2]{miyaoka93}, Miyaoka proved that the anticanonical bundle of a smooth 
projective morphism $f:X\to C$ onto a smooth proper curve cannot be ample. 
In \cite[Theorem 3.1]{adk08}, this result was 
generalized by dropping the smoothness assumption, and 
replacing $-K_{X/C}$ with $-(K_{X/C}+\Delta)$, 
where $\Delta$ is an effective Weil divisor on $X$ such that $(X,\Delta)$ is log canonical over the generic point of $C$. In this section we give a further generalization of this result and provide
applications to the theory of Fano foliations.

\begin{thm}\label{thm:-KX/Y_not_ample}
\label{cor:-KX/Y_not_nef_and_big}
Let $X$ be a normal projective variety, and $f:X\to C$ a surjective morphism 
with connected fibers onto a smooth curve. 
Let $\Delta_+\subseteq X$ and $\Delta_{-}\subseteq X$ be effective
Weil $\bQ$-divisors with no common components such that  $f_*\sO_X(k\Delta_{-})=\sO_C$ 
for every non negative integer $k$.
Set $\Delta:=\Delta_+-\Delta_{-}$, and assume that $K_X+\Delta$ is $\bQ$-Cartier. 
\begin{enumerate}
	\item If $(X,\Delta)$ is log canonical over the generic point of $C$, 
		then $-(K_{X/C}+\Delta)$ is not ample.
	\item If $(X,\Delta)$ is klt over the generic point of $C$,
		then $-(K_{X/C}+\Delta)$ is not nef and big.
\end{enumerate}
\end{thm}

\begin{proof}
To prove (1), we assume to the contrary that $(X,\Delta)$ is log canonical over the generic 
point of $C$, and $-(K_{X/C}+\Delta)$ is ample.  Let $\pi: \tilde X\to X$ be a log
resolution of singularities of $(X,\Delta)$, $A$ an ample divisor on
$C$, and $m\gg 0$ such that $D=-m(K_{X/C}+\Delta)-f^*A$ is very
ample.  Then
$$
K_{\tilde X}+\pi^{-1}_* \Delta_+-\pi^{-1}_* \Delta_-=\pi^*(K_X +\Delta_+-\Delta_-)+E_+-E_-,
$$
where $E_+$ and $E_-$ are effective $\pi$-exceptional divisors with
no common components and the support of
$\pi^{-1}_*\Delta+E_++E_-$ is a snc divisor.

Set $\tilde f:=f\circ \pi$ and let $\tilde D\in |\pi^*D|$ be a general
member. Setting $\tilde\Delta_+=\pi^{-1}_*\Delta_+ +\frac 1m\tilde D+E_-$,
we obtain that $(\tilde X,\tilde \Delta_+)$ is log canonical 
over the generic point of $C$
and that
$$
 K_{\tilde X}+\tilde \Delta_+ \sim_{\bQ} \tilde f^*K_C +E_+ +\pi^{-1}_* \Delta_- 
-\frac{1}{m}\tilde f^*A. 
$$
 
Furthermore, since $E_+$ is effective and $\pi$-exceptional, 
$\pi_* \sO_{\tilde  X}(lE_+)=\sO_X$ for any $l\in\bN$. Then for any $l\in \bN$,

\begin{multline*}
\tilde f_*\sO_{\tilde X}(lm(K_{\tilde X/C}+\tilde \Delta_+))
\simeq
\tilde f_*\sO_{\tilde X}(l(mE_+ +m\pi^{-1}_* \Delta_-  -\tilde f^*A)) \\
\simeq
\tilde f_*\sO_{\tilde X}(l(mE_+ +m\pi^{-1}_*  \Delta_-))\otimes \sO_C(-lA).
\end{multline*}

Observe that $\tilde f_*\sO_{\tilde X}(lm(E_+ +\pi^{-1}_*  \Delta_-))=\sO_C$. Indeed, let $U\subseteq C$ be a non
empty open subset and let $\tilde \lambda\in H^0(\tilde f^{-1}(U),\sO_{\tilde X}(lm(E_+ +\pi^{-1}_*  \Delta_-)))$
that is, $\tilde \lambda$ is a rational function on
$\tilde X$ such that $\textup{div}(\tilde \lambda)+lm(E_+ +\pi^{-1}_* \Delta_- )\ge 0$ over $\tilde f^{-1}(U)$. 
Let $\lambda$ be the unique rational function on $X$ such that 
$\tilde \lambda=\pi \circ \lambda$. Then $\textup{div}(\lambda)+lm\Delta_- \ge 0$ over $f^{-1}(U)$ since
$E_+$ is $\pi$-exceptional. 
Since $f_*\sO_X(lm\Delta_{-})=\sO_C$ by assumption, 
there exists a regular function $\mu$ on $U$ such that 
$\lambda = f \circ \mu$ over $f^{-1}(U)$. Thus the natural map 
$\sO_C \hookrightarrow \tilde f_*\sO_{\tilde X}(lm(E_+ +\pi^{-1}_*  \Delta_-))$
is an isomorphism as claimed.

Finally, observe that 
$\tilde f_*\sO_{\tilde X}(lm(K_{\tilde X/C}+\tilde\Delta_+))$ is
semi-positive by \cite[Theorem 4.13]{campana04}, but that contradicts the fact that $A$ is ample.
This proves (1).

\medskip

To prove (2), we assume to the contrary that $(X,\Delta)$ is klt over the generic 
point of $C$, and $-(K_{X/C}+\Delta)$ is nef and big.
There exists an effective $\bQ$-Cartier $\bQ$-divisor $N$ on $X$ such that 
$-(K_{X/C}+\Delta)-\epsi N$ is ample for $0<\epsi \ll 1$.
Let $0<\epsi \ll 1$ be
such that $(X,\Delta+\epsi N)$ is klt over the generic
point of $C$. 
Set $\Delta_+':=\Delta_++\epsi N$, $\Delta_-':=\Delta_-$, and $\Delta':=\Delta+\epsi N$. 
Then $$-(K_{X/C}+\Delta')=-(K_{X/C}+\Delta)-\epsi N$$
is ample, contradicting part (1) above. This proves (2).
\end{proof}

\begin{rem}
Examples arising from Fano foliations show that 
Theorem~\ref{thm:-KX/Y_not_ample} is sharp.
To fix notation, let $\sF\subsetneq T_X$ be an  algebraically integrable Fano foliation  on a smooth
projective variety $X$, with general log leaf $(\tilde F, \tilde \Delta)$.
We let $C\subset \Chow(X)$ be a general complete curve contained in the closure of the subvariety parametrizing general leaves of $\sF$.
We denote by $U$ the normalization of the universal cycle over $C$, with universal morphism $e:U\to X$.
Since $C$ is general, $e:U\to X$ is birational onto its image. 
By Remark~\ref{remark:definition_log_leaf}, there is a canonically defined Weil divisor $\Delta$ on $U$ such that $-(K_{U/C}+\Delta)=e^*(-K_{\sF})$.
In particular, since $-K_{\sF}$ is ample, $-(K_{U/C}+\Delta)$ is always nef and big.
It is ample if and only if the leaves parametrized by $C$ have no common point. 
Moreover,  $(U,\Delta)$ is log canonical over the generic point of $C$ if and only if $(\tilde F, \tilde \Delta)$ is log canonical.

By choosing $\sF$ to have log canonical singularities along a general leaf,
we see that we cannot strengthen the conclusion of Theorem~\ref{thm:-KX/Y_not_ample} by replacing ``ample" with ``nef and big".

On the other hand, consider the rank $2$ del Pezzo foliation $\sF$ on $X=G(2,V)$ defined in \ref{example:grass}.
Then $(\tilde F, \tilde \Delta)\simeq  (\p^2, 2H)$, where $H$ is a line in $\p^2$. So it is not log canonical, while in this case
$-(K_{U/C}+\Delta)$ is ample. So we  cannot relax the assumption that 
$(U,\Delta)$ is log canonical over the generic point of $C$ in Theorem~\ref{thm:-KX/Y_not_ample}.
\end{rem}

As a first application of Theorem~\ref{thm:-KX/Y_not_ample}, we derive a special property of 
Fano foliations with mild singularities. This property will play a key role in our study of Fano foliations.

\begin{prop}\label{lemma:common_point}
Let $X$ be a normal projective variety, and $\sF\subsetneq T_X$ an algebraically integrable 
Fano foliation on $X$. 
If $\sF$ has log canonical singularities along a general leaf, 
then there is a common point in the closure of a general leaf of $\sF$.
\end{prop}

\begin{proof}
Let $W$ be the normalization of the closure in $\Chow(X)$ of the subvariety parametrizing 
general leaves of $\sF$, and  $U$ the normalization of the universal cycle over $W$,
with universal family morphisms:

\centerline{
\xymatrix{
U \ar[r]^{e}\ar[d]_{\pi} & X \ . \\
 W &
}
}
\noindent Denote by $U_w$ the fiber of $\pi$ over a point $w\in W$.

For every $x\in X$,
${\pi}|_{e^{-1}(x)}:e^{-1}(x) \to W$ is finite. 
If we show that $\dim( e^{-1}(x)) \ge \dim(W)$ for some $x\in X$, then 
we conclude that
$\pi\big(e^{-1}(x)\big)=W$, and thus
$x \in e\big(U_w\big)$ for every $w\in W$, i.e., 
$x$ is contained in the closure of a general leaf of $\sF$.

Suppose to the contrary that $\dim\big( e^{-1}(x)\big) < \dim( W)$ for every $x\in X$. 
Let $C \subset W$ be a general complete intersection curve, and let $U_C$
be the normalization of $\pi^{-1}(C)$,
with natural morphisms $\pi_C: U_C\to C$ and $ e_C: U_C\to X$.
Since $C$ is general,  $C$ is not contained in $\pi \big( e^{-1}(x)\big)$  for any $x \in X$,
and thus the morphism $ e_C: U_C\to X$ is finite onto its image.
In particular, $e_C^*(- K_{\sF})$ is ample.

By Remark~\ref{remark:definition_log_leaf}, 
$\sF$ induces a generically surjective morphism
$\Omega_{U_C/C}^{r_{\sF}} \to e_C^*\det(\sF)^*$.
By Lemma~\ref{lemma:reduced_fiber} below followed by Lemma~\ref{lemma:extensionpfafffields}, 
after replacing $C$ with a finite cover if necessary, we 
may assume that $\pi_C$ has reduced fibers.
This implies that
$\det(\Omega_{U_C/C}^{1})\simeq \sO_{U_C}(K_{U_C/C})$. Thus
there exists an effective integral divisor $\Delta_C$ on $U_C$ such that 
$$
 -(K_{U_C/C}+\Delta_C) = e_C^*(- K_{\sF}).
 $$
Since  $\sF$ has log canonical singularities along a general leaf, 
the pair $(U_C, \Delta_C)$ is log canonical over the generic point of $C$.
But this contradicts Theorem~\ref{thm:-KX/Y_not_ample}, and the result follows.
\end{proof}

\begin{lemma}[{\cite[Theorem 2.1']{bosch95}}]\label{lemma:reduced_fiber} 
Let $X$ be a quasi-projective variety, and 
$f : X \to C$ a flat surjective morphism onto a smooth curve with
reduced general fiber. Then there exists a finite morphism $C' \to C$ such that $f' : X' \to C'$
is flat with reduced fibers. Here $X'$ denotes the normalization of $C' \times_C X$ and $f ' : X' \to C'$
is the morphism induced by the projection $C'\times_C X \to C'$.
\end{lemma}

Using these ideas, next we give elementary proof of the Lipman-Zariski conjecture for klt spaces (see 
\cite[Theorem 6.1]{greb_kebekus_kovacs_peternell10}).

\begin{prop}\label{prop:minimal_singularities}
Let $\sF$ be an algebraically integrable  foliation on a normal  projective variety $X$.  Suppose
that $\sF$ has log terminal singularities and is locally free along a general leaf.
Then 
the leaf through a
general point of $X$ is proper and smooth and 
there exists an almost proper map $X \dashrightarrow Y$ whose general 
fibers are leaves of $\sF$.
\end{prop}

\begin{proof}
Let $W$ be the normalization of the closure in $\Chow(X)$ of the subvariety parametrizing 
general leaves of $\sF$, and  $U$ the normalization of the universal cycle over $W$. 
By \cite[Theorem 3.35, 3.45]{kollar07} 
(see also  \cite[Corollary 4.7]{greb_kebekus_kovacs10} and Theorem~\ref{thm:canonical_resolution}), 
there is a log resolution of singularities
$d:Y \to U$ such that $d_*T_{Y}(-\log \Sigma)=T_{ U}$ where 
$\Sigma\subset Y$ is the largest
reduced divisor contained in $d^{-1}(\textup{Sing}( U))$.
Consider the commutative diagram:

\centerline{
\xymatrix{
Y\ar[r]^{d}\ar@/_0.5pc/[dr]\ar@/^1.5pc/[rr]^{g} 
&  U \ar[r]^{ e}\ar[d]_{ \pi} & X \\
&  W &
}
}

By assumption, there is a dense open subset  $ W_0\subset  W$ 
such that $e^*\sF$ is locally free along $U_0:=\pi^{-1}( W_0)$.
We set $Y_0:=d^{-1}( U_0)$ and $\Sigma_0:=\Sigma|_{Y_0}$.
By Remark~\ref{remark:definition_log_leaf}, since $e^*\sF$ is locally free along $U_0$,
$\sF$ induces a foliation  
${({e}^*\sF)}|_{U_0}\subset T_{U_0}$. Thus there is an injection 
$$
{g^*\sF}|_{Y_0}\hookrightarrow T_{Y_0}(-\log \Sigma_0)\subset T_{Y_0}.
$$
Hence there exists an effective integral
divisor $\Sigma'_0$ on $Y_0$ such that
$$K_{Y_0}+\Sigma_0+\Sigma'_0={(g^*K_\sF)}|_{Y_0}.$$

Let $w\in W_0$ be a general point, and $(U_w,\Delta_w)$ the corresponding log leaf. 
Set $Y_w:=d^{-1}(U_w)$, $d_w:=d|_{Y_w}:Y_w\to U_w$, $\Sigma_w:=\Sigma|_{Y_w}$, 
and $\Sigma'_w:=\Sigma'|_{Y_w}$. 
By assumption, $(U_w,\Delta_w)$ is log terminal. Thus $\Delta_w=0$.
Hence $K_{U_w}={(e^*K_{\sF})}|_{ U_w}$, and we get 
$$
K_{Y_w}+\Sigma_w+\Sigma'_w={(g^*K_\sF)}|_{Y_w}=d_w^* K_{U_w}.
$$
Notice that $d_w:Y_w \to  U_w$ is a log resolution of singularities, and recall
that $U_w$ is log terminal.
On the other hand, $\Sigma_w$ and $\Sigma'_w$ are effective integral divisors on $Y_w$.
So  we must have $\Sigma_w=\Sigma'_w=0$.
This implies that $U_w$ is smooth, and 
by Lemma~\ref{lemma:singular_locus_normalization} below, $\sF$ is regular along 
the image of $U_w$.
\end{proof}

\begin{lemma}\label{lemma:singular_locus_normalization}
Let $Y$ be a variety with normalization morphism $n:\tilde Y\to Y$. 
Let $\sG$ be a locally free sheaf of rank $r_\sG$ on $Y$ and $\eta : \Omega_Y^1 \to \sG$ be any morphism. Let 
$\tilde \eta : \Omega_{\tilde Y}^1 \to n^*\sG$ be the extension given by
 \cite{seidenberg66}. Let $r$ be a positive integer and
let $S_r(\eta)$ (resp. $S_r(\tilde\eta)$) be the locus where 
$\wedge^r \eta : \Omega_Y^r \to \wedge^r \sG$ 
(resp. $\wedge^r \tilde \eta : \Omega_{\tilde Y}^r \to n^*\wedge^r \sG$)
is not surjective.
Then $S_r(\tilde\eta)=n^{-1}(S_r(\eta))$.
\end{lemma}

\begin{proof}
If $\eta$ has rank $> r$ at a some point $y$ in $Y$ then 
$n^*\eta$ has rank $> r$ at every point in $n^{-1}(y)$ and thus
$\tilde \eta$ has rank $> r$ at every point in $n^{-1}(y)$. Therefore
$S_r(\tilde\eta)\subset n^{-1}(S_r(\eta))$.

Let us assume that $\eta$ has rank $\le r$ at some point $y$ in $Y$. By shrinking $Y$ if necessary, we may decompose
$\sG$ as $\sO_Y^{\oplus r_\sG-r}\oplus \sG_1$ in such a way that the induced morphism 
$\eta_0:\Omega_Y^1 \to \sO_Y^{\oplus r_\sG-r}$ is zero at $y$.
Write $\eta=(\eta_0,\eta_1)$ and let $\tilde \eta_0:\Omega_{\tilde Y}^1 \to n^*\sO_Y^{\oplus r_\sG-r}$ 
(respectively $\tilde \eta_1:\Omega_{\tilde Y}^1 \to n^*\sG_1$)
be the extension of $\eta_0:\Omega_Y^1 \to \sO_Y^{\oplus r_\sG-r}$ (respectively $\eta_1:\Omega_Y^1 \to \sG_1$) given by 
\cite{seidenberg66}. Then $\tilde\eta=(\tilde\eta_0,\tilde\eta_1)$ and the claim then follows from
\cite[Lemme 1.2]{druel04}.
\end{proof}

\begin{cor}[Lipman-Zariski conjecture for klt spaces. See also {\cite[Theorem 6.1]{greb_kebekus_kovacs_peternell10}}]
Let $X$ be a klt space such that the tangent space $T_X$ is locally free. Then $X$ is smooth. 
\end{cor}

\begin{proof}The result follows from
proposition \ref{prop:minimal_singularities} applied to the foliation induced by the projection morphism $X\times C \to C$,
where $C$ is a smooth complete curve.
\end{proof}

\begin{prop}
Let $\sF$ be a $1$-Gorenstein algebraically integrable  foliation on a  normal projective variety $X$. Suppose
that $\sF$ has log terminal singularities along a general leaf.
Then $\det(\sF)$ is not nef and big. 
\end{prop}

\begin{proof}
We let $C\subset \Chow(X)$ be a general complete curve contained in the closure of the subvariety parametrizing general leaves of $\sF$.
We denote by $U$ the normalization of the universal cycle over $C$, 
with natural morphisms $\pi:U\to C$ and $e:U\to X$.
Since $C$ is general, $e:U\to X$ is birational onto its image. 
Thus if $-K_{\sF}$ is nef and big, then so is $e^* (-K_{\sF})$. 

By Remark~\ref{remark:definition_log_leaf}, $\sF$ induces a Pfaff field
$\Omega_{U/C}^{r} \to e^*\sO_C(-K_{\sF})$, where $r$ denotes the rank of $\sF$.
By Lemma \ref{lemma:reduced_fiber} followed by Lemma \ref{lemma:extensionpfafffields}, 
after replacing $C$ with a finite cover if necessary,
we may assume that $\pi$ has reduced fibers. This implies that
$\det (\Omega_{U/C}^{r})\simeq \sO_{U}(K_{U/C})$. Thus
there exists a canonically defined effective divisor $\Delta$ on $U$ 
such that 
$$
-(K_{U/C}+\Delta) =  e^*(- K_{\sF}).
$$

By assumption, $(U,\Delta)$ is log terminal over the generic point of $C$.
So, by Theorem~\ref{thm:-KX/Y_not_ample}, $e^* (-K_{\sF})$ cannot be nef and big.
\end{proof}


\section{Foliations and rational curves}
\label{section:rat_curves}

If a smooth projective variety $X$ admits a Fano foliation $\sF$, then it is 
uniruled, as we have observed in Remark~\ref{uniruledness}.
In order to study the pair $(X,\sF)$, it is useful to understand the behavior of $\sF$ with respect 
to  families of rational curves on $X$. 
This is the theme of this section.
We start by recalling some definitions and results from the theory of rational curves on 
smooth projective
varieties. We refer to \cite{kollar96} for more details.

Let $X$ be a smooth projective variety, and $H$ a family of rational curves on $X$, i.e., 
an irreducible component of $\rat(X)$. 
If $C$ is a curve from the family $H$, with
normalization morphism $f:\p^1\to C\subset X$, then
we denote by $[C]$ or
$[f]$ any point of $H$ corresponding to $C$.  
We denote by $Locus(H)$ the locus of $X$ swept out by curves from $H$.
We say that $H$ is \emph{unsplit} if it is proper, and \emph{minimal} if, 
for a general point $x\in Locus(H)$,
the closed subset $H_x$ of $H$ parametrizing curves through $x$ is proper.
We say that $H$ is \emph{dominating}
if $\overline{Locus(H)}=X$.  
In this case we say that a curve $C$ parametrized by $H$ is a \emph{moving curve} on $X$,
and that any curve from $H$ is a deformation of $C$.

\begin{say}[Minimal dominating families of rational curves]
Let $X$ be a smooth projective uniruled variety. 
Then $X$ always carries a minimal dominating family of rational curves.
Fix one such family $H$, and let $[f]\in H$ be a general point. 
By \cite[IV.2.9]{kollar96},
$f^*T_X\simeq \sO_{\p^1}(2)\oplus \sO_{\p^1}(1)^{\oplus d}\oplus
\sO_{\p^1}^{\oplus (n-d-1)}$, where $d=\deg(f^*T_X)-2\geq 0$.

Given a general point $x\in X$, let $\tilde H_x$ be the normalization of $H_x$.
By \cite[II.1.7, II.2.16]{kollar96}, 
$\tilde H_x$ is a finite union of smooth projective varieties of dimension
$d=\deg(f^*T_X)-2$.  
Define the tangent map $ \tau_x: \ \tilde H_x \map  \p(T_xX^*) $ 
by sending a curve that is smooth at $x$ to its
tangent direction at $x$.  Define $\cC_x$ to be the image of
$\tau_x$ in $\p(T_xX^*)$.  This is called the \emph{variety of
minimal rational tangents} at $x$ associated to the minimal family  $H$.
The map $\tau_x: \ \tilde H_x \to \cC_x$ is in fact the normalization
morphism by \cite{kebekus02} and \cite{hwang_mok04}.  
\end{say}

\begin{defn} 
Let $\pi_0:X_0\to Y_0$ be a proper morphism
defined on a dense open subset of $X$. 
A family of rational curves $H$ on $X$ is said to be \emph{horizontal (with respect to $\pi_0$)}
if the general member of $H$ meets $X_0$ and is not contracted by $\pi_0$.
If moreover $Locus(H)$ dominates $Y_0$, then we say that $H$ is 
\emph{h-dominating (with respect to $\pi_0$)}.

Notice that if $X$ admits a horizontal family of rational curves,
then it admits a minimal horizontal family of rational curves. Indeed, it is enough to take a horizontal 
family having minimal degree with respect to some fixed ample line bundle on $X$.
Similarly for h-dominating families.
\end{defn}

\begin{lemma} \label{lemma:horizontal}
Let $X$ be a smooth projective variety, and $\pi_0:X_0\to Y_0$ a 
surjective proper morphism defined on a dense open subset of $X$. 
Suppose that  $H$ is a minimal horizontal family of rational curves with respect to $\pi_0$.
Then $-K_X\cdot H\leq \dim Y_0+1$, where  $-K_X\cdot H$ denotes the intersection number 
of  $-K_X$ with any curve from the family $H$. Moreover
\begin{itemize}
	\item If  $-K_X\cdot H= \dim Y_0+1$, then $H$ is dominating.
	\item If  $-K_X\cdot H= \dim Y_0$, then $Locus(H)$ has codimension at most 1 in $X$.
\end{itemize}
\end{lemma}

\begin{proof}
Let $x$ be a general point in $Locus(H)$, and denote by $Locus(H_x)$ the 
locus of $X$ swept out by curves from $H_x$. By assumption, any irreducible component 
$Z$ of $Locus(H_x)$ is proper. Moreover, by \cite[IV.3.13.3]{kollar96}, any curve in $Z$
is numerically proportional in $X$ to a curve from the family $H$. In particular $Z$ cannot 
contain any curve contracted by $\pi_0$.  Therefore $\dim\big(Locus(H_x)\big)\leq \dim(Y_0)$.

On the other hand, by \cite[IV.2.6.1]{kollar96},
$\dim(X)+ (-K_X\cdot H) \leq  \dim\big(Locus(H)\big) + \dim\big(Locus(H_x)\big) +1$. Thus
$$
-K_X\cdot H \ \leq \ \dim(Y_0)+1-\Big(\dim(X)-\dim\big(Locus(H)\big)\Big)\ \leq \  \dim Y_0+1.
$$
If $-K_X\cdot H= \dim Y_0+1$, then we must have $\dim\big(Locus(H)\big)=\dim(X)$, i.e., 
$H$ is dominating. If $-K_X\cdot H= \dim Y_0$, then we must have $\dim(X)-\dim\big(Locus(H)\big)\leq 1$.
\end{proof}

\begin{say}[Rationally connected quotients] 
Let $H_1, \dots, H_k$ be families of rational curves
on $X$.  For each $i$, let $\overline H_i$ denote the closure of $H_i$
in $\Chow(X)$.  Two points $x,y\in X$
are said to be $(H_1, \dots, H_k)$-equivalent if they can be connected by a chain
of 1-cycles from $\overline H_1\cup \cdots \cup \overline H_k$.  
This defines an equivalence relation on $X$.  
By \cite{campana} (see also \cite[IV.4.16]{kollar96}), there exists a
proper surjective equidimensional morphism $\pi_0:X_0 \to T_0$ from a
dense open subset of $X$ onto a normal variety whose fibers are $(H_1,
\dots, H_k)$-equivalence classes.  We call this map the
 \emph{$(H_1,  \dots, H_k)$-rationally connected quotient of $X$}.  
When $T_0$
is a point we say that $X$ is $(H_1, \dots, H_k)$-rationally
connected.
\end{say}

From now on we investigate the behavior of foliations on a smooth projective variety $X$
with respect  to  families of rational curves on $X$. We start with a simple but useful observation.

\begin{lemma}\label{lemma:curve_tangent_to_F}
Let $X$ be a smooth projective variety, $H$ a family of rational curves on $X$, 
and $\sF$ an algebraically integrable foliation on $X$.
Suppose that $\ell$ is contained in a leaf of $\sF$ and avoids the singular locus of $\sF$
for some $[\ell]\in H$. 
Then the same holds for general $[\ell]\in H$. 
\end{lemma}

\begin{proof}
Let $W$ be the closure in $\Chow(X)$ of the subvariety parametrizing general leaves of $\sF$, with
universal family morphisms:
\[
\xymatrix{
U \ar[d]_p \ar[r]^q & X. \\
W} 
\]
Let $\sA_W$ be a general very ample effective divisor on $W$, and set $\sA=q_*(p^*(\sA_W))$.

The condition that $\ell$ is contained in a leaf of $\sF$ and avoids the singular locus $S$ of $\sF$
is equivalent to the condition that $\ell\cap S=\emptyset$ and $\sA\cdot \ell=0$.
Hence, if this condition holds for some $[\ell]\in H$, then it holds for general $[\ell]\in H$.
\end{proof}

\begin{lemma}\label{lemma:foliation_rc_quotientB}
Let $X$ be a smooth projective uniruled variety, $H_1, \cdots, H_k$ 
unsplit families of rational curves on $X$, and $\sF$ 
an algebraically integrable foliation on $X$. 
Denote by $\pi_0:X_0 \to T_0$ the $(H_1, \cdots, H_k)$-rationally
connected quotient of $X$.
Suppose that a general curve from each of the families $H_i$'s
is contained in a leaf of $\sF$ and avoids the singular locus of $\sF$.
Then
there is an inclusion $T_{X_0/T_0} \subset \sF|_{X_0}$. 
\end{lemma}

\begin{proof}
Let $W$ be the closure in $\Chow(X)$ of the subvariety parametrizing general leaves of $\sF$, with
universal family morphisms:
\[
\xymatrix{
U \ar[d]_p \ar[r]^q & X. \\
W} 
\]
Let $A_W$ be a general very ample effective divisor on $W$, and set $A=q_*(p^*(A_W))$.
By assumption, a general curve $\ell\subset X$ parametrized by each $H_i$
is contained in a leaf of $\sF$, and avoids the singular locus of $\sF$. Thus $A\cdot \ell = 0$. 

Let $X_t=(\pi_0)^{-1}(t)$ be a general fiber of $\pi_0$. 
By \cite[IV.3.13.3]{kollar96}, 
every proper curve $C\subset X_t$ is numerically equivalent in $X$ to 
a linear combination of curves from the families $H_i$'s , and so
$A\cdot C=0$.  
This shows that $A|_{X_t}\equiv 0$, and thus $X_t \subset q(p^{-1}(w))$ for some $w\in W$, i.e.,
$X_t$ is contained in a leaf of $\sF$. 
We conclude that $T_{X_0/T_0}\subset \sF|_{X_0}$
by Lemma~\ref{lemma:foliation_morphism} below.
\end{proof}

\begin{lemma}\label{lemma:foliation_morphism}
Let $\sF$ be a foliation of rank $r_\sF$ on a normal variety $X$,  and $\pi : X \to Y$ an equidimensional
morphism with connected fibers onto a normal variety.
Suppose that the general fiber of $\pi$ is contained in a leaf of $\sF$. 
Then $\sF$ induces a foliation $\sG$ of rank $r_\sG=r_\sF-\big(\dim(X)-\dim(Y)\big)$ on $Y$,  
together with an
exact sequence $$0 \to T_{X/Y} \to \sF \to (\pi^*\sG)^{**}.$$
\end{lemma}

\begin{defn}
Under the hypothesis of Lemma~\ref{lemma:foliation_morphism}, we say that $\sF$ is the 
\emph{pullback via $\pi$ of the foliation $\sG$}.
\end{defn}

\begin{proof}[Proof of Lemma~\ref{lemma:foliation_morphism}]
Notice that the induced map $T_{X/Y} \to T_X / \sF$ is generically zero by assumption.
Since $T_X / \sF$ is torsion free, it must be identically zero, hence we have an inclusion
$T_{X/Y} \subset \sF$.

First we define the foliation $\sG\subset T_Y$ induced by $\sF$ analytically. 
Let $y \in Y$ be a general point. Choose an analytic open neighborhood $V\subset Y$ of $y$,
and a local holomorphic section $s:V \to X$ of $\pi$. 
There exists an analytic open neighborhood $U\subset X$ of $x=s(y)$, and a complex analytic space $W$
such that the leaves
of $\sF|_U$ are the fibers of a holomorphic map $p : U \to W$.
After shrinking $V$ if necessary, we get a holomorphic map $p\circ s: V \to W$,
which defines a foliation $\sG_s$ on $V$. Notice that 
\begin{equation}\label{eq:f*F}
T_y\sG_s \ = \ d\pi_x\big(T_x\sF\big),
\end{equation}
where $d\pi_x:T_xX\to T_yY$ denotes the tangent map of $\pi$ at $x$, 
$T_x\sF$ and $T_y\sG$ denote the fibers of $\sF\subset T_X$ and $\sG_s\subset T_Y$ at 
$x$ and $y$, respectively.
Notice that $\sG_s$ does not depend on the choice
of local section $s:V \to X$, since $X$ is normal, $T_{X/Y} \subset \sF$, and 
$\pi$ has connected fibers. Moreover, these foliations defined locally glue and extend 
to a foliation $\sG$ of rank $r_\sG=r_\sF-\big(\dim(X)-\dim(Y)\big)$ on $Y$.

Next we give an algebraic description of $\sG$.
Since $\pi$ is equidimensional, there are  dense open subsets $X_0\subset X$ and 
$Y_0\subset Y$  such that 
$\codim_X\big(X\setminus X_0)\geq 2$, $\codim_Y\big(Y\setminus Y_0)\geq 2$,
$\pi_0:=\pi|_{X_0}$ maps $X_0$ into $Y_0$, $\sF_0:=\sF|_{X_0}$ is a subbundle of $T_{X_0}$,
and $\sG_0:=\sG|_{Y_0}$ is a subbundle of $T_{Y_0}$.
Consider the tangent map $d\pi_0 : T_{X_0} \to (\pi_0)^*T_{Y_0}$.
By \eqref{eq:f*F}, $\sG_0$ coincides with the saturation of the subsheaf 
$(\pi_0)_*\big(d\pi_0(\sF_0)\big)$ in $T_{Y_0}$, and the induced map 
$\alpha: d\pi_0(\sF_0)\to (\pi_0)^*T_{Y_0}\big/(\pi_0)^*\sG_0$ is generically zero.
Since $\sG_0$ is a subbundle of $T_{Y_0}$, $(\pi_0)^*T_{Y_0}\big/(\pi_0)^*\sG_0$
is torsion free, and hence the map $\alpha$ must be identically zero. 
So we have an exact sequence $0 \to T_{X_0/Y_0} \to \sF_0 \to (\pi_0)^*\sG_0$.
Note that the sheaves $T_{X/Y}$, $\sF$ and $(\pi^*\sG)^{**}$ are reflexive. Since reflexive
sheaves on a normal variety are normal sheaves (\cite[Proposition 1.6]{hartshorne80}), and $\codim_X\big(X\setminus X_0)\geq 2$, 
we obtain an exact sequence $0 \to T_{X/Y} \to \sF \to (\pi^*\sG)^{**}$.
\end{proof}

In the setting of Lemma~\ref{lemma:foliation_rc_quotientB}, if moreover the families 
$H_i$'s are dominating, then we may drop the assumption that $\sF$ is algebraically integrable.
This is the content of the next lemma.

\begin{lemma}\label{lemma:foliation_rc_quotient}
Let $X$ be a smooth projective variety, $H_1, \cdots, H_k$ unsplit  dominating 
families of rational curves on $X$, and $\sF$ a foliation on $X$. 
Denote by $\pi_0:X_0 \to T_0$ the $(H_1, \cdots, H_k)$-rationally
connected quotient of $X$.
If  $T_{\p^1} \subset f^*\sF$ for general $[f] \in H_i$, $0\leq i \leq k$, then
there is an inclusion $T_{X_0/T_0} \subset \sF|_{X_0}$. 
\end{lemma}

\begin{proof}
Recall that for general $[f] \in H_i$ one has $f^*T_X\simeq \sO_{\p^1}(2)\oplus \sO_{\p^1}(1)^{\oplus d_i}\oplus
\sO_{\p^1}^{\oplus (n-d_i-1)}$, where $n=\dim X$ and $d_i=\deg(f^*T_X)-2$.
Hence, the assumption  $\sO_{\p^1}(2)\simeq T_{\p^1} \subset f^*\sF$ implies that the natural  inclusion
$T_{\p^1} \subset f^* T_X$ factors through $f^*\sF\into f^*T_X$. 
Therefore a general curve from each of the families $H_i$'s
is contained in a  leaf of $\sF$.

Let $x\in X$ be a general point. We define inductively a sequence of (irreducible) subvarieties of $X$ as follows. 
Set $V_0(x) := \{x\}$, and let $V_{j+1}(x)$ be the closure of the union of curves from 
the families $H_i$, $0\leq i \leq k$, that pass through a general point of $V_j(x)$. 

Then $\dim V_{j+1}(x) \geq \dim V_{j}(x)$, and equality holds if and only if $V_{j+1}(x)=V_{j}(x)$. 
In particular,
there exists $j_0$ such that $V_{j}(x)=V_{j_0}(x)$ for every $j\geq j_0$. We set $V(x)=V_{j_0}(x)$.
Since $x$ is general, $V(x)$ is smooth at $x$. Notice also that $V(x)$ is irreducible, and
that $V(x)$ is contained in the leaf of $\sF$ through $x$  by construction.

We define the subfoliation $\sV\subset \sF$ by setting $\sV_x=T_xV(x)$ for general $x\in X$. The leaf of 
$\sV$ through $x$ is precisely $V(x)$. In particular $\sV$ is an algebraically integrable foliation of $X$.
Moreover, by construction, a general curve from each of the families $H_i$'s
is contained in a leaf of $\sV$, and avoids the singular locus of $\sV$ by \cite[II.3.7]{kollar96}.
The result then follows from Lemma~\ref{lemma:foliation_rc_quotientB}.
\end{proof}

Next we apply the results from the previous section to characterize pairs $(X,\sF)$
 when $\sF$ is a Fano foliation that is ample when restricted to a general member of
 a minimal covering family of rational curves on $X$.

\begin{lemma}\label{lemma:Fano_ample} 
Let $X$ be an $n$-dimensional smooth projective variety admitting a minimal dominating 
family of rational curves $H$.
Let $\sF\subsetneq T_X$ be a  Fano foliation of rank $r$ on $X$. 
If $f^*\sF$ is an ample vector bundle for general $[f]\in H$, then 
$(X,\sF)\simeq \big(\p^n,\cO(1)^{\oplus r}\big)$.
\end{lemma}

\begin{proof}
Denote by $\pi_0:X_0 \to T_0$ the $H$-rationally
connected quotient of $X$.
By \cite[Proposition 2.7]{adk08}, after shrinking $X_0$ and $T_0$ if necessary,
we may assume that $\pi_0$ is a $\p^k$-bundle, and the inclusion
$\sF|_{X_0}\into T_{X_0}$ factors through the natural inclusion
$T_{X_0/T_0}\into T_{X_0}$. 
Recall that $f^*T_{X_0/T_0}\cong \cO(2)\oplus \cO(1)^{\oplus k-1}$.
If $\cO(2)\subset f^*\sF$ for general $[f]\in H$, then the general curve from $H$ is tangent 
to the foliation $\sF$. Hence the general fiber of $\pi_0$ is contained in a leaf of $\sF$. 
Since $\sF|_{X_0}\subset T_{X_0/T_0}$, we must have $\sF|_{X_0}=T_{X_0/T_0}$.
If $\cO(2)\not\subset f^*\sF$ for general $[f]\in H$, then $f^*\sF\cong \cO_{\p^k}(1)^{\oplus r}$.
Since $\sF|_{X_0}$ is saturated in $T_{X_0/T_0}$, we must have $r<k$.
Then \cite[Th\'eor\`eme 3.8]{cerveau_deserti} implies that 
the foliation induced by $\sF$
on a general fiber of $\pi_0$ is 
$\cO_{\p^k}(1)^{\oplus r}\into T_{\p^k}$.
In either case, we conclude that $\sF$ is algebraically integrable and has log canonical 
singularities along a general leaf. 
Proposition~\ref{lemma:common_point} then implies that 
there is a point $x\in X$ contained in the closure of a general leaf of $\sF$.
This is only possible if $T_0$ is a point, and
$(X,\sF)\cong \big(\p^n,\cO(1)^{\oplus r}\big)$.
\end{proof}

\begin{lemma}\label{lemma:extending_in_codim_1}
Let $X$ be a smooth projective variety, $X_0 \subset X$ a dense open subset, $T_0$ a 
positive dimensional normal variety, 
and $\pi_0:X_0 \to T_0$ a proper surjective equidimensional  morphism with 
and rationally connected general fiber. 
Let $T $ be the normalization of the closure of $T_0$ in $ \Chow(X)$, and
$U$ the universal cycle over $T$.

Through a general point of $T_0$ there 
exists a curve $C_0\subset T_0$ such that the following holds.
Let $C \to T$ be the normalization of the closure  of $C_0$ in $T$,
$U_C$ the normalization of $U \times_{T} C$, and $\pi_C: U_C \to C$ the induced morphism. Then
\begin{itemize}
\item all irreducible fibers of $\pi_C: U_C \to C$ are reduced,
\item $U_C \to X$ is finite, and
\item no fiber of $\pi_C$ is entirely mapped into the exceptional locus of the universal 
	morphism $U\to X$.
\end{itemize}
Moreover, given any subset $Z\subset T$ such that $\codim_T(Z)\geq 2$, $C_0$ can be chosen so that 
the image of $C$ in $T$ avoids $Z$.
\end{lemma}

\begin{proof} Consider the universal morphisms 

\centerline{
\xymatrix{
U \ar[d]_{\pi }\ar[r]^{e} & X. \\
T
}
}
\noindent Given $t\in T$, we write $U_t=\pi^{-1}(t)$.

Since $X$ is smooth, the exceptional locus $E$ of $e$ has pure codimension one in $U$.
Let $F$ be an irreducible component of $E$.
Since $X_0\cap e(E)=\emptyset$ and $\pi$ is equidimensional,
$\pi(F)$ has codimension one in $T$. 
Moreover, $\forall t\in T$, either $F\cap U_t=\emptyset$, or $F\cap U_t$
is a union of irreducible components of $U_t$.
So we may assume that $X_0=X\setminus e\big(\pi^{-1}(\pi(E))\big)$.
Let $E'\subset E$ be the union of irreducible components $F$ of $E$ such that 
$\pi^{-1}(\pi(F))\subset E$.
If $t \in  \pi(E) \setminus \pi(E')$, then $U_t$ has at least two irreductible components,
at least one of which is not contained in $E$.

Let $S_0\subset T_0$ be the locus over which the fibers of $\pi_0$ are multiple, and let $S$ be its closure in $T$.
By \cite{ghs03}, $\codim_{T}(S)\geq 2$.
Let $C\subset X\setminus e\big(E \cup \pi^{-1}(S\cup Z)\big)$
be a general complete intersection curve, and set 
$C_0:=\pi_0(C\cap X_0)\subset T_0\setminus S$.
Let $U_C$ be the normalization of $U \times_{T} C$, and $\pi_C: U_C \to C$ the induced morphism. 
By construction, the irreducible fibers of $\pi_C$ over $C_0$ are reduced.
Moreover, the image of $C$ in $T$ does not meet $\pi(E')$.
Hence the  fibers of $\pi_C$ over $C\setminus C_0$ have at least two irreductible components,
at least one of which is not mapped into $E$.
\end{proof}

\begin{lemma}\label{lemma:foliation_by_curve_is_algebraic}
Let $X$ be a smooth projective variety, $X_0 \subset X$ a dense open subset, $T_0$ a 
positive dimensional normal variety, 
and $\pi_0:X_0 \to T_0$ a proper surjective equidimensional morphism of relative dimension $r-1$. 
Let $\sF$ be a rank $r$ Fano foliation on $X$, and assume that there is an exact sequence
$$
0 \to T_{X_0/T_0} \to \sF|_{X_0} \to (\pi_0^*\sG_0),
$$
where $\sG_0$ is an invertible subsheaf of $T_{T_0}$.
Suppose that the general fiber $F$ of $\pi_0$ is rationally connected, and satisfies
$c_1(\sA)^{r-1}\cdot F\leq 2$ for some ample line bundle $\sA$ on $X$. 

Then $\sG_0$ defines a foliation by rational curves on $T_0$.
\end{lemma}

\begin{proof}
Let $T $ be the normalization of the closure of $T_0$ in $ \Chow(X)$, and
$U$ the normalization of the universal cycle over $T$. 
We denote by $\pi:U\to T$ and $e:U\to X$ the universal morphisms, and by
$E$ the exceptional locus of $e$.
Let $\sF_U$ be the foliation induced by $\sF$ on $U$.
Notice that $\sF_U$ is regular along the general fiber of $\pi$.
Moreover  $\sF_U$ and $e^*\sF$ agree on $U\setminus E$. 

By Lemma \ref{lemma:foliation_morphism}, there exists a smooth open subset $T_1\subset T$ with 
$\codim_T(T\setminus T_1)\geq 2$, a rank $1$ subbundle $\sG_1\subset T_{T_1}$ , and an exact sequence
$$
0 \to T_{U_1/T_1} \to \sF_{U_1} \to (\pi_1^*\sG_1),
$$
where $U_1=\pi^{-1}(T_1)$, $\sF_{U_1}=\sF_U|_{U_1}$, and $\pi_1=\pi|_{U_1}:U_1\to T_1$.
In particular, there is a canonically defined effective divisor $D_1$ on $U_1$ such that 
\begin{equation}\label{6.10.0}
\sO_{U_1}(-K_{\sF_{U_1}}) \ \simeq \ \pi_1^*\sG_1  \ \otimes \ \Big(\sO_{U_1}(-D_1) \ [\otimes] \ \det\big(T_{U_1/T_1}\big)\Big).
\end{equation} 
Moreover, since $\sF_{U_1}$ is regular along the general fiber of $\pi_1$, the divisor $D_1$ does not dominate $T_1$.

Let $C_0 \to T_0$ be the curve provided by Lemma \ref{lemma:extending_in_codim_1}, and $n:C\to T$
the normalization of its closure in $T$. We also require that $n(C)\subset T_1$.
Let $U_C$ be the normalization of $U_1 \times_{T_1} C$, and denote by $\pi_C:U_C\to C$, 
$q:U_C\to U_1$, and $e_C:U_C\to X$ the natural morphisms:
\[
\xymatrix{
U_C \ar[d]_{\pi_C}\ar[r]^{q}\ar@/^2pc/[rr]^{e_C} & U_1 \ar[d]^{\pi_1} \ar[r]^e & X. \\
C \ar[r]_n & T_1} 
\]
 By Lemma \ref{lemma:extending_in_codim_1},
\begin{enumerate}
\item[(a)] all irreducible fibers of $\pi_C: U_C \to C$ are reduced,
\item[(b)] $e_C:U_C \to X$ is finite, and
\item[(c)] no fiber of $\pi_C$ is entirely mapped into $E$ by $q$.
\end{enumerate}

Since $c_1(\sA)^{r-1}\cdot F\leq 2$, condition (a) above implies that every fiber of $\pi_C$ is reduced at all of its generic points.
Thus $\pi_C$ is smooth at the generic points of every fiber (see \cite[Chap. IV Corollaires 15.2.3 and 14.4.2]{ega28}). 
By shrinking $T_1$ if necessary, we may assume that the same holds for $\pi_1$.
By \cite[Lemme 4.4]{druel99}, 
$\det(T_{U_C/C})\simeq\sO_{U_C}(-K_{U_C/C})$, and 
$\det(T_{U_1/T_1})\simeq\sO_{U_1}(-K_{U_1/T_1})$. Thus
\begin{equation}\label{6.10.2}
\sO_{U_C}(-K_{U_C/C})\simeq q^*\det(T_{U_1/T_1}).
\end{equation}

Since $\sF_{U_1}$ and $e^*\sF$ agree on $U_1\setminus E$, there are effective divisors $\Delta_+$ and $\Delta_-$ on $U_C$,
both supported on components of fibers of $\pi_C$ that are mapped into $E$ by $q$, such that 
\begin{equation}\label{6.10.1}
q^*(-K_{\sF_{U_1}})\ =\ e_C^*(-K_{\sF}) \ + \ \Delta_+ \ - \ \Delta_-.
\end{equation} 
Condition (c) above implies that $\Delta_+$ and $\Delta_-$ are supported on reducible fibers of $\pi_C$, and no
fiber of $\pi_C$ is entirely contained in their supports. In particular, $(\pi_C)_*\sO_{U_C}(k\Delta_-)=\sO_C$ \  $\forall k\geq 0$.

By pulling back \eqref{6.10.1} to $U_C$, and combining it with \eqref{6.10.0} and \eqref{6.10.2}, 
we get that 
$$
\sO_{U_C}\big(-(K_{U_C/C}+D+  \Delta_+ -  \Delta_-)\big) \ \simeq \ \sO_{U_C}\big(e_C^*(-K_{\sF})\big) \ \otimes \ \pi_C^*(n^*\sG_1)^*,
$$
where $D$ is an effective divisor on $U_C$ that does not dominate $C$. 
By Theorem~\ref{thm:-KX/Y_not_ample}, $-(K_{U_C/C}+D+  \Delta_+ -  \Delta_-)$ is not ample. 
On the other hand, condition (b) above implies that $e_C^*(-K_{\sF})$ is ample.
So we must have $\deg_C(n^*\sG_1)>0$, and thus the general leaf of the foliation on $T_0$ defined by $\sG_0$
is a rational curve by  
Theorem~\ref{bogomolov_mcquillan}.
\end{proof}


\section{Algebraic integrability of del Pezzo foliations}\label{section:thma}

In this section we prove Theorem~\ref{thma}.
Our argument involves 
constructing subfoliations of del Pezzo foliations which 
inherit some of their positivity properties.
One way to construct such subfoliations
is via Harder-Narasimhan filtrations, as we now explain.

\begin{say}[Harder-Narasimhan filtration]
Let $X$ be an $n$-dimensional projective variety, and $\sA$ an ample
line bundle on $X$.  Let $\sF$ be a torsion-free sheaf of rank $r$ on $X$.  We
define the slope of $\sF$ with respect to $\sA$ to be
$\mu_{\sA}(\sF)=\frac{c_1(\sF)\cdot \sA^{n-1}}{r}$.  We
say that  $\sF$ is
\emph{$\mu_{\sA}$-semistable} if for any  
subsheaf $\sE$ of $\sF$ we have $\mu_{\sA}(\sE)\leq\mu_{\sA}(\sF)$.

Given a torsion-free sheaf $\sF$ on $X$, there exists a filtration
of $\sF$ by  subsheaves
$$
0=\sE_0\subsetneq \sE_1\subsetneq \ldots\subsetneq \sE_k=\sF,
$$
with $\mu_{\sA}$-semistable quotients $\cQ_i=\sE_i/\sE_{i-1}$, and
such that $\mu_{\sA}(\cQ_1) > \mu_{\sA}(\cQ_2) > \ldots >
\mu_{\sA}(\cQ_k)$.  This is called the \emph{Harder-Narasimhan
 filtration} of $\sF$ (see \cite{HN75}, \cite[1.3.4]{HuyLehn}).
\end{say}

\begin{lemma}\label{lemma:foliation_destabilizing}
Let $X$ be a normal projective variety, $\sA$ an ample line bundle on $X$,
and $\sF\subsetneq T_X$ a foliation on $X$.
Let 
$0=\sF_0 \subset \sF_1 \subset \cdots \subset \sF_k=\sF$ be the
Harder-Narasimhan filtration of $\sF$ with respect to $\sA$.
Then $\sF_i\subsetneq T_X$ defines a foliation on $X$ for every 
$i$ such that $\mu_{\sA}(\sF_i)>0$.
\end{lemma}

\begin{proof}
This is well known. See  for instance \cite[Lemma 9.1.3.1]{Shepherd-barron92}.
\end{proof}

\begin{notation}
Let $X$ be a normal projective variety,  $\sA$ an ample line bundle on $X$,
and $\sF$ a coherent torsion free sheaf of $\sO_X$-modules. 
Let $m_i\in\bN$, $1\leq i\leq \dim(X)-1$,  be large enough integers,
$H_i \in |m_i \sA|$ be general members, and set $C:=H_1\cap\cdots\cap H_{\dim(X)-1}$.
By the Mehta-Ramanathan Theorem
(see \cite[6.1]{MR_semistable_curves} or \cite[7.2.1]{HuyLehn}),
the Harder-Narasimhan filtration of $\sF$ with respect to $\sA$ commutes with
restriction to $C$.
In this case we say that $C$ is a \emph{general complete intersection curve
for $\sF$ and $\sA$ in the sense of Mehta-Ramanathan}. 
If $\sF$  and $\sA$ are clear from the context, we simply say that $C$ is a \emph{general complete intersection curve}.
\end{notation}

\begin{lemma}\label{lemma:HN_on_the_base}
Let $X$ and $Y$ be  smooth complex projective varieties with 
$\dim(Y) \ge 1$, $X_0$ an open subset of $X$ 
with $\codim_X(X\setminus X_0)\geq 2$, $Y_0$ a dense open subset of $Y$
and $\pi_0:X_0\to Y_0$
a proper surjective equidimensional morphism. 
Let $\sG$ be a coherent torsion-free sheaf of $\sO_Y$-modules on $Y$, and
$\sF$ a coherent  torsion-free sheaf of $\sO_X$-modules such that
$\sF|_{X_0}\simeq \pi_0^*\sG|_{Y_0}$.
Let $\sA$ be an ample line bundle on $X$, 
and $C\subset X_0$ a general complete intersection curve for $\sF$ and $\sA$ in the sense of
Mehta-Ramanathan. Suppose that $\sF|_{C}$ is not semistable.

Then there exists a subsheaf $\sH\subsetneq \sG$ such that
$\mu\Big(\big({\pi_0^*({\sH}|_{Y_0})}\big)|_{C}\Big)>\mu(\sF|_{C})$.
\end{lemma}

\begin{proof}
Consider  general elements $H_i \in |m_i \sA|$, for 
$i\in\{1, \ldots, \dim(X)-1\}$, where the $m_i\in\bN$ are large enough so that the 
Harder-Narasimhan filtration of $\sF$ commutes with restriction to the complete intersection curve
$C=H_1 \cap \cdots \cap H_{\dim(X)-1}$. Set $Z:= H_1 \cap \cdots \cap H_{\dim(X)-\dim(Y)}$, and $Z_0:= Z \cap X_0$. 
Then $Z$ is a smooth variety of 
dimension equal to $\dim(Y)$, and the restriction $\varphi_0 := \pi_0 |_{Z_0}:Z_0\to Y_0$ is a finite morphism.  
Denote by  $j: Z_0 \hookrightarrow Z$ the inclusion.

Let $K$ be a splitting field of the function field $K(Z_0)$ over $K(Y_0)$, and let $\psi: Z' \to Z\subset X$ be 
the normalization of $Z$ in $K$. Set $Z'_0:=\psi^{-1}(Z_0)$,  and let 
$j': Z'_0 \hookrightarrow Z'$ be the inclusion. Let $\psi_0$ be the restriction of $\psi$ to 
$Z'_0$, and set $\sF' := (\psi^* \sF)^{**} = j'_*\big(\psi_0^* \varphi_0^* (\sG|_{Y_0})\big)$. 
Let $G$ be the Galois group of $K(Z'_0)$ over $K(Y_0)$.

By \cite[Lemma 3.2.2]{HuyLehn}, $\sF'$ is not semistable with respect to $\psi^* H|_{Z}$.
Let $\sE'$ be the maximally destabilizing subsheaf of $\sF'$. Then 
$\mu_{(\psi^* \sA)}(\sE')>\mu_{(\psi^* \sA)}(\sF')$. This implies
$\mu_{(\psi^* \sA)}({\sE'}|_{C'})>\mu_{(\psi^* \sA)}({\sF'}|_{C'})$, where
$C'=C\times_Z Z'$.

Because of its uniqueness, the maximally destabilizing subsheaf $\sE'$ of $\sF'$ 
is invariant under the action of $G$ on $\sF'$. Thus, up to shrinking $Y_0$ if necessary, we may assume that 
there is a subsheaf $\sH\subset\sG$ such that $\sE'\simeq \big(j_*\psi_0^* \varphi_0^* ({\sH}|_{Y_0})\big)$.
 By \cite[Lemma 3.2.1]{HuyLehn}, 
$\mu_{(\psi^* \sA)}(\sE')=\mu_{(\psi^* \sA)} ({\sE'}|_{C'})=d \cdot \mu_{\sA} ({\sH}|_{C})$,
and
$\mu_{(\psi^* \sA)}(\sF')=\mu_{(\psi^* \sA)} ({\sF'}|_{C'})=d \cdot \mu_{\sA} ({\sG}|_{C})$,
where $d$ denotes the degree of $C'\to C$. 
\end{proof}

\begin{prop}\label{prop:foliation_max}
Let $X$ be a normal projective variety, 
$\sA$ an ample line bundle on $X$,
and $\sF\subsetneq T_X$ a foliation on $X$.
Suppose that $\mu_{\sA}(\sF)>0$, and
let $C\subset X$ be a general complete intersection curve. 
Then either 
\begin{itemize}
\item $\sF$ is algebraically integrable with rationally connected general leaves, or
\item there exits an algebraically integrable subfoliation $\sG\subsetneq \sF$
with rationally connected general leaf such that 
$\det(\sG)\cdot C \ge \det(\sF)\cdot C$. 
\end{itemize}
\end{prop}

\begin{rem}
Let $\sF$ be coherent torsion-free sheaf of $\sO_X$-modules on $X$, and $C\subset X$ a general complete intersection curve. 
Then $\sF$ is locally free along $C$, so that
$\det(\sF)\cdot C$ is well-defined.
\end{rem}

\begin{proof}[{Proof of Proposition \ref{prop:foliation_max}}]
Suppose  that $\sF$ is semistable with respect to $\sA$. Then
$\sF|_{C}$ is semistable with 
slope $\mu({\sF}|_{C})>0$, and
$\sF|_{C}$ is an ample vector bundle 
by \cite[Theorem 2.4]{hartshorne71}. 
Thus $\sF$ is algebraically integrable with rationally connected general leaf by  
Theorem~\ref{bogomolov_mcquillan}.

Suppose that $\sF$ is not semistable.
Let 
$0=\sF_0 \subset \sF_1 \subset \cdots \subset \sF_k=\sF$ be the
Harder-Narasimhan filtration of $\sF$ with respect to $\sA$. 
Note that $k\ge 2$. Set 
$i_0:=\max\{i \ge 1\, |\, \mu_{\sA}(\sF_{i}/\sF_{i-1})>0 \}$, and $\sG:=\sF_{i_0}$.
Since $\mu_{\sA}(\sG)>0$, $\sG\subsetneq T_X$ 
defines a foliation on $X$ by Lemma~\ref{lemma:foliation_destabilizing} above. 
Since $\mu_{\sA}(\sF_1) > \mu_{\sA}(\sF_2/\sF_1) > \cdots >
\mu_{\sA}(\sF_k/\sF_{k-1})$,
we must have
\begin{itemize}
\item $\mu_{\sA}(\sF_{i}/\sF_{i-1})>0$ for all $i\le i_0$, and
\item $\mu_{\sA}(\sF_i/\sF_{i-1})\le 0$ for any $i\ge i_0+1$.
\end{itemize}
Thus
\begin{enumerate}
\item ${(\sF_i/\sF_{i-1})}|_{C}$ is an ample vector bundle on $C$
for any $i\le i_0$ by \cite[Theorem 2.4]{hartshorne71}, and
\item $\det(\sF_i/\sF_{i-1})\cdot C\le 0$ for $i\ge i_0+1$.
\end{enumerate}
From (1) it follows that $\sG|_{C}$ is an ample vector bundle on $C$. 
Thus $\sG$ is algebraically integrable with rationally connected general leaf by Theorem~\ref{bogomolov_mcquillan}.
From (2) it follows that 
$$
\begin{array}{rcl}
\det(\sG)\cdot C & = & \sum_{1\le i\le i_0}\det(\sF_i/\sF_{i-1}))\cdot C\\
& = & \det(\sF)\cdot C-\sum_{i_0+1\le i\le k}\det(\sF_i/\sF_{i-1}))\cdot C\\
& \ge & \det(\sF)\cdot C.\\
\end{array}
$$
\end{proof}

These results allow us to prove Theorem~\ref{thma} in the special case
when $X$ has Picard number $1$.

\begin{prop}\label{proposition:algebraicity_picard_number_one}
Let $\sF$ be a del Pezzo foliation on a smooth projective variety $X\not\simeq \p^n$ with $\rho(X)=1$. 
Then $\sF$ is algebraically integrable with rationally connected general leaf.
\end{prop}

\begin{proof}
By Remark~\ref{uniruledness}, $X$ is uniruled.
Since $\rho(X)=1$, $X$ is in fact a Fano manifold.
Let $\sA$ be an ample line bundle  on $X$ such that $\textup{Pic}(X)=\bZ[\sA]$.
By assumption, $\det(\sF)\simeq\sA^{\otimes r-1}$, where $r=r_\sF \ge 2$.

Suppose that $\sF$ is not 
algebraically integrable with rationally connected general leaf.
By Proposition \ref{prop:foliation_max},
there exits a subfoliation
$\sG\subsetneq \sF$ 
such that
$\det(\sG)\simeq \sA^{\otimes k}$
for some  $k\ge r-1\ge r_{\sG}$.
By \cite[Theorem 1]{adk08}, $(X,\sA)\simeq (\p^n,\sO_{\p^n}(1))$.
\end{proof}

In the next two propositions we address 
algebraically integrability of del Pezzo foliations on projective space bundles.
Recall from \ref{example:ample_on_p^n} the description of degree $0$ foliations on $\p^n$,
$\sH = \sO_{\p^n}(1)^{\oplus r}\subsetneq T_{\p^n}$.
The leaves of $\sH$ are fibers of a linear projection $\p^n\map \p^{n-r}$
from an $(r-1)$-dimensional linear subspace of $\p^n$. 
We want to describe families of such foliations.

\begin{say}[Families of degree $0$ foliations on $\p^m$]\label{V_in_E}
Let $T$ be a smooth positive dimensional variety, 
$\sE$ a locally free sheaf of rank $m+1\geq 2$ on $T$, and set $X:=\p_T(\sE)$.
Denote by $\sO_X(1)$ the tautological line bundle on $X$,
by $\pi:X\to T$ the natural projection, and by $Y\simeq \p^m$ a general fiber of $\pi$.
Let $\sH\subsetneq T_{X/T}$ be a foliation of rank $s\geq 1$ on $X$, and suppose that 
$\sH|_{Y}\simeq \sO_{\p^m}(1)^{\oplus s}\subsetneq T_{\p^m}$.
We first observe that there is an open subset $T_0\subset T$, 
with $\codim_T(T\setminus T_0)\geq 2$,
such that, for any $t\in T_0$, \ $\sH|_{X_t}\simeq \sO_{\p^n}(1)^{\oplus s}\subsetneq T_{\p^m}$, where $X_t=\pi^{-1}(t)\simeq \p^m$. Indeed, there exists 
an open subset $T_0\subset T$, 
with $\codim_T(T\setminus T_0)\geq 2$ such that 
$\big(T_{X/T}\big)/\sH$ is flat over $T_0$.
Then, for any $t\in T_0$, the inclusion $\sH\subsetneq T_{X/T}$ restricts to an inclusion
$\sH|_{X_t}\subsetneq T_{\p^m}$.
In particular, $\sH|_{X_t}$ is torsion free for any $t\in T_0$.
By removing a subset of codimension $\geq 2$ in $T_0$ if necessary, 
we may assume that 
$\sO_{\p^m}(s)\simeq \big(\det (\sH)\big)|_{X_t}\simeq 
\det \big(\sH|_{X_t}\big)\subset \wedge^sT_{\p^m}$
for any $t\in T_0$.
By Bott's formulae, $\sH|_{X_t}$ is saturated in $T_{\p^m}$.
So $\sH|_{X_t}$ is a degree $0$ foliations of rank $s$ on $\p^m$, i.e.,
$\sH|_{X_t}\simeq \sO_{\p^m}(1)^{\oplus s}$.

Set $\sV':=\pi_*(\sH(-1))\subset \pi_*(T_{X/T}(-1))\simeq \sE^*$, and denote by $\sV$
the saturation of $\sV'$ in $\sE^*$. Note that
$\pi^*\sV$ is a reflexive sheaf by
\cite[Proposition 1.9]{hartshorne80}.
The above observations imply that over $T_0$ we have $\sV'=\sV$, and $\pi^*\sV\simeq \sH(-1)$.
Hence $\sH=(\pi^*\sV)(1)$. 
In particular, 
$$
\det (\sH)\ \simeq \ \pi^*(\det \sV)\otimes  \sO_X(s).
$$

Let $\sK$ be the kernel of the dual map $\sE\to \sV^*$.
By removing a subset of codimension $\geq 2$ in $T_0$ if necessary, 
we may assume that there is an exact sequence of vector bundles on $T_0$:
$$
0\ \to \ \sK|_{T_0} \ \to \ \sE|_{T_0} \ \to \ \sV^*|_{T_0} \ \to \ 0.
$$
Consider the $\p^{m-s}$-bundle $Z:=\p_{T_0}\big(\sK|_{T_0}\big)$, 
with natural projection $q:Z\to T_0$.
The above exact sequence induces a rational map 
$p:\pi^{-1}(T_0)\map Z$ over $T_0$, which restricts to a surjective morphism 
$p_0:X_0\to Z$, where $X_0$ is the complement in $\pi^{-1}(T_0)$ of the $\p^{s-1}$-subbundle 
$\p_{T_0}\big(\sV^*|_{T_0} \big)\subset \p(\sE|_{T_0})$. 
By construction, $\sH|_{X_0} = T_{X_0/Z}$.
Note also that $\codim_X(X\setminus X_0)\geq 2$. 
\end{say}

\begin{prop}\label{proposition:p^n-bdle_F_factors}
Let $C$ be a smooth complete curve, $\sE$ an ample locally free sheaf of rank $m+1\geq 3$ on 
$C$, and set $X:=\p_C(\sE)$.  Denote by $\sO_X(1)$ the tautological line bundle on $X$, and by
$\pi:X\to C$ the natural projection.
Let $\sF\subset T_{X/C}$ be a foliation of rank $r\geq 2$ on $X$ such that
$\det(\sF)\simeq\sO_X(r-1)\otimes \pi^*\sL$ for some nef line bundle $\sL$ on $C$.
Then $\sF$ is algebraically integrable with rationally connected general leaf. 
\end{prop}

\begin{proof}
Denote by $Y\simeq \p^m$ the general fiber of $\pi$.
We may assume that $\sF\subsetneq T_{X/C}$.
So the restriction of $\sF$ to $Y$ is a Fano foliation of rank $r$ and index $r-1$ on $\p^m$.
Recall from \ref{example:del_Pezzo_on_p^n} the classification of such foliations established in 
\cite{loray_pereira_touzet_2}. 
The foliation $\sF$ is algebraically integrable if and only if so is $\sF|_{Y}\subsetneq T_{Y}$.
So we may assume that 
$\sF|_{Y}$ is the pullback via a linear projection $\p^m \map \p^{m-r+1}$
of a foliation on $\p^{m-r+1}$ induced by a global holomorphic 
vector field, i.e., $\sF|_{Y}\simeq \sO_{\p^m}(1)^{r-1}\oplus\sO_{\p^m}$.

Set $\sV':=\pi_*(\sF(-1))\subset \pi_*(T_{X/C}(-1))\simeq \sE^*$, and denote by $\sV$
the saturation of $\sV'$ in $\sE^*$.  
Notice that  the inclusion $(\pi^*\sV')(1)\subset \sF$ extends to an inclusion 
$\sH:=(\pi^*\sV)(1)\subset \sF\subset T_{X/C}$.
So $\sH$ is a subfoliation of $\sF$ of rank $r-1$ such that $\sH|_{Y}\simeq  \sO_{\p^n}(1)^{r-1}$.

Let the notation be as in \ref{V_in_E}, with $T=T_0=C$ and $s=r-1$.
There is a surjective morphism  $p_0:X_0\to Z$
such that  $\sH|_{X_0} = T_{X_0/Z}$, and 
$\det (\sH) \simeq  \pi^*\big(\det (\sV)\big)\otimes  \sO_X(r-1)$.

By~Lemma \ref{lemma:foliation_morphism}, $\sF$ induces 
a rank 1 foliation $\sG\subset T_Z$ on $Z$ such that 
$\det (\sF|_{X_0}) \simeq \det \big(T_{X_0/Z}\big)\otimes \ p_0^*\sG$.
(Notice that $\sG$ is an invertible sheaf by \cite[Proposition 1.9]{hartshorne80}.)
Recall that $\codim_X(X\setminus X_0)\geq 2$.
So 
$\sG\simeq q^*\big(\det (\sV^*)\otimes \sL\big)$, and $\det (\sV^*)$ is ample since so is $\sE$.
Let $B\subset Z$ be a general complete intersection curve.  
Then $\sG|_{B}$ is an ample line bundle.
Thus $\sG$ is a foliation by rational curves by Theorem~\ref{bogomolov_mcquillan}.
The general leaf of $\sF$ is the closure of the inverse image by $p_0$ of a general leaf
of $\sG$. Hence it is algebraic and rationally connected.
\end{proof}

\begin{prop}\label{prop:classification_bundle}
Let $\sE$ be an ample locally free sheaf of rank $m+1\geq 2$ on $\p^l$, 
and set $X:=\p_{\p^l}(\sE)$.
Denote by $\sO_X(1)$ the tautological line bundle on $X$, 
$\pi:X\to \p^l$ the natural projection, and $Y\simeq \p^m$ the general fiber of $\pi$.
Let $\sF\subsetneq T_{X}$ be a foliation of rank $r\geq 2$ on $X$ such that
$\det(\sF)\simeq\sO_X(r-1)\otimes \pi^*\sL$ for some nef line bundle $\sL$ 
on $\p^l$.
Suppose that $\sF\nsubseteq T_{X/\p^l}$, and set $\sH:=\sF\cap T_{X/\p^l}$.
Then
\begin{enumerate}
	\item $\sF$ is algebraically integrable with rationally connected general leaf;
	\item $r\in \{2,3\}$, and $r=3$ implies $l=1$;
	\item $\sL\simeq\sO_{\p^l}$ unless $r=2$, $l=1$, and $\sL\simeq\sO_{\p^1}(1)$;
	\item if $m\geq 2$, then $\sH|_{Y}\simeq \sO_{\p^m}(1)^{\oplus r-1}\subsetneq T_{\p^m}$, and 
		there is an exact sequence 
		$$
		0\ \to \ \sH \ \to \ \sF \ \to \ \pi^*\sR,
		$$
		where $\sR\subset T_{\p^l}$ is an ample invertible subsheaf.
	\item if $m=1$, then $l\geq r=3$, 
		$X\simeq \p^1\times \p^l$, $\sH=T_{X/\p^l}$, and $\sF$ 
		is the pullback via the natural projection $\p^1\times \p^l\to \p^l$ 
		of a degree zero foliation $\sO_{\p^l}(1)\oplus\sO_{\p^l}(1)\subsetneq T_{\p^l}$ on $\p^l$. 
\end{enumerate}
\end{prop}

\begin{proof}
Denote by $\ell\subset Y\simeq \p^m$ a general line.
Set  $\sQ:=\sF/\sH\subseteq \pi^*T_{\p^l}$. 
Notice that $\sH$ is saturated in $T_X$, and stable under the Lie bracket.
So it defines a foliation of rank $r_{\sH}<r$ on $X$. 
Note also that $\sQ$ is torsion-free.
The sheaves $\sF$, $\sH$ and $\sQ$ are locally free in a neighborhood of $\ell$, and
we have an exact sequence of vector bundles
$$
0 \ \to \ \sH|_{\ell} \ \to \ \sF|_{\ell} \ \to \ {\sQ}|_{\ell} \ \to \ 0.
$$ 

Notice that 
${\sQ}|_{\ell}\subset ({\pi^*T_{\p^l}})|_{\ell}\simeq\sO_{\p^1}^{\oplus l}$, and
$\det(\sH|_{\ell})\simeq  \sO_{\p^1}(r-1) \otimes \det(\sQ|_\ell)^* 
\subset \wedge^{r_{\sH}}(T_{\p^m}|_{\ell})$.
So  $\deg\big(\det(\sQ)|_{\ell}\big)\in \{-1, 0\}$.
We claim that $\det(\sQ)|_{\ell}\simeq \sO_{\p^1}$. 
Suppose to the contrary  that $\det(\sQ)|_{\ell}\simeq \sO_{\p^1}(-1)$.
Then  $r_{\sH}=r-1=m$, \ $\sH=T_{X/\p^l}$, and 
$\sF|_{\ell}\simeq \sO_{\p^1}(2)\oplus \sO_{\p^1}(1)^{\oplus r-2}\oplus
\sO_{\p^1}(-1)$.
By lemma \ref{lemma:foliation_morphism}, 
$\sF$ induces a rank 1 foliation on $\p^l$.
Let $C_0 \subset \p^l$ be the germ of a (complex analytic) leaf of this foliation.
Then 
$\pi_0^{-1}(C_0)$ is a germ of a  leaf of $\sF$, and thus 
$f^*\sF\simeq \sO_{\p^1}(2)\oplus \sO_{\p^1}(1)^{\oplus r-2}\oplus
\sO_{\p^1}$, a contradiction.
This proves that $\det(\sQ)|_{\ell}\simeq \sO_{\p^1}$. 
Since $\det(\sH|_{\ell})\simeq  \sO_{\p^1}(r-1) \subset \wedge^{r_{\sH}}(T_{\p^m}|_{\ell})$
and $r_{\sH}<r$, one of the following occurs.
\begin{enumerate}
\item[(a)] 
$m\geq r$, $\sH|_{\ell}\simeq  \sO_{\p^1}(1)^{\oplus r-1}$, and ${\sQ}|_{\ell}\simeq\sO_{\p^1}$; or
\item[(b)] 
$m=r-2$,  $\sH=T_{X/\p^l}$, and ${\sQ}|_{\ell}\simeq\sO_{\p^1}\oplus \sO_{\p^1}$.
\end{enumerate}

\medskip

First we treat case (a).
By generic flatness, we have $\sH|_{Y}\subsetneq T_{\p^m}$.
Notice that $\sH|_{Y}$ is closed under the Lie bracket, and 
$\det \big(\sH|_{Y}\big)\simeq \sO_{\p^m}(r-1)\subset \wedge^{r-1}T_{\p^m}$. 
By Bott's formulae, $\sH|_{Y}$ is saturated in $T_{\p^m}$.
So $\sH|_{Y}$ is a degree $0$ foliations of rank $r-1$ on $\p^m$, i.e.,
$\sH|_{Y}\simeq \sO_{\p^m}(1)^{\oplus r-1}\subsetneq T_{\p^m}$.

By \cite[Proposition 1.9]{hartshorne80}, $\sQ^{**}$ is locally free,
and thus $\sQ^{**}\simeq \pi^*\sR$  
for some invertible subsheaf $\sR\subset T_{\p^l}$. 
We have $\sO_X(r-1)\otimes \pi^*\sL\simeq \det(\sF)\simeq \det(\sH)\otimes \pi^*\sR$.

Let the notation be as in \ref{V_in_E}, with $T=\p^l$ and $s=r-1$. 
So we have a a surjective morphism  $p_0:X_0\to Z$
such that  $\sH|_{X_0} = T_{X_0/Z}$, and 
$\det (\sH) \simeq  \pi^*\big(\det (\sV)\big)\otimes  \sO_X(r-1)$.
Hence $\sR\simeq \det (\sV^*)\otimes \sL \subset T_{\p^l}$. 
Recall from \ref{V_in_E} the exact sequence of vector bundles on $T_0$:
$$
0\ \to \ \sK|_{T_0} \ \to \ \sE|_{T_0} \ \to \ \sV^*|_{T_0} \ \to \ 0,
$$
where $T_0\subset \p^l$ is an open subset such that $\codim_{\p^l}(\p^l\setminus T_0)\geq 2$.
Since $\sE$ is ample, and $\sV$ has rank $r-1$, we must have 
$\det (\sV^*)\simeq  \sO_{\p^l}(k)$ for some $k\geq r-1$. 
By Bott's formulae,  $r\le 3$. 
Moreover, if $l\geq 2$, then $r=2$ and $\sL\simeq \sO_{\p^l}$.
If $r=3$, then $l=1$ and $\sL\simeq \sO_{\p^1}$.
If $r=2$ and $l=1$, then $\sL\simeq \sO_{\p^1}(k)$ for some $k\in \{1,2\}$.

By~Lemma \ref{lemma:foliation_morphism}, $\sF$ induces 
a rank 1 foliation $\sG\subset T_Z$ on $Z$ such that 
$\det (\sF|_{X_0}) \simeq \det \big(T_{X_0/Z}\big)\otimes \ p_0^*\sG$.
Notice that $\sG$ is an invertible sheaf by \cite[Proposition 1.9]{hartshorne80}.
Recall that $\codim_X(X\setminus X_0)\geq 2$.
So we have 
$\sG\simeq q^*\big(\det (\sV^*)\otimes \sL\big)$, and $\det (\sV^*)$ is ample since so is $\sE$.
Recall also that $q:Z\to T_0$ is a  $\p^{m-r+1}$-bundle.
So through a general point of $Z$ there exists a complete curve $B$  not contracted by $q$, 
and avoiding the singular locus of the foliation $\sG\subset T_Z$.
The restriction $\sG|_{B}$ is an ample line bundle.
Thus $\sG$ is a foliation by rational curves by Theorem~\ref{bogomolov_mcquillan}.
The general leaf of $\sF$ is the closure of the inverse image by $p_0$ of a general leaf
of $\sG$. Hence it is algebraic and rationally connected.

\medskip

Next we consider case (b):  $\sH=T_{X/\p^l}$, $l\geq 3$, and 
${\sQ}|_{\ell}\simeq\sO_{\p^1}\oplus \sO_{\p^1}$.
By~Lemma \ref{lemma:foliation_morphism}, $\sF$ is the pullback via 
$\pi$ of a rank 2 foliation $\sG\subset T_{\p^l}$.
In particular, $\det(\sF)\simeq \det(T_{X/\p^l})\otimes \pi^*\det(\sG)$.
Hence $\det(\sG)\simeq \det(\sE)\otimes \sL\subset \wedge^2T_{\p^l}$.
By assumption, $\sL$ is nef and $\sE$ is ample of rank $\geq 2$.
So, by Bott's formulae, we must have 
$\det(\sQ)\simeq\sO_{\p^l}(2)$, $\textup{rank}(\sE)=2$
and
$\sL\simeq \sO_{\p^l}$.
Since $\sE$ is an ample vector bundle, $\sE|_{\p^1}\simeq\sO_{\p^1}(1)^{\oplus 2}$ for any line $\p^1\subset\p^l$.
By \cite[Theorem 3.2.1]{OSS},  
$\sE\simeq \sO_{\p^l}(1)\oplus \sO_{\p^l}(1)$. 
Thus $X\simeq \p^1\times \p^l$, and $\sG$ is a degree zero foliation on $\p^l$.
The leaves of $\sG$ are $2$-planes containing a fixed line in $\p^l$.
Hence the leaves of $\sF$ are algebraic and rationally connected.
\end{proof}

\begin{rem}\label{rmk:log_leaf_p^m_over_p^l}
Let the notation and assumptions be as in Proposition~\ref{prop:classification_bundle},
and denote by $(\tilde F, \tilde \Delta)$ the general log leaf of $\sF$.
If $m\geq 2$, then $\pi$ induces a $\p^{r-1}$-bundle structure $\pi_{\tilde F}:\tilde F\to \p^1$.
If $l=1$, then then $\tilde \Delta$ is a prime divisor of $\pi_{\tilde F}$-relative degree $1$.
If $l>1$ (in which case $r=2$), 
then $\tilde \Delta$ is the union of a prime divisor of $\pi_{\tilde F}$-relative 
degree $1$ and a fiber of $\pi_{\tilde F}$.
If $m=1$, $\pi$ induces a $\p^{1}$-bundle structure $\pi_{\tilde F}:\tilde F\to \p^2$,
and $\tilde \Delta$ is a fiber of $\pi_{\tilde F}$.
Notice that in all cases $(\tilde F, \tilde \Delta)$ is log canonical.
\end{rem}

\begin{rem}
In Section~\ref{section:examples_pn_bundles} we will classify locally free sheaves $\sE$
on $\p^l$ for which $X=\p(\sE)$ admits a del Pezzo foliation $\sF\nsubseteq T_{X/\p^l}$.
Moreover, we will give a precise geometric description of such foliations.
\end{rem}

Now we consider another special case of Theorem~\ref{thma}.

\begin{prop}\label{proposition:p^m_bundle_over_p^n}
Let $\sF$ be a del Pezzo foliation of rank $r\geq 2$ on a smooth projective variety $X$.
Let $H$ be minimal dominating family of rational curves on $X$, with associated 
rationally connected quotient $\pi_0:X_0 \to T_0$.
Suppose that $\sF|_{X_0}\not\subset T_{X_0/T_0}$, and 
$f^*\sF\simeq \sO_{\p^1}(1)^{\oplus r-1}\oplus \sO_{\p^1}$ for a general member
$[f]\in H$. 

Then there are integers $l\geq 1$, $m\geq 2$, and an ample locally free sheaf $\sE$ 
on $\p^l$ such that $X\simeq \p_{\p^l}(\sE)$.
Moreover, under this isomorphism, $\pi_0$ becomes the natural projection
$\p_{\p^l}(\sE)\to \p^l$, and $\det(\sF)\simeq\sO_X(r-1)$, where 
$\sO_X(1)$ is the tautological line bundle on $\p_{\p^l}(\sE)$.

In particular, $\sF$ is algebraically integrable with rationally connected general leaf
by Proposition~\ref{prop:classification_bundle}.
\end{prop}

\begin{proof}
Write $\det(\sF)=\sA^{\otimes r-1}$ for an ample line bundle $\sA$ on $X$. Then 
$f^*\sA\simeq \sO_{\p^1}(1)$ for any $[f]\in H$, which implies that $H$ is unsplit.
By \cite[Lemma 2.2]{adk08}, we may assume that 
$\textup{codim}_X(X\setminus X_0)\geq 2$, 
$T_0$ is smooth, and $\pi_0$ is proper, surjective, equidimensional, and has
irreducible and reduced fibers.
We denote by $m$ the relative dimension of $\pi_0$, and set $l:=\dim T_0$.

First we show that  $\pi_0$ is a $\p^m$-bundle.
Set $\sH_0:=\sF|_{X_0}\cap T_{X_0/T_0}=\ker(\sF|_{X_0} \to \pi_0^*T_{T_0})$.
Let $[f]\in H$ be a general member.
By assumption, $f^*\sF\simeq \sO_{\p^1}(1)^{\oplus r-1}\oplus \sO_{\p^1}$ and
$\sF|_{X_0}\not\subset T_{X_0/T_0}$.
Moreover $f^*T_{T_0}\simeq \sO_{\p^1}^{\oplus \dim(T_0)}$.
So we conclude that $\sH_0$ has rank equal to $r-1<m$,
and $f^*\sH_0\simeq \sO_{\p^1}(1)^{\oplus r-1}$.
Thus $\pi_0$ is a $\p^m$-bundle by \cite[Proposition 2.7]{adk08}.
Denote by $Y \simeq \p^m$ a general fiber of $\pi_0$.

Let $\sH\subset\sF$ be a saturated subsheaf extending $\sH_0\subset\sF|_{X_0}$, and
set  $\sQ:=(\sF/\sH)^{**}$. 
Then $\sH$ is reflexive, 
$\sQ$ is locally free of rank one by \cite[Proposition 1.9]{hartshorne80}, and 
$\det(\sF)\simeq \det(\sH)\otimes \sQ$.
Moreover, $\sQ|_Y\simeq \sO_{\p^m}$. 
Therefore there exists an invertible subsheaf $\sG_0\subset T_{T_0}$ such that
$\sQ|_{X_0}= \pi_0^*\sG_0$,
and an inclusion 
\begin{equation}\label{7.8.1}
(\sA|_{X_0})^{\otimes r-1}\ \simeq \ \det (\sF|_{X_0})\  \into \ \wedge^{r-1}T_{X_0/T_0}\otimes \pi_0^*\sG_0 \ . 
\end{equation}

The next step is to show that $\sG_0$ induces a foliation by rational curves on $T_0$.
For this purpose, let $B\subset X_0$ be a general smooth complete curve, set 
$\sG_B:=(\pi_0|_B)^*\sG_0$,  $X_B:=X_0 \times_{T_0} B$, and consider the induced 
$\p^m$-bundle $\pi_B:X_B\to B$.
Denote by $\sA_{X_B}$ the ample line bundle on $X_B$ obtained by pulling back $\sA$ from $X$.
Then \eqref{7.8.1} yields an inclusion 
$$
\sA^{\otimes r-1} \otimes \pi_B^*(\sG_B^*) \ \subset \ \wedge^{r-1}T_{X_B/B} \ .
$$
It follows from \cite[Lemma 5.2]{adk08} that $\deg_B(\sG_B) > 0$.
By Theorem~\ref{bogomolov_mcquillan},
this implies that the general leaf of the foliation induced by $\sG_0\subset T_{T_0}$ is a rational curve.

Let $C\simeq \p^1$ be a smooth compactification of a general leaf $C_0$ of 
the foliation induced by $\sG_0\subset T_{T_0}$, and let $X_C$ be the 
normalization of the closure  of $\pi_0^{-1}(C_0)$ in $X$, with
induced morphism $\pi_C:X_C\to C$. 
Denote by $\sA_{X_C}$ 
the pullback of $\sA$ to $X_C$. 
Every fiber of $\pi_C$ is generically reduced and irreducible since it has degree one with respect to the ample line bundle $\sA_{X_C}$.
Since $C$ is smooth, $\pi_C$ is flat, and $X_C$ is normal, every fiber satisfy Serre's
condition $\textup{S}_1$, and hence it is integral.
Therefore $\pi_C:X_C\to C\simeq \p^1$ is a $\p^m$-bundle by \cite[Corollary 5.4]{fujita75}.
Notice that the image of $X_C$ in $X$ is invariant under $\sF$.
Thus, by Lemma~\ref{lemma:extensionpfafffields},  
$\sF$ induces a foliation $\sF_{X_C}$ of rank $r$ on $X_C$ such that
$\det(\sF_{X_C})\simeq (\sA_{X_C})^{\otimes r-1}\otimes \pi_C^*\sL$
for some nef line bundle $\sL$ on $C$.
It follows from Proposition~\ref{prop:classification_bundle} that 
$r\in\{2,3\}$, and $\sF_{X_C}$ is algebraically integrable
with rationally connected general leaf. Hence the same holds for $\sF$.

Let $(\tilde F,\tilde \Delta)$ be a general log leaf of $\sF$.
Denote by $\tilde e:\tilde F\to X$ the natural  morphism, and 
by $\sA_{\tilde F}$ the pullback of $\sA$ to $\tilde F$.
Recall the formula:
\begin{equation}\notag
(\sA_{\tilde F}^{\oplus r-1}) \otimes \sO_{\tilde F}( \tilde \Delta) \ \simeq \  
\sO_{\tilde F}\big(\tilde e^*(-K_{\sF}) + \tilde \Delta\big) \ = \  \sO_{\tilde F}(-K_{\tilde F}).
\end{equation}

Notice that $\tilde F$ is also the normalization of a general leaf of $\sF_{X_C}$
for some $C$ as above. 
By Remark~\ref{rmk:log_leaf_p^m_over_p^l}, 
$\pi_C$ induces a $\p^{r-1}$-bundle structure $\pi_{\tilde F}:\tilde F\to \p^1$.
So we can write 
$\tilde F \simeq \p_{\p^1}\big((\pi_{\tilde F})_*\sA_{\tilde F}\big)$, and 
$(\pi_{\tilde F})_*\sA_{\tilde F}\simeq \sO_{\p^1}(a_1)\oplus\cdots \oplus\sO_{\p^1}(a_{r})$,
with $1\leq a_1\leq \cdots \leq a_r$. 
Then $\sO_{\tilde F}(-K_{\tilde F})\simeq \pi_{\tilde F}^*\sO_{\p^1}(-a_1-\cdots - a_{r}+2)
\otimes \sA_{\tilde F}^{\otimes r}$.
Substituting this in the formula above, we get:
$$
\tilde \Delta \in \ \big| \pi_{\tilde F}^*\sO_{\p^1}(-a_1-\cdots - a_{r}+2) \otimes \sA_{\tilde F}\big|.
$$
In particular, we see that $\tilde \Delta$ contains a unique irreducible component that dominates 
$\p^1$ under $\pi_{\tilde F}$, which we denote by $\tilde \sigma$. Moreover, the restriction of 
$\pi_{\tilde F}$ to $\tilde \sigma$ makes it a $\p^{r-2}$-bundle over $\p^1$.

Let $\sigma':\p^1\to \tilde F$ be the section 
of $ \pi_{\tilde F}$ corresponding to a general surjection $\sE\twoheadrightarrow \sO_{\p^1}(a_r)$, and 
set $C'=\sigma'(\p^1)\subset \tilde F$.
Then $C'$ is a moving curve on $\tilde F$, and thus
$$
0 \ \leq \ \tilde \Delta\cdot C' \ = \ -a_1-\cdots - a_{r-1}+2 \ \leq \ 3-r.
$$
If $r=3$, then $\tilde \Delta\cdot C'=0$, which implies that $\tilde \Delta$ does not contain any
fiber of $\pi_{\tilde F}$ as irreducible component, i.e., $\tilde \Delta = \tilde \sigma$.
Similarly, if $r=2$, then either $\tilde \Delta = \tilde \sigma$, or $\tilde \Delta = \tilde \sigma+\tilde f$,
where $\tilde f$ is a fiber of $\pi_{\tilde F}$.
In any case, we see that $(\tilde F, \tilde \Delta)$ is log canonical. 
Therefore, by Proposition~\ref{lemma:common_point}, there is a point $x_0\in X$ contained 
in the closure of a general leaf of $\sF$.

Suppose that $\tilde \Delta = \tilde \sigma$. We will show that $l=1$.
Let $T$ be the normalization of the closure of $T_0$ in $\textup{Chow}(X)$, with universal family 
morphisms $\pi:U\to T$ and $e:U\to X$. 
We denote by $\sF_U$ the foliation on $U$ induced by $\sF$, and by $\sG_T$ the foliation on $T$ induced by $\sG_0$.
Consider the commutative diagram:
\[
\xymatrix{
\tilde F \ar[dr]_{\pi_{\tilde F}}\ar[r] \ar@/^2pc/[rrr]^{\tilde e} & X_C \ar[d]^{\pi_C}\ar[r] & U \ar[d]^{\pi} \ar[r]^e & X. \\
& C \ar[r] & T} 
\]
Let $E\subset U$ be the exceptional locus of $e$. Since $X$ is smooth, $E$ has pure codimension one in $U$.
Since $c_1(\sA)^m\cdot Y=1$, the fibers of $\pi$ are irreducible, and thus $E$
is a union of fibers of $\pi$.

We claim that the image of $\tilde F$ in $U$ does not meet $E$. This implies that $e$ is an isomorphism 
over a neighborhood of $x_0$.
Suppose otherwise that there is a point $c\in C$ that is mapped into $\pi(E)\subset T$.
Set $\tilde f:=\pi_{\tilde F}^{-1}(c)\subset \tilde F$, and denote by $x\in e(E)\subset X$ the
image of a general point of $\tilde f$. 
Note that $e^{-1}(x)$ is positive dimensional, while its intersection with the image of $\tilde F$ 
in $U$ is zero-dimensional.  
Hence there is a positive dimensional family of general leaves of $\sF_U$ meeting 
$e^{-1}(x)$, yielding a positive dimensional family of general leaves of $\sF$ passing through $x$.
This shows that $\tilde f\subset \tilde \Delta$, contradicting the assumption that 
$\tilde \Delta = \tilde \sigma$, and proving the claim.

Since $e$ is an isomorphism over a neighborhood of $x_0$, and the general leaf of $\sF$ 
contains $x_0$, we conclude that the general leaf of $\sG_T$  contains the point
$t_0=\pi\big(e^{-1}(x_0)\big)\in T$. 
Let $u\in \pi^{-1}(t_0)$ be a general point, and let $c\in C$ be a point mapped to $t_0$.
Then there is a leaf of $\sF_{X_C}$  whose image in $U$ contains the point $u$.
Thus, if $l=\dim T>1$, then we can find a positive dimensional family of general leaves of $\sF_U$
containing $u$. 
Since $e$ is birational at $u$, we conclude that $u$ lies in the singular locus of $\sF$. 
Thus $e\big(\pi^{-1}(t_0)\big)$ is contained in the singular locus of $\sF$, 
and $\tilde \Delta$ contains a fiber of
$\pi_{\tilde F}$, contradicting our assumptions.
We have just proved that if $\tilde \Delta = \tilde \sigma$, then $l=1$, and thus 
$X=X_C\simeq \p_{\p^1}\big(\pi_*\sA\big)$.

From now on suppose that $r=2$ and $\tilde \Delta = \tilde \sigma+\tilde f$.
In particular we must have $l>1$.
We must show that in this case $\pi_0$ extends to  a $\p^m$-bundle $\pi: X\to \p^l$.
Let $H'\subset \RatCurves(X)$ be a family that parametrizes (among possibly other curves) the image
of $\tilde \sigma$ in $X$. Since $\sA_{\tilde F}\cdot \tilde \sigma = 1$, $H'$ is unsplit.
Notice that $H$ and $H'$ are numerically independent in $N_1(X)$, and $X$ is $(H,H')$-rationally
connected. 
The latter is because $\tilde F$ is itself rationally connected with respect to families obtained
from restriction of  $H$ and $H'$, and there is a point $x_0\in X$ contained 
in the closure of a general leaf of $\sF$.
It follows from \cite[IV.3.13.3]{kollar96} that  $\rho(X) = 2$.

Next we show that $H$ generates an extremal ray of  the Mori cone $\NE(X)$
(see \cite{kollar_mori} for the definition and properties of the Mori cone).
First we claim that the common point $x_0$ lies in the image of $\tilde \sigma$ in $X$.
In any case, $x_0\in \tilde e(\tilde \Delta)$ by Lemma~\ref{lemma:singular_locus_normalization}.
Notice that $\sF_{X_C}$ is regular at the generic point of any fiber  of $\pi_C$.
Thus the image of the singular locus of $\sF_{X_C}$ in $\tilde F$ is $\tilde \sigma$.
Hence, given any point in the image of $\tilde F\setminus \tilde \sigma$ in $X_C$, there is a leaf of $\sF_{X_C}$
that does not pass through this point.
So we must have 
$x_0\in \ell':=\tilde e(\tilde \sigma)$.
Now let $Z\subset X$ be the closure of the union of the curves $ \ell'$
when $F$ 
runs through general leaves of $\sF$. 
Then $Z$ is irreducible, it dominates $T_0$, and $\dim Z=\dim T_0$. 
By \cite[IV.3.13.3]{kollar96}, $N_1(Z)$ is generated by $[\ell']$.
By construction, a general point of $X$ can be connected to $Z$ by a curve from $H$.
Since $H$ is unsplit, this is a closed condition, and it holds for every point of $X$. 
It follows from \cite[(Proof of) Lemma 1.4.5]{beltrametti_sommese_wisniewski92} 
(see also \cite[Remark 3.3]{occhetta06}) that any curve on $X$ is numerically equivalent to
a linear combination $\lambda \ell'+\mu \ell$, where $\lambda\geq 0$ and $\ell$ is a 
curve parametrized by $H$.
This implies that $[\ell]$ generates an extremal ray of $\NE(X)$.
Indeed, suppose $[\ell] = \alpha_1 + \alpha_2$, with $\alpha_1, \alpha_2\in \NE(X)\setminus \{0\}$.
Write $\alpha_i=\lambda_i [\ell']+\mu_i [\ell]$, with $\lambda_i\geq 0$. Then
$(1-\mu_1-\mu_2)\ell \equiv (\lambda_1+\lambda_2)\ell'$. 
Since $\ell$ and $\ell'$ are numerically independent, we must have $\lambda_1=\lambda_2=0$.
Thus $[\ell]$ generates an extremal ray of $\NE(X)$, and $\pi_0$ extends 
to a morphism $\bar\pi:X\to W$, namely the contraction of the extremal ray generated by $H$.

We claim that the morphism $\bar\pi:X\to W$ is equidimensional. 
Let $X_t$ be  a component of a fiber of $\bar\pi$. 
Since $\bar\pi$ is the contraction of the extremal rays generated by $H$, 
$N_1(X_t)$ is generated by classes  of curves from the family $H$.
Moreover, since any point of $X$ can be connected to $Z$ by a curve from $H$, $X_t$ meets $Z$. 
On the other hand, $N_1(Z)$ is generated by $[\ell']$.   
Thus $Z\cap X_t$ must be $0$-dimensional. 
We conclude from these observations that $\dim X_t +\dim Z=\dim X$, and $\bar\pi$ is equidimensional.

By \cite[Lemma 2.12]{fujita87}, $W$ is smooth and $\pi$ is a $\p^m$-bundle. 
Recall from the beginning of the proof that there exists a 
line bundle $\sG \subset T_W$ on $W$ such that  $\sG \cdot B>0$ 
for a  curve  $B\subset W$. Since $\rho(W)=1$, it follows that $\sG$ is ample.
By \cite{wahl83}, since $l>1$, we must have $(W,\sG) \simeq (\p^l,\sO_{\p^l}(1))$. 
\end{proof}

We end this section by proving Theorem~\ref{thma}.

\begin{proof}[{Proof of Theorem~\ref{thma}}]

Write $\det(\sF)=\sA^{\otimes r-1}$ for an ample line bundle $\sA$ on $X$. 
Set $r:=r_\sF \ge 2$ and $n:=\dim(X)\geq 3$. The proof is by induction on $n$.

By Remark~\ref{uniruledness}, we know that $X$ is uniruled.
Fix a minimal dominating family $H$ of rational curves on $X$, 
and let $\pi_0:X_0 \to T_0$ be the $H$-rationally connected quotient of $X$.

\begin{step} 
Let $[f]\in H$ be a general member.
We show that one of the following holds.
\begin{enumerate}
	\item $f^*\sF\simeq \sO_{\p^1}(1)^{\oplus r-1}\oplus
	\sO_{\p^1}$, and $H$ is unsplit, or
	
	\item $f^*\sF\simeq \sO_{\p^1}(2)\oplus
	\sO_{\p^1}$ ($r=2$), or
	
	\item $f^*\sF\simeq \sO_{\p^1}(2)\oplus \sO_{\p^1}(1)^{\oplus r-3}\oplus
	\sO_{\p^1}^{\oplus 2}$ ($r \ge 3$), and $H$ is unsplit.
\end{enumerate}

\medskip

Write $f^*\sF\simeq \sO_{\p^1}(a_1)\oplus\cdots\oplus \sO_{\p^1}(a_r)$, where
$a_1 \leq \cdots \leq a_r$ and $a_1+\cdots+a_r=(r-1)\sA\cdot\ell$.
By \cite[IV.2.9]{kollar96}, $f^*T_X\simeq \sO_{\p^1}(2)\oplus \sO_{\p^1}(1)^{\oplus d}\oplus
\sO_{\p^1}^{\oplus (n-d-1)}$ where $d=\deg(f^*T_X)-2\geq 0$. 
By Lemma~\ref{lemma:Fano_ample}, $f^*\sF$ cannot be ample.
Therefore, either the $a_i$'s are as in one of the three cases listed above, or 
$f^*\sF\simeq \sO_{\p^1}(2)\oplus \sO_{\p^1}(1)^{\oplus r-2}\oplus
\sO_{\p^1}(-1)$.
Moreover, in the latter case and in cases (1) or (3) above, we have $\sA\cdot \ell=1$, 
and  hence $H$ is unsplit.

We claim that $f^*\sF\not \simeq \sO_{\p^1}(2)\oplus \sO_{\p^1}(1)^{\oplus r-2}\oplus
\sO_{\p^1}(-1)$.
Suppose otherwise.
Then $T_{\p^1} \subset f^*\sF$ for general $[f] \in H$, and
Lemma~\ref{lemma:foliation_rc_quotient}
implies that  $T_{X_0/T_0} \subsetneq \sF|_{X_0}$.
Hence $f^*T_{X_0/Y_0}\simeq \sO_{\p^1}(2)\oplus \sO_{\p^1}(1)^{\oplus r-2}$.
By \cite{araujo06}, $\pi_0$ is a projective space bundle, and 
we may assume that $\textup{codim}_X(X\setminus X_0)\geq 2$.
By lemma \ref{lemma:foliation_morphism}, $\sF$ induces a foliation by curves on $T_0$. For a general point 
$t_0$ in $T_0$, let $C_0 \subset T_0$ be the germ of the leaf (in the complex analytic topology) through $t_0$. Then 
$\pi_0^{-1}(C_0)$ is the germ of a leaf of $\sF$, and thus 
$f^*\sF\simeq \sO_{\p^1}(2)\oplus \sO_{\p^1}(1)^{\oplus r-2}\oplus
\sO_{\p^1}$, a contradiction.
\end{step}

\begin{step} We treat case (1): $f^*\sF\simeq \sO_{\p^1}(1)^{\oplus r-1}\oplus \sO_{\p^1}$ and $H$ 
is unsplit. 

\medskip

By \cite[Lemma 2.2]{adk08}, we may assume that 
$\textup{codim}_X(X\setminus X_0)\geq 2$, 
$T_0$ is smooth, and $\pi_0$ is proper, surjective, equidimensional, and has
irreducible and reduced fibers.

If $\dim(T_0)=0$, then, since the family $H$ is unsplit, $\rho(X)=1$, and the result follows from 
Proposition~\ref{proposition:algebraicity_picard_number_one}.
So we may assume  that $\dim(T_0) \ge 1$.
By Proposition~\ref{proposition:p^m_bundle_over_p^n}, we may assume also that
$\sF|_{X_0}\subsetneq T_{X_0/T_0}$.
Denote by $Y$ a general fiber of $\pi_0$, and set $m=\dim Y$. 
Then $\sF|_Y\subset T_Y$ is a Del Pezzo foliation on the 
smooth projective variety $Y$. 
If $Y\not\simeq \p^m$, then the result follows by induction on the dimension. 
So we may assume that $Y\simeq \p^m$.
Since $\sA|_Y\simeq \sO_{\p^m}(1)$, $\pi_0$ is a $\p^m$-bundle by \cite[Corollary 5.4]{fujita75}.
By removing a subset of codimension $\geq 2$ in $T_0$ if necessary, we may assume that $\sF|_{X_0}$ is a subbundle of $T_{X_0/T_0}$.

Let $B\subset X_0$ be a general smooth complete curve, set 
$X_B:=X_0 \times_{T_0} B$, and consider the induced 
$\p^m$-bundle $\pi_B:X_B\to B$.
Denote by $\sA_{X_B}$ 
the pullback from $X$ of the line bundle $\sA$, 
and by $\sF_{X_B}\subsetneq T_{X_B/B}$ the pullback of $\sF$.  
Then $\det(\sF_{X_B})=\sA_{X_B}^{\otimes r-1}$.
Thus $\sF_{X_B}$ is algebraically integrable and has rationally connected general leaf by 
Proposition~\ref{proposition:p^n-bdle_F_factors}.
Since $B$ is general, the same holds for $\sF$.
\end{step}

\begin{step} We treat case (2): $r=2$ and $f^*\sF\simeq \sO_{\p^1}(2)\oplus \sO_{\p^1}$. 
We will show in particular that if $\sF|_{X_0}\not\subset T_{X_0/T_0}$,
then $\pi_0$ is a $\p^1$-bundle, and $\sF$ is the pullback via $\pi_0$ of a foliation by rational 
curves on $T_0$. 

\medskip

Let $x\in X$ be a general point, and $\cC_x\subset \p(T_xX^*)$ the variety of
minimal rational tangents at $x$ associated to $H$.
Since $T_{\p^1} \subset f^*\sF$ for general $[f] \in H$, we have 
$\cC_x \subset \p(\sF_x^*)\subset \p(T_xX^*)$.

We claim that $\dim(H_x)=0$.
Indeed, suppose $\dim(H_x)>0$. Then $\cC_x=\p(\sF_x^*)\simeq\p^1$.
By \cite{araujo06}, after shrinking $X_0$ and $T_0$ if necessary,
$\pi_0:X_0\to T_0$ becomes a $\p^2$-bundle over a smooth base.
Hence 
$\sF|_{X_0}=T_{X_0/T_0}\into T_{X_0}$, and thus $f^*\sF\simeq \sO_{\p^1}(2)\oplus \sO_{\p^1}(1)$,
a contradiction. This proves the claim.

Suppose $\sharp(H_x)\ge 2$, and fix $[\ell]\in H_x$.
Then the surface obtained as the union of curves from $H$ meeting $\ell$ at general points is 
invariant under $\sF$, and thus it is the leaf of $\sF$ through $x$.
This shows that the general leaf of $\sF$  is algebraic and rationally connected.

From now on we assume that  $\sharp(H_x)=1$.
After shrinking $X_0$ and $T_0$ if necessary,
we may assume that $\pi_0$ is a $\p^1$-bundle over a smooth base, and 
there is an exact sequence
$$
0 \to T_{X_0/T_0} \to \sF|_{X_0} \to (\pi_0^*\sG_0),
$$
where $\sG_0$ is an invertible subsheaf of $T_{T_0}$.
By Lemma~\ref{lemma:foliation_by_curve_is_algebraic}, 
$\sG_0$ defines a foliation by rational curves on $T_0$.
The general leaf of $\sF$ is the closure of the inverse image by $\pi_0$ of a general leaf
of $\sG$. Hence it is algebraic and rationally connected.
\end{step}

\begin{step} Finally we treat case (3): $r \ge 3$, 
$f^*\sF\simeq \sO_{\p^1}(2)\oplus \sO_{\p^1}(1)^{\oplus r-3}\oplus \sO_{\p^1}^{\oplus 2}$ 
and $H$ is unsplit.
In particular, we show that  one of the following holds:
\begin{itemize}
	\item  $\pi_0$ is a quadric bundle of relative dimension $r-1$, 
		and $\sF$ is the pullback via $\pi_0$ of a foliation by rational 
		curves on $T_0$. 
	\item  $\pi_0$ is a $\p^{r-2}$-bundle, and $\sF$ is the pullback by $\pi_0$ of a foliation by 
		rationally connected surfaces  on $T_0$. 
\end{itemize}

Since $\sO_{\p^1}(2)\subset f^*\sF$, we must have $T_{X_0/T_0} \subset \sF|_{X_0}$ by lemma \ref{lemma:foliation_rc_quotient}.
Thus $f^*T_{X_0/T_0}\simeq \sO_{\p^1}(2)\oplus \sO_{\p^1}(1)^{\oplus r-3}\oplus
\sO_{\p^1}^{\oplus k}$ with $k\in\{0,1,2\}$, and $\dim T_0>0$.

Suppose that $k=2$. Then $\sF|_{X_0}=T_{X_0/T_0}$. In particular, $\sF$ has log canonical singularities along a general leaf.
But this contradicts Proposition~\ref {lemma:common_point}, which asserts that there is a common point through a general leaf of $\sF$. 

Next suppose that $k=1$, and denote by $Y$ the general fiber of $\pi_0$. Then $\dim(T_0)\ge 2$, $\dim Y=r-1$, and $\det(T_Y)\simeq (\sA|_{Y})^{\otimes r-1}$.
Thus $Y\simeq Q_{r-1}$ by \cite{kobayashi_ochiai}.
By Lemmas~\ref{lemma:foliation_morphism} and \ref{lemma:foliation_by_curve_is_algebraic},
$\sF$ induces a foliation $\sG_0$ by rational curves on $T_0$.
The general leaf of $\sF$ is the closure of the inverse image under $\pi_0$ of a general leaf of $\sG_0$. 
Hence it is algebraic and rationally connected.

Finally we suppose that $k=0$, and denote by $Y$ the general fiber of $\pi_0$. Then $\dim(T_0)\ge 3$, $\dim Y=r-2$, and $\det(T_Y)\simeq (\sA|_{Y})^{\otimes r-1}$.
Thus $Y\simeq \p^{r-2}$ by \cite{kobayashi_ochiai}, and, since $H$ is unsplit, $\pi_0$ can be extended in codimension $1$ to a $\p^{r-2}$-bundle over
a smooth base. We still denote this extension by $\pi_0:X_0\to T_0$. 
By lemma \ref{lemma:foliation_morphism}, $\sF$ induces a rank $2$ foliation $\sG_0$ on $T_0$.
By removing a subset of codimension $\geq 2$ in $T_0$ if necessary, we may assume that $\sG_0$ is a subbundle of $T_{T_0}$.
Since $\pi_0$ is smooth, $\sF|_{X_0}=(d\pi_0)^{-1}(\pi_0^*\sG_0)$, and  there exists an
exact sequence of vector bundles
$$
0 \to T_{X_0/Y_0} \to \sF|_{X_0} \to \pi_0^*\sG_0 \to 0.
$$

Let $\sQ$ be a coherent sheaf of $\sO_X$-modules extending 
$\pi_0^*\sG_{0}$.
Let $B\subset X_0$ be a general complete intersection curve for $\sQ$ and $\sA$ in the sense of
Mehta-Ramanathan. 
Consider the induced $\p^{r-2}$-bundle 
$\pi_B:X_B=B\times_{T_0}X_0\to B$, with natural morphisms:
\[
\xymatrix{
X_B \ar[d]_{\pi_B}\ar[r]\ar@/^2pc/[rr]^{q} & X_0 \ar[d]^{\pi_0} \ar[r] & X. \\
B \ar[r]_n & T_0} 
\]
Then 
$\sO_{X_B}\big(q^*(-K_{\sF})\big)=\sO_{X_B}(-K_{X_B/B})\otimes \pi_B^*\big(n^*(\det (\sG_0))\big)$. 
We know that $-K_{X_B/B}$ is not ample by \cite[Theorem 2]{miyaoka93}. Hence $\deg_B\big(n^*(\det (\sG_0))\big) > 0$.

If $n^*\sG_0$ is ample, then the general leaf of $\sG_0$ is algebraic and rationally connected  by
Theorem~\ref{bogomolov_mcquillan}.  
Since the general leaf of $\sF$ is the closure of the inverse image under $\pi_0$ of a general leaf of $\sG_0$,
it follows that the general leaf of $\sF$ is algebraic and rationally connected.
So we may assume that $n^*\sG_0$ is not ample,
and hence not semistable.

By Lemma~\ref{lemma:HN_on_the_base},  there is a saturated rank 1 subsheaf $\sM_0\subset \sG_0$ such that
$\deg_B(n^*\sM_0) > \frac{1}{2} \deg_B(n^*\sG_0) > 0$. 
By Theorem~\ref{bogomolov_mcquillan}, the general leaf of the foliation 
$\sM_0\subset T_{T_0}$ is a rational curve. 
Let $C\simeq \p^1$ be a smooth compactification of a general leaf $C_0$ of 
$\sM_0$, and let $X_C$ be the 
normalization of the closure  of $\pi_0^{-1}(C_0)$ in $X$, with
induced morphism $\pi_C:X_C\to C$. 
Denote by $\sA_{X_C}$ the pullback of $\sA$ to $X_C$.
Every fiber of $\pi_C$ is generically reduced and irreducible since it has degree one with respect to the ample line bundle $\sA_{X_C}$.
Since $C$ is smooth, $\pi_C$ is flat, and $X_C$ is normal, every fiber satisfy Serre's
condition $\textup{S}_1$, and hence it is integral.
Therefore $\pi_C:X_C\to C\simeq \p^1$ is a $\p^m$-bundle by \cite[Corollary 5.4]{fujita75}.
In particular, $X_C$ is smooth.

By removing a subset of codimension $\geq 2$ in $T_0$ if necessary, we may assume that 
$\sM_0$ is a line bundle on $T_0$, and that there is a line bundle $\sM_0'$ on $T_0$ fitting into an exact sequence
$$
0 \to \sM_0 \to \sG_0 \to \sM'_0 \to 0.
$$
Let $\sL$ and $\sL'$ be the unique line bundles on $X$ extending $\pi_0^*\sM_0$ and $\pi_0^*\sM'_0$,
respectively. 
Let $\sK$ be the saturated subsheaf of $T_X$ extending $(d\pi_0)^{-1}(\pi_0^*\sM_0)$. 
Then  $\sK$ is a foliation of rank $r-1$ on $X$, and $\det(\sK)\simeq\sA^{\otimes r-1}\otimes \sL'^{*}$.
Notice also that $X_C$ is the normalization of a general 
leaf of $\sK$, and that $\sK$ is regular along a general fiber of $\pi_0$.

Denote by $(X_C, D)$ the general log leaf of $\sK$, and by $q:X_C\to X$ the natural morphism. 
Since $\sK$ is regular along a general fiber of $\pi_0$,
$D$ is supported on a finite union of fibers of $\pi_C$. 
Recall that $\sO_{X_C}(-K_{X_C})\simeq q^*\det(\sK)\otimes \sO_{X_C}(D)$.
So we have 
$$
q^*\sL'\ \simeq \ \sO_{X_C}(K_{X_C})\otimes \sA_{X_C}^{\otimes r-1} \otimes \sO_{X_C}(D).
$$
On the other hand, since $\dim X_C = r-1$ and $\sA_{X_C}$ is ample, 
Fujita's theorem (\cite[]{}) implies that $\sO_{X_C}(K_{X_C})\otimes \sA_{X_C}^{\otimes r-1}$ is nef. 
We remark that exactly one of the following holds.
\begin{enumerate}
	\item[(a)] For any moving section $\sigma\subset X_C$ of $\pi_C$, we have $q^*\sL'\cdot \sigma >0$.
	\item[(b)] $\sO_{X_C}(-K_{X_C})\simeq \sA_{X_C}^{\otimes r-1}$ and $D=0$. In this case 
$X_C\simeq \p^1\times \p^1$, $\sA_{X_C}\simeq\sO_{\p^1\times \p^1}(1,1)$ and $q^*\sL'\simeq \sO_{X_C}$.
\end{enumerate}

Let $W$ be the normalization of the closure in $\Chow(X)$ of the subvariety parametrizing 
general leaves of $\sK$, and  $U$ the normalization of the universal cycle over $W$,
with universal family morphisms:

\centerline{
\xymatrix{
U \ar[r]^{e}\ar[d]_{\pi} & X \ . \\
 W &
}
}
\noindent There are open subsets $X_1\subset X$ and $W_1\subset W$, with $\codim_X(X\setminus X_1)\geq 2$, 
such that $e$ is an isomorphism over $X_1$, and 
$\pi$ induces an equidimensional morphism $\pi_1:X_1\to W_1$ with connected fibers.
Notice that $\sK|_{X_1}=T_{X_1/W_1}$. 
By Lemma~\ref{lemma:foliation_morphism}, $\sF$ induces a rank $1$ foliation $\sN_1\subset T_{W_1}$, and  an
exact sequence 
$$
0 \to \sK|_{X_1} \to \sF|_{X_1} \to (\pi_1^*\sN_1)^{**}.
$$ 
Thus, there is a canonically defined effective divisor $D_1$ on $X_1$ such that 
$$
\sL'|_{X_1}\otimes\sO_{X_1}(D_1) \simeq (\pi_1^*\sN_1)^{**}.
$$
By removing a subset of codimension $\geq 2$ in $W_1$ if necessary, we may assume that 
$W_1$ is smooth and $\sN_1$ is locally free.
Our aim is to show that the general leaf of $\sN_1$ is a rational curve.
Since the general leaf of $\sF$ is the closure of the inverse image under $\pi_1$ of a general leaf of $\sN_1$,
it will follow that the general leaf of $\sF$ is algebraic and rationally connected.

Let $\sigma\subset X_C$ be a moving section. Then its image in $X$ is a moving curve.
Let $\ell'$ be a general deformation of $q(\sigma)$. In particular, we may assume that $\ell'\subset X_0\cap X_1$, and that
both $\sF$ and $\sK$ are regular along $\ell'$ by \cite[II.3.7]{kollar96}.

Suppose that $\ell'$ is not tangent to $\sK$. By Lemma~\ref{lemma:curve_tangent_to_F}, we must have $D\neq 0$. So we are in case (a)
above, and thus $ \sL'\cdot \ell' = q^*\sL'\cdot \sigma >0$. 
Now consider the curve $\pi_1(\ell')$ in $W_1$. It is a moving curve, the foliation $\sN_1\subset T_{W_1}$ is regular along $\pi_1(\ell')$, and
$\sN_1\cdot \pi_1(\ell')>0$.
By Theorem~\ref{bogomolov_mcquillan}, this implies that  the leaves of 
$\sN_1$ are rational curves.

Next we suppose that $\ell'$ is tangent to $\sK$. 
Set $C'=\pi_0(\ell')\subset T_0$. Then $C'$ is a leaf of $\sM_0\subset T_0$, and $\sM_0$ is regular along $C'$, i.e., $\sM_0|_{C'}\simeq T_{C'}$.
Moreover, the same holds for a general deformation of $C'$ by  Lemma~\ref{lemma:curve_tangent_to_F}.
Thus $T_{T_0}|_{C'}\simeq \sO_{\p^1}(2)\oplus \sO_{\p^1}^{\oplus \dim T_0 -1}$.
As in Step 1, we see that $\sG_0|_{C'}\simeq \sO_{\p^1}(2)\oplus \sO_{\p^1}$.
This implies that $\sL'\cdot \ell'=\sM_0\cdot C'=0$.
So we are in case (b) above:
$X_C\simeq \p^1\times \p^1$, $\sA_{X_C}\simeq\sO_{\p^1\times \p^1}(1,1)$ and $D=0$.
We may assume that $\sigma$ is the ruling of $X_C$ dominating $C$.
Thus $\ell'$ determines an unsplit dominating
family $H'$ of rational curves on $X$, and
$\pi_1:X_1\to W_1$ is nothing but 
the $(H,H')$-rationally connected quotient of $X$. 
By Lemma~\ref{lemma:foliation_by_curve_is_algebraic},
$\sN_1$ is a foliation by rational curves on $W_1$.
\end{step} 
\end{proof}


\section{On del Pezzo foliations with mild singularities}\label{section:thmb}

In this section we prove Theorem~\ref{thmb}.
In fact, Theorem~\ref{thmb} 
follows immediately from 
Theorem~\ref{thmb'} below and Proposition~\ref{lemma:common_point}.

\begin{thm}\label{thmb'}
Let $\sF\subsetneq T_X$ be a del Pezzo foliation of rank $r$ on a smooth projective $n$-dimensional variety $X$.
Suppose that $\sF$ is locally free along  the closure of  a general leaf.
Then 
\begin{enumerate}
	\item either $X$ a $\p^m$-bundle over $\p^{n-m}$, with $1\leq m\leq n-1$, or
	\item for any minimal dominating family of rational curves on $X$, with associated
		rationally connected quotient $\pi_0:X_0\to T_0$, we have $\sF|_{X_0}\subset T_{X_0/T_0}$.
\end{enumerate}
\end{thm}

When $X\not\simeq \p^n$, we know from Theorem~\ref{thma} that a del Pezzo foliation $\sF$ on $X$
is algebraically integrable. 
In the proof of Theorem~\ref{thmb'} in this case, we will consider a suitable resolution of singularities of the general leaf of
$\sF$, whose existence  is guaranteed by the following theorem.

\begin{thm}[{\cite[Theorem 3.35, 3.45]{kollar07} and \cite[Corollary 4.7]{greb_kebekus_kovacs10}}]
\label{thm:canonical_resolution}
Let $X$ be a normal variety. Then there exists a resolution of singularities
$d:Y\to X$ such that 
\begin{itemize}
\item $d$ is an isomorphism over $X\setminus\textup{Sing}(X)$, and
\item $d_* T_Y(-\textup{log }D)\simeq T_X$ where $D$ is the largest
reduced divisor contained  in $d^{-1}(\textup{Sing}(X))$. 
\end{itemize}
\end{thm}

We call a resolution $d$ as in Theorem~\ref{thm:canonical_resolution} a 
\emph{canonical desingularization} of $X$.

\begin{say} \label{notation:D+Delta}
Let $X$ be a normal projective variety, and $\sF$ an algebraically integrable 1-Gorenstein foliation on $X$.
We denote by $F$ the closure of a general leaf of $\sF$, $\tilde F$ its normalization, $(\tilde F,\tilde \Delta)$
the corresponding log leaf, and $Y$ a canonical 
desingularization of $\tilde F$: 

\centerline{
\xymatrix{
Y \ar[r]^{d}\ar@/^2pc/[rrr]^{\bar e} & \tilde F \ar[r] \ar@/_2pc/[rr] _{\tilde e} & F \ar[r]_{e} & X. \\
}
}

Suppose that $\sF$ is locally free along $F$.
Let $D\subset Y$
be the largest
reduced divisor contained in $d^{-1}(\textup{Sing}(\tilde F))$. 
Then
$\bar e^*\sF\subset T_{Y} (-\log D) \subset T_{Y}$.
Therefore there exists an effective divisor $\Delta$ on $Y$ such that 
$$
\bar e^*(-K_\sF)+\Delta+D= -K_{Y}.
$$
Recall from Definition \ref{defn:log_leaf} that $K_{\tilde F}+\tilde \Delta = \tilde e^*K_{\sF}$.
Hence we have  $\Delta+D=d^*\tilde \Delta-K_{Y/\tilde F}$.
Moreover $\textup{Sing}(\tilde F)\subset\Supp(\tilde\Delta)$
by Lemma \ref{lemma:singular_locus_normalization}.
Thus $\Supp(\Delta+D)\subset d^{-1}(\Supp(\tilde\Delta))$.
\end{say}

\begin{rem} \label{rem:D+Delta}
Let the notation and assumptions be as in \ref{notation:D+Delta} above, and
suppose moreover that $\sF$ is a Fano foliation.
Then $\Delta+D \neq 0$. 
Indeed, if $\Delta+D=0$, it follows from the above discussion that $\tilde F$ is smooth and $\tilde \Delta=0$. 
Therefore, $\sF$ is induced by an almost proper map $X\map T$, contradicting  
Proposition~\ref{lemma:common_point}.
\end{rem}

\begin{proof}[Proof of Theorem~\ref{thmb'}]
Write $\det(\sF)=\sA^{\otimes r-1}$, with $\sA$ an ample 
line bundle on $X$, and denote by $S$ the singular locus of $\sF$.
We follow the notation introduced in \ref{notation:D+Delta} above.

Let $H$ be a minimal dominating family of rational curves on $X$, and  $\pi_0:X_0\to T_0$
the associated rationally connected quotient.
Let $[f]\in H$ be a general member.
Suppose that $\sF|_{X_0}\not\subset T_{X_0/T_0}$.
Recall from the proof of Theorem~\ref{thma} that one of the following holds.
\begin{enumerate}
	\item Either
	$f^*\sF\simeq \sO_{\p^1}(1)^{\oplus r-1}\oplus
	\sO_{\p^1}$, or
	
	\item $r=2$, $f^*\sF\simeq \sO_{\p^1}(2)\oplus
	\sO_{\p^1}$, $\pi_0$ is a $\p^1$-bundle, 
	and $\sF$ is the pullback via $\pi_0$ of a foliation by rational 
	curves on $T_0$, or
	
	\item $r\geq 3$, $H$ is unsplit, $f^*\sF\simeq \sO_{\p^1}(2)\oplus \sO_{\p^1}(1)^{\oplus r-3}\oplus
	\sO_{\p^1}^{\oplus 2}$, and one of the following holds.
		\begin{enumerate}
			\item Either $\pi_0$ is a quadric bundle of relative dimension $r-1$, 
				and $\sF$ is the pullback via $\pi_0$ of a foliation $\sG_0$ by rational 
				curves on $T_0$, or
			\item  $\pi_0$ is a $\p^{r-2}$-bundle, and $\sF$ is the pullback by $\pi_0$ of a foliation 
				$\sG_0$ by 
				rationally connected surfaces  on $T_0$. 
		\end{enumerate}
\end{enumerate}

If we are in case (1), then $\pi_0$ makes $X$ a $\p^m$-bundle over $\p^{n-m}$, with $1\leq m\leq n-1$, by Proposition~\ref{proposition:p^m_bundle_over_p^n}.

\medskip

Suppose we are in case (2). Notice that $\sF$ is regular along a general curve from $H$.

The restriction of  $\pi_0$ to $F\cap X_0$ induces a surjective morphism with connected fibers 
$\varphi:Y\to \p^1$. Let $f\subset Y$ be a general fiber of $\varphi$ , and set
$\ell:=\bar e(f)\subset F$. Then $\ell\cap S=\emptyset$, and $F$ is smooth along $\ell$.
Thus $f \cap \Supp(\Delta+D)=\emptyset$. Hence $\Supp(\Delta+D)$ is  a union 
of irreducible components of fibers of $\varphi$.

We claim that $Y=\tilde F\cong \p^1\times \p^1$. Suppose otherwise. Then there exists a section 
$C_0$ of $\varphi$ such that $C_0^2<0$. Since $C_0\not\subset \Supp(\Delta+D)$, we have:
$$
1\ \geq \ C_0^2 +2 \ =\  -K_Y\cdot C_0 \ =\  -K_\sF\cdot \bar e_*(C_0)+(\Delta+D)\cdot C_0\ \geq \ 1.
$$
Hence we must have $(\Delta+D)\cdot C_0=0$, and thus the inclusion $\bar e^*\sF\subset T_Y$
is an isomorphism in a neighborhood of $C_0$. In particular, no smooth fiber of $\varphi$ 
is contained in $\Supp(\Delta+D)$. 
Suppose $\varphi$ has reducible fibers, and let $C_1$ be an irreducible component 
of a reducible fiber such that $C_0\cap C_1\not=\emptyset$.
Then $C_1^2<0$ and $C_1\not\subset \Supp(\Delta+D)$. As before, we get that 
$(\Delta+D)\cdot C_1=0$, and thus the inclusion $\bar e^*\sF\subset T_Y$
is an isomorphism in a neighborhood of $C_1$.
Proceeding by induction, we conclude that no irreducible component of a reducible fiber 
of $\varphi$ is contained in $\Supp(\Delta+D)$.
Thus $\Delta+D=0$. But this is impossible by Remark~\ref{rem:D+Delta}.
Therefore we must have $Y=\tilde F\cong \p^1\times \p^1$, as claimed. 
In particular, $D=0$.

Let $C_0$ be a section of $\varphi$ such that $C_0^2=0$, and denote by $f$ a fiber of $\varphi$.
By  Remark~\ref{rem:D+Delta}, $\Delta \neq 0$. Thus $\Delta\equiv mf$ for some integer $m\geq 1$. 
We have:
$$
2 \ = \ -K_Y\cdot C_0\  = \ -K_\sF\cdot \bar e_*(C_0)+\Delta\cdot C_0\ \geq \ 1 + m\ \geq \ 2.
$$
Hence we must have $-K_\sF\cdot \bar e_*(C_0)=1$ and $\Delta = f$. 
Set $\ell'=\bar e(C_0)$, and 
let $H'$ be the family of rational curves on $X$ containing $[\ell']$.
Then  $-K_\sF\cdot \ell'=1$, and thus $H'$ is unsplit. 
Let  $[f']\in H'$ be a general member.
As in Step 1 of the proof of Theorem~\ref{thma}, we see that $(f')^*\sF\simeq \sO_{\p^1}(1)\oplus \sO_{\p^1}$.
Let $\pi':X'\to T'$ be the $H'$-rationally connected quotient of $X$.
Notice that $\ell$ and $\ell'$ are numerically independent in $X$.
Therefore $\ell$ is not contracted by $\pi'$.
On the other  hand,  $\ell$ is contained in a leaf of $\sF$.
So we must have $\sF|_{X'}\not\subset T_{X'/T'}$.
From the analysis of case (1) above and Proposition~\ref{proposition:p^m_bundle_over_p^n}, 
we conclude that $\pi'$ makes $X$ a $\p^{n-1}$-bundle over $\p^{1}$.

\medskip

Next we show that case (3a) does not occur.

Suppose to the contrary that $H$ is unsplit, $\pi_0:X_0\to T_0$
 is a quadric bundle of relative dimension $r-1$, 
and $\sF$ is the pullback via $\pi_0$ of a foliation $\sG_0$ by rational curves on $T_0$. 
By \cite[Lemma 2.2]{adk08}, $\pi_0$ can be extended in codimension $1$ in $X$ to a proper surjective equidimensional morphism 
with irreducible and reduced fibers. We still denote this extension by $\pi_0:X_0\to T_0$.

The morphism $\pi_0$ induces a morphism  $\pi_B:Y\to \p^1$, where $\p^1\map B$ is a smooth 
compactification of a general leaf $B$ of the foliation $\sG_0$ on $T_0$. 
Since $\sF$ is regular along a general fiber of $\pi_0$, $\Supp\big(\Delta+D\big)$ is a union 
of irreducible components of fibers of $\pi_B$.

Let $\ell'\subset Y$  be a general curve from a minimal horizontal 
family of rational curves with respect to $\pi_B$. 
Then $\ell'\not\subset \Supp(\Delta+D)$, and 
$$
-K_Y\cdot \ell' = -K_{\sF}\cdot \bar e_*\ell' + (\Delta+D)\cdot \ell' 
= (r-1)\sA\cdot \bar e_*\ell' + (\Delta+D)\cdot \ell' \geq r-1\geq 2.
$$
On the other hand, by Lemma~\ref{lemma:horizontal}, $-K_Y\cdot \ell' \leq 2$.
So we must have $r=3$, $\sA\cdot \ell'=1$, and $(\Delta+D)\cdot \ell'=0$.
The latter implies that $\bar e(\ell')\cap S=\emptyset$. 
Thus $\bar e(\ell')\cap \textup{Sing}(F)=\emptyset$, 
and all fibers of $\pi_B$ meet the regular locus of $\sF$.
By Remark~\ref{rem:D+Delta},  $\Delta +D \neq 0$. 
Thus $\pi_B$ has at least one reducible fiber, and at least one irreducible component 
of such reducible fiber meets  the regular locus of $\sF$.
By letting $B$ run through general leaves of the foliation $\sG_0$, 
the images in $X$ of reducible fibers of $\pi_B$ sweep out a divisor on $X$. 
But this contradicts the fact that $\pi_0$ has irreducible fibers in codimension $1$.

\medskip

Finally we show that case (3b) can only occur if $X$ is a $\p^m$-bundle over $\p^{n-m}$.

So suppose $H$ is unsplit,  $r\geq 3$,  $f^*\sF\simeq \sO_{\p^1}(2)\oplus \sO_{\p^1}(1)^{\oplus r-3}\oplus
\sO_{\p^1}^{\oplus 2}$,  $\pi_0$ is a $\p^{r-2}$-bundle, and $\sF$ is the pullback via 
$\pi_0$ of a foliation $\sG_0$ by rationally connected surfaces  on $T_0$. 
Recall that $\pi_0$ can be extended to a $\p^{r-2}$-bundle in codimension $1$ in $X$.
We still denote this extension by $\pi_0:X_0\to T_0$.

Let $Z$ be the normalization of a general leaf of $\sG_0$. Then $\pi_0$ induces a 
$\p^{r-2}$-bundle $\pi_Z:Y_0\to Z_0$, where $Y_0$ and $Z_0$ are dense open subsets
of $Y$ and $Z$, respectively.
Let $H'_Y$ be a minimal h-dominating
family of rational curves with respect to $\pi_Z$, and $[\ell']\in H'_Y$ a general member.
Fix a family  $H'$ of rational curves on $X$
containing a point of $\rat(X)$ corresponding to $\bar e(\ell')$.

Since $\sF$ is regular along a general fiber of $\pi_0$, $\ell'\not\subset \Supp(\Delta+D)$, and
 $$
-K_Y\cdot \ell' = -K_{\sF}\cdot \bar e_*\ell' + (\Delta+D)\cdot \ell' 
= (r-1)\sA\cdot \bar e_*\ell' + (\Delta+D)\cdot \ell' \geq r-1\geq 2.
$$
On the other hand, by Lemma~\ref{lemma:horizontal}, $-K_Y\cdot \ell' \leq 3$.
So $-K_Y\cdot \ell' \in \{2,3\}$. Moreover $\sA\cdot\bar e_*\ell' =1$, and thus $H'$ is unsplit.

First let us assume that $-K_Y\cdot \ell' =3$. Then $r=4-(\Delta+D)\cdot \ell'\in \{3,4\}$.
By Lemma~\ref{lemma:horizontal}, $H'_Y$ is a dominating family of rational curves on $Y$.
Thus $H'$ is an unsplit dominating family of rational curves on $X$.
Let $[f']\in H'$ be a general member, and $\pi':X'\to T'$ 
the $H'$-rationally connected quotient of $X$.
Notice that $H$ and $H'$ are numerically independent in $X$.
Therefore the general curve from $H$ is not contracted by $\pi'$, while
it is contained in a leaf of $\sF$.
So we must have $\sF|_{X'}\not\subset T_{X'/T'}$.
From the analysis of the previous cases, we conclude that
either $f'$ falls under case (1) above, and so $X$ is a $\p^m$-bundle over $\p^{n-m}$,
or $f'^*\sF\simeq \sO_{\p^1}(2)\oplus \sO_{\p^1}(1)^{\oplus r-3}\oplus
\sO_{\p^1}^{\oplus 2}$,  $\pi'$ is a $\p^{r-2}$-bundle, and $T_{X'/T'}\subset \sF|_{X'}$.
In the latter case, $\sF$ is regular along a general fiber of $\pi'$.
Thus $(\Delta+D)\cdot \ell'=0$ and $r=4$.
Let $\pi'':X''\to T''$ be
the ($H,H'$)-rationally connected quotient of $X$.
By Lemma~\ref{lemma:foliation_rc_quotient}, $T_{X''/T''}\subset \sF|_{X''}$, 
and thus $rank(T_{X''/T''})\leq rank(\sF)=4$. On the other hand,
since $H$ and $H'$ are numerically independent in $X$,
the fibers of the $\p^2$-bundles $\pi_0$ and $\pi'$ 
cannot meet along a positive dimensional variety. Therefore the fibers of 
$\pi''$ have dimension at least $4$. 
We conclude that $\sF|_{X''}= T_{X''/T''}$. But this is impossible by Proposition~\ref{lemma:common_point}.

From now on we assume that $-K_Y\cdot \ell' =2$. Then $(\Delta+D)\cdot \ell'=0$ and $r=3$.
By Lemma~\ref{lemma:horizontal}, either $H'_Y$ is a dominating family of rational curves on $Y$,
or $Locus(H'_Y)$ has codimension $1$ in $Y$.

If $H'_Y$ dominating, then  $H'$ is an unsplit dominating family of rational curves on $X$.
Let $\pi':X'\to T'$ be
the $H'$-rationally connected quotient of $X$.
As before, we conclude that $T_{X'/T'}\subset \sF|_{X'}$.
Let $\pi'':X''\to T''$ be
the ($H,H'$)-rationally connected quotient of $X$.
By Lemma~\ref{lemma:foliation_rc_quotient}, $T_{X''/T''}\subset \sF|_{X''}$. 
By Proposition~\ref{lemma:common_point}, $T_{X''/T''}\neq \sF|_{X''}$. Thus 
$rank(T_{X''/T''})=2$, and $\sF$ is the pullback via $\pi''$ of a foliation by rational curves on $T''$. 
The same argument used in case (3a) above shows that this is impossible.

Finally, we assume that $Locus(H'_Y)$ has codimension $1$ in $Y$.
Since $(\Delta+D)\cdot \ell'=0$, we have $\bar e(\ell')\cap S=\emptyset$.
Therefore the general member of $H'$ avoids $S$ and is tangent to $\sF$ by 
Lemma~\ref{lemma:curve_tangent_to_F}. 
Let $\pi'':X''\to T''$ be
the ($H,H'$)-rationally connected quotient of $X$, and denote by $F''$ a general fiber of $\pi''$.
By Lemma~\ref{lemma:foliation_rc_quotientB}, $T_{X''/T''}\subset \sF|_{X''}$. 
In particular,  $\dim F'' \leq 3=rank(\sF)$. 
We will show that $\dim F''=3$. 
From this it follows that $\sF|_{X''}=T_{X''/T''}$, contradicting 
Proposition~\ref{lemma:common_point}, and finishing the proof of Theorem~\ref{thmb'}.

Let $y\in Locus(H'_Y)$ be a general point. By \cite[IV.2.6.1]{kollar96},
$$
3+2 =\dim Y + (-K_Y\cdot \ell' ) \leq \dim \big(Locus(H'_Y)\big) + 
\dim\Big(Locus\big((H'_Y)_y\big)\Big) +1  \leq 2 +  \dim\Big(Locus\big((H'_Y)_y\big)\Big)   +1. 
$$
Thus $\dim\Big(Locus\big((H'_Y)_y\big)\Big)=2$. 
Since $Locus\big((H'_Y)_y\big) \subset Locus(H'_Y)$, and the latter is irreducible and 
$2$-dimensional, we conclude that $\Big(Locus\big((H'_Y)_y\big)\Big) = Locus(H'_Y)$.
Then the image of $Locus(H'_Y)$ in $X$ is contained in a general fiber of $\pi''$.
Moreover,  it  does not contain any curve from the family $H'$,
since $H$ and $H'$ are numerically independent in $X$. Thus $\dim F''= 3$. 
\end{proof}


\section{Del Pezzo foliations on projective space bundles}\label{section:examples_pn_bundles}


Our first aim in this section is to give a precise geometric description of 
del Pezzo foliations $\sF$ on projective space bundles $X\to \p^l$ such that 
$\sF\nsubseteq T_{X/\p^l}$.

\begin{say}[Two special cases]\label{say:special_cases}
Let $\sE$ be an ample locally free sheaf of rank $m+1\geq 2$ on $\p^l$, 
and set $X:=\p_{\p^l}(\sE)$.
Denote by $\sO_X(1)$ the tautological line bundle on $X$, and by
$\pi:X\to \p^l$ the natural projection.
Let $\sF\nsubseteq T_{X/\p^l}$ be a del Pezzo foliation on $X$, and write
$\det(\sF)\simeq \sA^{\otimes r_{\sF}-1}$ for an ample 
line bundle $\sA$ on $X$.
By Proposition~\ref{prop:classification_bundle}, $r_{\sF}\in \{2,3\}$.
We first determine the restriction of $\sA$ to a general line $\ell$ on a fiber 
of $\pi$. 
As in Step 1 of the proof of Theorem~\ref{thma}, we verify that $\sA\cdot \ell=1$ unless 
$m=1$, $r_{\sF}=2\leq l$, $\sF|_{\ell}\simeq \sO(2)\oplus \sO$,  and 
$T_{X/\p^l}\subsetneq \sF$.

Suppose we are in the latter case.
We claim that $X\simeq \p^1\times \p^l$, and $\sF$ is the pullback 
via $\pi$ of a degree 0 foliation of rank 1 on $\p^l$.
Indeed, by Lemma~\ref{lemma:foliation_morphism}, 
$\sF$ is the pullback by  $\pi$ of a rank 1 foliation $\sG\subset T_{\p^l}$.
So $\det(\sF)\simeq\det(T_{X/\p^l})\otimes \pi^*\det(\sG)$, and 
$\sA\simeq \sO_X(2)\otimes \pi^*\big(\det(\sE^*)\otimes\sG\big)$.
Write $\sG\simeq\sO_{\p^l}(k)$ for some integer $k$.
If $k\le 0$, then $\det(T_{X/\p^l})\simeq\sA\otimes \pi^*\sO_{\p^m}(-k)$ is ample, contradicting
\cite[Theorem 2]{miyaoka93}.
By Bott's formulae, $k=1$.
Let $\p^1\subset\p^l$ be a line, and write 
$\sE|_{\p^1}\simeq\sO_{\p^1}(a)\oplus\sO_{\p^1}(b)$ with $a\le b$.
Let $\sigma:\p^1\to \p_{\p^1}(\sE|_{\p^1})$ be the section corresponding to
the projection $\sO_{\p^1}(a)\oplus\sO_{\p^1}(b)\twoheadrightarrow\sO_{\p^1}(a)$, and set
$C:=\sigma(\p^1)$. Then
$$1 \le \sA\cdot C=\big(\sO_X(2)\otimes \pi^*(\det(\sE^*)\otimes\sG)\big)\cdot C =2a-(a+b)+1.$$
Thus $a\ge b$, and so $a=b$. By \cite[Theorem 3.2.1]{OSS}, 
$\sE\simeq\sO_{\p^m}(a)^{\oplus 2}$. 
This proves the claim.

So we may restrict ourselves to the case when $\sA$ restricts to $\sO(1)$ on the fibers 
of $\pi$. Then, by replacing $\sE$ with $\pi_*\sA$ if necessary, we may assume that 
$\det(\sF)\simeq \sO_X(r_{\sF}-1)$.

By Proposition~\ref{prop:classification_bundle}, if $m=1$, then 
$l\geq r=3$, $X\simeq \p^1\times \p^l$, $\sH=T_{X/\p^l}$, and $\sF$ 
is the pullback via the natural projection $\p^1\times \p^l\to \p^l$ 
of a degree zero foliation $\sO_{\p^l}(1)\oplus\sO_{\p^l}(1)\subsetneq T_{\p^l}$ on $\p^l$. 
So we may assume that $m\geq 2$.
\end{say}

\begin{thm}\label{thm:description}
Let $\sE$ be an ample locally free sheaf of rank $m+1\geq 3$ on $\p^l$, 
and set $X:=\p_{\p^l}(\sE)$.
Denote by $\sO_X(1)$ the tautological line bundle on $X$, and by
$\pi:X\to \p^l$ the natural projection.
Let $\sF\nsubseteq T_{X/\p^l}$ be a foliation of rank $r\geq 2$ on $X$ 
such that $\det(\sF)\simeq\sO_X(r-1)$.
\begin{enumerate}
	\item The possible values for the pair  $(l,r)$ are $(1,2)$, $(1,3)$ and $(l,2)$, with $l\geq 2$.
	\item There exists a subbundle $\sV\subset \sE^*$ such that 
		$\sF\cap T_{X/\p^l}\simeq (\pi^*\sV)(1)$, an inclusion $j:\det(\sV^*)\into T_{\p^l}$, and 
		a commutative diagram of exact sequences

\centerline{
\xymatrix{
0 \ar[r] &   \sF\cap T_{X/\p^l}\simeq(\pi^*\sV)(1) \ar[r]\ar[d] & \sF \ar[r]\ar[d] & 
\sI_{W}\otimes \pi^*\det(\sV^*) \ar[d]^{\pi^*j}\ar[r] & 0 \\
0 \ar[r] &  T_{X/\p^l} \ar[r] &  T_X \ar[r] & \pi^*T_{\p^l} \ar[r] & 0
}
}
\noindent where $W\subset X$ is a closed subscheme with $\codim_XW\geq 2$.

	\item If $l\geq 2$, then  $\sV\simeq \sO_{\p^l}(-1)$. 
		If $l=1$, then either $\sV\simeq\sO_{\p^1}(-1)$, or $\sV\simeq\sO_{\p^1}(-2)$, 
		or $\sV\simeq\sO_{\p^1}(-1)^{\oplus 2}$.
	\item Let $\sK$ be the kernel of the dual map $\sE\twoheadrightarrow \sV^*$, and consider
		the $\p^{m-r+1}$-bundle $Z:=\p_{\p^l}(\sK)$, with natural projection $q:Z\to \p^l$.
		Then $\sF$ is the pullback by the linear projection $X/\p^l\map Z/\p^l$
		of a foliation on $Z$ induced by a nonzero global section of $T_Z\otimes q^*\det(\sV)$.
	\item If $l=1$ and $m\geq r+1$, then $\sF$ is locally free and $W=\emptyset$. 
		If moreover $\sV\simeq\sO_{\p^1}(-1)$, then $\sF\simeq (\pi^*\sV)(1) \oplus \pi^*\det(\sV^*)$.
\end{enumerate}
\end{thm}

\begin{proof}
Item (1) was proved in Proposition~\ref{prop:classification_bundle}

Set $\sH:=\sF\cap T_{X/\p^l}$, and recall from the proof of Proposition~\ref{prop:classification_bundle} 
that $\sH\simeq (\pi^*\sV)(1)$, where $\sV\subset \sE^*$ is a saturated subsheaf of rank $r-1$ 
(and thus a subbundle in codimension $1$ in $\p^l$). 
Moreover, there is an inclusion $\det(\sV^*)\subset  T_{\p^l}$, and an isomorphism
$\big(\sF/\sH\big)^{**}\simeq  \pi^*\det(\sV^*)$.
We claim that $\sV$ is in fact a subbundle of $\sE^*$. If $l=1$, then this is clear.
Moreover, in this case either $\sV\simeq\sO_{\p^1}(-1)$, or $\sV\simeq\sO_{\p^1}(-2)$, or
 $\sV\simeq\sO_{\p^1}(-1)^{\oplus 2}$.
If $l\geq 2$, then $r=2$, and $\sV$ is locally free of rank $1$. 
The condition $\det(\sV^*)\subset T_{\p^l}$ implies that $\sV\simeq \sO_{\p^l}(-1)$.
Since $\sE$ is ample, $\sV$ must be a subbundle of $\sE^*$.
This proves (2) and (3).

Let $\sK$ be the kernel of the dual map $\sE\twoheadrightarrow \sV^*$.
Consider the $\p^{m-r+1}$-bundle $Z:=\p_{\p^l}(\sK)$, with natural projection $q:Z\to \p^l$.
The surjection $\sE\twoheadrightarrow\sV^*$ induces a rational map 
$p:X\map Z$ over $\p^l$, which restricts to a surjective morphism 
$p_0:X_0\to Z$, where $X_0$ is the complement in $X$ of the $\p^{r-2}$-subbundle 
$\p(\sV^*)\subset \p(\sE)$. 
As in \ref{V_in_E}, we have $\sH|_{X_0} = T_{X_0/Z}$.
By Lemma~\ref{lemma:foliation_morphism}, $\sF|_{X_0}$ is the pullback via $p_0$
of a rank 1 foliation $\sG\subsetneq T_Z$.
One checks easily that $\sG\simeq q^*\det(\sV^*)$. 
This proves (4).

In order to prove (5), recall from \ref{V_in_E} that, since $l=1$, 
$\sH|_F$ is a degree 0 foliation of rank $r-1$ on 
$F\simeq \p^m$ for \emph{any} fiber $F$ of $\pi$.
Thus the map $\det(\sH)\hookrightarrow\wedge^{r-1}T_X$ vanishes along a closed subset 
of dimension equal to $1+(r-2)=r-1\le \dim(X)-3$.
I.e., $\sH$ is a subbundle of $\sF$ in codimension $\le 2$.
By lemma \ref{lemma:local_freeness},
$\sF$ is locally free and $W=\emptyset$. 
If moreover $\sV\simeq\sO_{\p^1}(-1)$, then
$$
\begin{array}{cccl}
H^1\big(X,\sH\otimes \pi^*\det(\sV)\big) & \simeq  & 
H^1\big(X,\pi^*\sV(1)\otimes \pi^*\sO_{\p^1}(-1)\big)
& \\
& \simeq & H^1\big(\p^1,\sO_{\p^1}(-2)^{\oplus 2}\otimes\sE\big) 
& \text{ by Leray's spectral sequence}\\
& = & 0 & \text{ since $\sE$ is ample.}
\end{array}
$$
Hence $\sF\simeq \sH\oplus \pi^*\det(\sV^*)$.
\end{proof}

Our next goal is to classify  locally free sheaves $\sE$
on $\p^l$ for which $X=\p_{\p^l}(\sE)$ admits a del Pezzo foliation $\sF\nsubseteq T_{X/\p^l}$.
For that purpose, we first recall the definition and basic properties of the Atiyah class of  
locally free sheaves on smooth varieties.

\begin{say}[The Atiyah class of a locally free sheaf]  \label{say:atiyah}
Let $T$ be a smooth variety, and $\sE$ a locally free sheaf of rank $m+1\geq 1$ on $T$.
Let $J_T^1(\sE)$ be the sheaf of $1$-jets of $\sE$. 
I.e., as a sheaf of abelian groups on $T$, 
$J_T^1(\sE)\simeq \sE\oplus (\Omega_T^1\otimes\sE)$,  and the $\sO_T$-module structure is given 
by $f(e,\alpha)=(fe,f\alpha-df\otimes e)$, where $f$, $e$ and $\alpha$ are 
local sections of $\sO_T$, $\sE$ and $\Omega_T^1\otimes\sE$, respectively. 
The \emph{Atiyah class} of $\sE$ is defined to be the element 
$at(\sE)\in H^1(T,\sE\textit{nd}(\sE)\otimes\Omega_T^1)$ 
corresponding to the Atiyah extension 
$$
0\to \Omega_T^1\otimes\sE \to J_T^1(\sE) \to \sE \to 0.
$$
It can be explicitly described as follows. 
Choose an affine open cover $(U_i)_{i\in I}$ of $T$ such that $\sE$ admits a frame
$f_i:\sO_{U_i}^{m+1} \overset{\sim}{\to}\sE|_{U_i}$ for each $U_i$. 
For $i,j\in I$, define 
$f_{ij}:={f_j^{-1}}|_{U_{ij}}\circ {f_i}|_{U_{ij}}$. Then 
$$
at(\sE)=\big[(-{f_j}|_{U_{ij}}\circ {df_{ij}}|_{U_{ij}}\circ {f_i^{-1}}|_{U_{ij}})_{i,j}\big]\in 
H^1(T,\sE\textit{nd}(\sE)\otimes\Omega_T^1).
$$ 
(See \cite[Proof of Theorem 5]{atiyah57}.).

Set $X:=\p_T(\sE)$, and denote by $\pi:X\to T$ the natural projection.
The push-forwarded Euler sequence
$0 \to \sO_T \to \sE\textit{nd}(\sE) \to \pi_* T_{X/T} \to 0$
yields a map
$$
H^1(T,\sE\textit{nd}(\sE)\otimes \Omega_T^1) \ \to \ 
H^1(T,\pi_*T_{X/T}\otimes \Omega_T^1) \ \simeq \ 
H^1(X,T_{X/T}\otimes \pi^*\Omega_T^1),
$$
where the last isomorphism is given by Leray's spectral sequence.
We denote by $\bar{at}(\sE)\in H^1(X,T_{X/T}\otimes \pi^*\Omega_T^1)$ 
the image of $at(\sE)$ under this map.

We claim that $\bar{at}(\sE)$ is the class in $ H^1(X,T_{X/T}\otimes \pi^*\Omega_T^1)$
of the exact sequence 
\begin{equation} \label{eq:atiyah}
0 \ \to \  T_{X/T} \ \to \  T_X \ \to \ \pi^*T_T\ \to \ 0 \ .
\end{equation}
To show this, we compute a cocycle that represents the extension class of \eqref{eq:atiyah}.
Let $(U_i)_{i\in I}$ be the affine open cover  of $T$ chosen above.
By shrinking $U_i$ if necessary, we may assume that $T_T$ admits a frame
$t_i:\sO_{U_i}^{l} \overset{\sim}{\to}{T_T}|_{U_i}$ for each $U_i$, where $l=\dim(T)$.
Let $t_i^{\vee}$ be dual frame of $\Omega_T^1$, and set
$\pi_i:=\pi|_{U_i}$.
The frame $f_i$ induces an isomorphism $U_i\times \p^m  \simeq \p_{U_i}(\sE|_{U_i})$ 
over $U_i$,
and a splitting $s_i:\pi_i^* ({T_T}|_{U_i})\to {T_{X}}|_{V_i}$
of \eqref{eq:atiyah} over $V_i:=\pi_i^{-1}(U_i)$.
For $i,j\in I$, define 
$$
a_{i,j}:= ({s_j}|_{V_{ij}} \otimes {{id}_{\pi^*\Omega_T^1}}|_{V_{ij}}
-{s_i}|_{V_{ij}} \otimes {{id}_{\pi^*\Omega_T^1}}|_{V_{ij}})
({\pi_i^*t_i}|_{V_{ij}}\otimes {\pi_i^*t_i^{\vee}}|_{V_{ij}}),
$$
where $t_i$ is viewed as a line vector whose entries are local sections of $T_T$, and
$t_i^{\vee}$ as a column vector. 
Then $\big[(a_{i,j})_{i,j}\big]\in H^1(X,T_{X/T}\otimes \pi^*\Omega_T^1)$
is a cocycle representing the class of \eqref{eq:atiyah}. 
Write $df_{ij}=(\alpha_{ijkn})_{k,n\in\{0,\ldots,m\}}$ with 
$\alpha_{ijkn}\in H^0(U_{ij},{\Omega_T^1}|_{U_{ij}})$ for $0\le k,n\le m$,
and set $\overline{df_{ij}}:=\Big(\alpha_{ijkn}y_n\frac{\partial}{\partial y_n}\Big)_{k,n\in\{0,\ldots,m\}}$,
where $(y_0:\cdots:y_m)$ are homogeneous coordinates on $\p^m$ 
associated to  the frame $f_i$, and 
$\alpha_{ijkn}y_n\frac{\partial}{\partial y_n}\in 
H^0\big(V_{ij},{T_{X/T}\otimes \pi^*\Omega_T^1}|_{V_{ij}}\big)$.
 We get
\begin{eqnarray*}
a_{i,j} & =  & \big({s_j}|_{V_{ij}} \otimes {{id}_{\pi^*\Omega_T^1}}|_{V_{ij}}
-{s_i}|_{V_{ij}} \otimes {{id}_{\pi^*\Omega_T^1}}|_{V_{ij}}\big)
\big({\pi_i^*t_i}|_{V_{ij}}\otimes {\pi_i^*t_i^{\vee}}|_{V_{ij}}\big)\\
& = & \big({s_j}|_{V_{ij}}({\pi_i^*t_i}|_{V_{ij}})-{s_i}|_{V_{ij}}({\pi_i^*t_i}|_{V_{ij}})\big)
\otimes {\pi_i^*t_i^{\vee}}|_{V_{ij}}\\
& = & \big(- {f_j}|_{U_{ij}}\cdot {{\overline{df_{ij}}}}|_{U_{ij}}\big)
\otimes {\pi_i^*t_i^{\vee}}|_{V_{ij}}.
\end{eqnarray*}
This proves our claim.
\end{say}

\begin{say}[Equivariance for locally free sheaves]\label{say:equivariant}
Let the notation be as in \ref{say:atiyah}.
Let $\sW$ be an invertible sheaf on $T$, and $V\in H^0(T,T_T\otimes\sW)$
a twisted vector field on $T$.
We say that $\sE$ is \emph{$V$-equivariant} if there exists a $\mathbb C$-linear map
$\tilde V: \sE\to \sW\otimes\sE$ lifting the derivation $V:\sO_T\to \sW$ 
(see \cite{carrell_lieberman}).
By \cite[Proposition 1.1]{carrell_lieberman}, $\sE$ is $V$-equivariant
if and only if $V_*at(\sE)\in H^1(T,\sE\textit{nd}(\sE)\otimes\sW)$
vanishes.
\end{say}

\begin{lemma}\label{lemma:extending_morphism}
Let $T$ be a smooth variety, and $\sE$ a locally free sheaf of rank $m+1\geq 1$ on $T$.
Set $X:=\p_T(\sE)$, denote by $\pi:X\to T$ the natural projection, and by $\sO_X(1)$
the tautological line bundle on $X$.
Suppose that there are locally free subsheaves
$i:\sH\into T_{X/T}$ and $j:\sQ\into T_T$ fitting into an exact sequence
$$
0 \ \to \ \sH \ \to \ \sF \ \to \ \pi^*\sQ \ \to \ 0.
$$
Denote by $e\in H^1(X,\sH\otimes \pi^*\sQ^*)$
the class of this extension.

\begin{enumerate}

\item There exists a morphism of $\sO_X$-modules 
$\sF \to T_X$ extending
$i:\sH \to T_{X/Y}$ and $\pi^*j:\pi^*\sQ \to \pi^* T_T$ if and only if 
$i_*e=j^*\bar{at}(\sE)$ in 
$H^1(X,T_{X/T}\otimes \pi^*\sQ^*)\simeq H^1(T,\pi_*T_{X/T}\otimes\sQ^*)$.

\item The set of morphisms of $\sO_X$-modules $\sF\to T_X$ extending $i$ and $\pi^*j$ is either empty,
or it is a torsor under
$Hom_{\sO_X}(\pi^*\sQ,T_{X/T})\simeq Hom_{\sO_T}(\sQ,\sE\textit{nd}(\sE))$.

\item Suppose that $\sH(-1)=\pi^*\sV$ for some locally free sheaf $\sV$ on $T$. 
Notice that
$i:\sH\to T_{X/T}$ induces a map $\sV\to \pi_*(T_{X/Y}(-1))\simeq \sE^*$.
Denote by $\bar e$ the image of $e$ under the composite map
$H^1(X,\sH\otimes \pi^*\sQ^*)\simeq H^1(T,\sV\otimes\sE\otimes\sQ^*)
\to H^1(T,\sE\textit{nd}(\sE)\otimes\sQ^*)$.
Then 
there exists a morphism of $\sO_X$-modules 
$\sF \to T_X$ extending
$i$ and $\pi^*j$ if and only if
$\bar e-j^*at(\sE) \in H^1(T,\sE\textit{nd}(\sE)\otimes\sQ^*)$
is in the image of the natural map
$H^1(T,\sQ^*)\to H^1(T,\sE\textit{nd}(\sE)\otimes \sQ^*)$.
\end{enumerate}
\end{lemma}

\begin{proof}
Notice that $i_*e$ is the extension class of the lower exact sequence
in the commutative diagram

\centerline{
\xymatrix{
0 \ar[r] &   \sH \ar[r]\ar[d]^{i} & \sF \ar[r]\ar[d] & \pi^*\sQ \ar@{=}[d]\ar[r] & 0 \\
0 \ar[r] &  T_{X/T} \ar[r] &  T_{X/T} \sqcup_\sH \sF \ar[r] & \pi^*\sQ \ar[r] & 0
}
}
\noindent where the left hand square is co-cartesian.
Similarly, $j^*(\bar{at}(\sE)) $ is the extension class of the upper exact sequence in
the commutative diagram

\centerline{
\xymatrix{
0 \ar[r] &  T_{X/T}  \ar[r]\ar@{=}[d] &  T_X \times_{\pi^*T_T} \pi^*\sQ \ar[r]\ar[d] & \pi^*\sQ \ar[d]^{\pi^*j}\ar[r] & 0 \\
0 \ar[r] &  T_{X/T} \ar[r] &  T_X \ar[r] & \pi^*T_T \ar[r] & 0
}
}
\noindent where the right hand square is cartesian.
Thus there exists a morphism of $\sO_X$-modules $k:\sF \to T_X$ that fits into a commutative diagram

\centerline{
\xymatrix{
0 \ar[r] &   \sH \ar[r]\ar[d]^{i} & \sF \ar[r]\ar[d]^{k} & \pi^*\sQ \ar[d]^{\pi^*j}\ar[r] & 0 \\
0 \ar[r] &  T_{X/T} \ar[r] &  T_X \ar[r] & \pi^*T_T \ar[r] & 0
}
}
\noindent if and only if $i_*e=j^*(\bar{at}(\sE))$ in 
$H^1(X,T_{X/T}\otimes \pi^*\sQ^*)\simeq H^1(T,\pi_*T_{X/T}\otimes\sQ^*)$.
This proves (1).

Let $k_1,k_2:\sF \to T_X$ be morphisms of $\sO_X$-modules
extending $i$ and $\pi^*j$. Then their difference
$k_1-k_2$ lies in $Hom_{\sO_X}(\sF,T_{X/T})\subset Hom_{\sO_X}(\sF,T_X)$, 
and $(k_1-k_2)|_{\sH}\equiv 0$. Conversely, if 
$\varphi \in Hom_{\sO_X}(\pi^*\sQ,T_{X/T})$, then $k_1+\varphi\circ p:\sF\to T_X$
extends $i$ and $\pi^*j$, where $p:\sF\to \pi^*\sQ$ is the surjective map from above.
This proves (2).

For statement $(3)$, 
observe that $j^*(\bar{at}(\sE))$ is the image of $at(\sE)$ under the composite map 
$$H^1(T,\sE\textit{nd}(\sE)\otimes\Omega_T^1 ) \to H^1(T,\sE\textit{nd}(\sE)\otimes\sQ^* )\to 
H^1(T,\pi_*T_{X/T}\otimes\sQ^*).$$
Moreover, $i:\sH\to T_{X/T}$ induces a map $\sV\to \pi_*(T_{X/Y}(-1))\simeq \sE^*$.
The class $i_*e$ is the image of $e\in H^1(\sH\otimes \pi^*\sQ^*)\simeq 
H^1(T,\sV\otimes\sE\otimes\sQ^*)$ under the composite map
$$H^1(T,\sV\otimes\sE\otimes\sQ^*) \to  H^1(T,\sE\textit{nd}(\sE)\otimes\sQ^* )\to 
H^1(T,\pi_*T_{X/T}\otimes\sQ^*)$$
since the map $\pi^*\sE\textit{nd}(\sE)\to T_{X/T}$ factors through
the natural map $\pi^*\sE\textit{nd}(\sE)\to \pi^*\sE^*(1)$.
The cohomology of the exact sequence 
$0 \to \sO_T \to \sE\textit{nd}(\sE) \to \pi_* T_{X/T} \to 0$
twisted by $\sQ^*$ yields the exact sequence
$$ H^1(T,\sQ^*) \to H^1(T,\sE\textit{nd}(\sE)\otimes\sQ^* )\to H^1(T,\pi_*T_{X/T}\otimes \sQ^*).$$
These observations put together prove (3). 
\end{proof}

We return to the problem of classifying locally free sheaves $\sE$
on $\p^l$ for which $X=\p_{\p^l}(\sE)$ admits a del Pezzo foliation $\sF\nsubseteq T_{X/\p^l}$.

\begin{thm}\label{thm:classification}
Let $\sE$ be an ample locally free sheaf of rank $m+1\geq 3$ on $\p^l$, 
and set $X:=\p_{\p^l}(\sE)$.
Denote by $\sO_X(1)$ the tautological line bundle on $X$, and by
$\pi:X\to \p^l$ the natural projection.
Let $r$ be an integer such that $2\leq r\leq m+l-1$.
Then there exists a foliation $\sF\nsubseteq T_{X/\p^l}$ of rank $r$ on $X$ 
such that $\det(\sF)\simeq\sO_X(r-1)$ if and only if
one of the following holds.
\begin{enumerate}
\item $l=1$, $r=2$, and $\sE\simeq\sO_{\p^1}(1)\oplus\sK$ for some ample vector bundle
$\sK$ on $\p^1$ such that $\sK\not\simeq\sO_{\p^1}(a)^{\oplus m}$ for
any integer $a$.
\item $l=1$, $r=2$, and 
$\sE\simeq \sO_{\p^1}(2)\oplus \sO_{\p^1}(a)^{\oplus m}$
for some integer $a \ge 1$.
\item $l=1$, $r= 3$, and 
$\sE\simeq \sO_{\p^1}(1)^{\oplus 2}\oplus \sO_{\p^1}(a)^{\oplus m-1}$
for some integer $a \ge 1$.
\item $l \ge 2$, $r=2$, and 
there exists a $V$-equivariant vector bundle $\sK$ on $\p^l$
for some
$V\in H^0(\p^l,T_{\p^l}\otimes\sO_{\p^l}(-1))\setminus \{0\}$
and an exact sequence 
$0\to\sK
\to\sE\to\sO_{\p^l}(1)
\to 0$. 
\end{enumerate}
\end{thm}

\begin{proof}
First we show that these are necessary conditions. 
Suppose  there exists a foliation $\sF\nsubseteq T_{X/\p^l}$ of rank $r$ on $X$ 
such that $\det(\sF)\simeq\sO_X(r-1)$.
By Theorem~\ref{thm:description}, we know that the possible values for the pair 
$(l,r)$ are $(1,2)$, $(1,3)$ and $(l,2)$, with $l\geq 2$.
Set $\sH:=\sF\cap T_{X/\p^l}$, and recall from Theorem~\ref{thm:description}
that $\sH\simeq (\pi^*\sV)(1)$, where $\sV\subset \sE^*$ is a subbundle of rank $r-1$.
Moreover if $l\geq 2$, then  $\sV\simeq \sO_{\p^l}(-1)$. 
If $l=1$, then either $\sV\simeq\sO_{\p^1}(-1)$, or $\sV\simeq\sO_{\p^1}(-2)$, 
or $\sV\simeq\sO_{\p^1}(-1)^{\oplus 2}$.
There is an inclusion $j:\det(\sV^*)\into T_{\p^l}$ .
Finally,  let $\sK$ be the kernel of the dual map $\sE\twoheadrightarrow \sV^*$, and consider
the $\p^{m-r+1}$-bundle $Z:=\p_{\p^l}(\sK)$, with natural projection $q:Z\to \p^l$.
By Theorem~\ref{thm:description}, $\sF$ is the pullback by the linear projection $X/\p^l\map Z/\p^l$
of a foliation on $Z$ induced by an inclusion $q^*\det(\sV^*)\subset T_Z$ that
lifts $q^*j:q^*\det(\sV^*)\into q^*T_{\p^l}$ .

Let  $at(\sK)\in H^1(\p^l,\sE\textit{nd}(\sK)\otimes\Omega_{\p^l}^1)$ be the Atiyah class of
$\sK$. Let $V\in H^0(\p^l,T_{\p^l}\otimes\det(\sV))$ be the section associated to  
$j:\det(\sV^*)\hookrightarrow T_{\p^l}$.
By Lemma \ref{lemma:extending_morphism},
there exists a map $q^*\det(\sV^*)\to T_Z$ lifting
$q^*j$
if and only if
$j_*at(\sK)\in H^1(\p^l,\sE\textit{nd}(\sK)\otimes\det(\sV))$ is in the image of
the natural map
$H^1(\p^l,\det(\sV))\to H^1(\p^l,\sE\textit{nd}(\sK)\otimes\det(\sV))$.

If $\sV\simeq\sO_{\p^l}(-1)$, then
$H^1(\p^l,\det(\sV))=0$. Thus 
there exists $q^*\det(\sV^*)\to T_Z$ lifting
$q^*j:q^*\det(\sV^*)\hookrightarrow q^*T_{\p^l}$ if and only if
$\sK$ is $V$-equivariant.
If $l\geq 2$, this is case (4) above.

From now on suppose $l=1$, and write $\sK\simeq \sO_{\p^1}(a_0)\oplus\cdots\oplus\sO_{\p^1}(a_{k})$,
with $k=m-r+1$.
In terms of this decomposition we have:

\medskip

\noindent 
$
\begin{array}{rcl}
{\begin{pmatrix}
at(\sO_{\p^1}(a_0)) & 0 & \cdots & 0\\
0 & \ddots & \ddots & \vdots \\
\vdots  & \ddots  & & 0\\
0  & \cdots & 0 & at(\sO_{\p^1}(a_{k}))\\
\end{pmatrix}}
& =  & at(\sK)
\end{array}
$

$
\hfill
\begin{array}{rcl}
\in H^1({\p^1},\sE\textit{nd}(\sK)\otimes\Omega_{\p^1}^1) & = &
{\begin{pmatrix}
H^1({\p^1},\sO_{\p^1}(a_0-a_0)\otimes\Omega_{\p^1}^1) & \cdots & 
H^1({\p^1},\sO_{\p^1}(a_0-a_{k})\otimes\Omega_{\p^1}^1)\\
\vdots & & \vdots \\
H^1({\p^1},\sO_{\p^1}(a_{k}-a_0)\otimes\Omega_{\p^1}^1) & \cdots & 
H^1({\p^1},\sO_{\p^1}(a_k-a_{k})\otimes\Omega_{\p^1}^1)\\
\end{pmatrix}}\\
\end{array}
$

\medskip

If $\det(\sV)\simeq\sO_{\p^1}(-2)\simeq \Omega_{\p^1}^1$,
then 
there exists a map $q^*\det(\sV^*)\to T_Z$ lifting
$q^*j$
if and only if
$a_0=\cdots=a_{k}$.
Since $\sE$ is an ample vector bundle, this implies that 
$\sE\simeq \sV^*\oplus \sK$.
This is case (2) or (3) above.

Finally, suppose that $\sV\simeq\sO_{\p^1}(-1)$.
Since $\sE$ is ample, we must have $\sE\simeq \sO_{\p^1}(1)\oplus \sK$,
and $\sK$ must be ample. 
Suppose that $\sK\simeq\sO_{\p^1}(a)^{\oplus n}$ for
some integer $a$. Then $Z\simeq\p^1\times\p^{m-1}$.
Denote by $g:Z\to \p^{m-1}$ the second projection.
Then 
$T_Z\otimes q^*\det(\sV)\simeq 
q^*\sO_{\p^1}(-1)\oplus \big(q^*\sO_{\p^1}(-1)\otimes g^*T_{\p^{m-1}}\big)$. Thus any 
nonzero global section of $T_Z\otimes q^*\det(\sV)$ vanishes along an hypersurface in $Z$, 
contradicting the fact that $q^*\det(\sV)$ is saturated in $T_Z$.

\medskip

Conversely, let us show that these are sufficient conditions.
Given $l$, $\sE$ and $r$ satisfying one of the conditions above, we will construct 
a foliation $\sF\nsubseteq T_{X/\p^l}$ of rank $r$ on $X$ 
such that $\det(\sF)\simeq\sO_X(r-1)$ in steps.
First we will find a vector bundle $\sV$ of rank $r-1$ on $\p^l$ 
fitting into an exact sequence of vector bundles 
$$
0 \ \to \ \sK \ \to \ \sE \ \to \ \sV^*\ \to 0,
$$
and such that 
there is an inclusion $j:\det(\sV^*)\hookrightarrow T_{\p^l}$.
We then set $Z:=\p_{\p^l}(\sK)$, with natural projection $q:Z\to \p^l$.
The exact sequence above  induces a rational map 
$p:X\map Z$ over $\p^l$, which restricts to a surjective morphism 
$p_0:X_0\to Z$, where $X_0$ is the complement in $X$ of the $\p^{r-2}$-subbundle 
$\p(\sV^*)\subset \p(\sE)$. 
Note that $\codim_X(X\setminus X_0)\geq 2$.
The next step consists of lifting the inclusion $j:\det(\sV^*)\hookrightarrow T_{\p^l}$
to an inclusion $q^*\det(\sV)\subset T_Z$.
We then check that $q^*\det(\sV)$ is saturated in $T_Z$, and let
$\sF$ be the unique saturated subsheaf of  $T_X$ extending
$dp_0^{-1}\big(q^*\det(\sV^*)\big)\subset T_{X_0}$.
It is a foliation on $X$ satisfying $\det(\sF)\simeq \sO_X(r-1)$.

\medskip

Case (1). Suppose that $l=1$, $r=2$, 
and $\sE\simeq\sO_{\p^1}(1)\oplus\sK$ for some ample vector bundle
$\sK$ on $\p^1$ such that $\sK\not\simeq\sO_{\p^1}(a)^{\oplus m}$ for
any integer $a$.
We set $\sV:=\sO_{\p^1}(-1)$ and 
let $j:\sV^*\simeq \sO_{\p^1}(1)\hookrightarrow T_{\p^1}$ be the inclusion associated to some 
$V\in H^0\big(\p^1,T_{\p^1}\otimes\sO_{\p^1}(-1)\big)\setminus \{0\}$. 
Then $\sK$ is 
$V$-equivariant, and so there exists a map $q^*\sO_{\p^1}(1)\into T_Z$ lifting $q^*j$.
It remains to show that $q^*\sO_{\p^1}(1)$ is saturated in $T_Z$.
To prove this, by Lemma~\ \ref{lemma:saturation}, it is enough to show that
$q^*\sO_{\p^1}(1)$ is a subbundle of $T_Z$ in codimension $1$. 
Suppose to the contrary that the map  $q^*\sO_{\p^1}(1)\to T_Z$ 
vanishes along an hypersurface $\Sigma$ in $Z$. Then the composed map 
$q^*\sO_{\p^1}(1)\to T_Z\to q^*T_{\p^1}$ vanishes along $\Sigma$, 
and $\Sigma$ must be a fiber of $q$.
By Lemma \ref{lemma:extending_morphism}, $\sK\simeq\sO_{\p^1}(a)^{\oplus n}$
for some $a\ge 1$, contradicting our assumptions.

\medskip

Case (2). Suppose that $l=1$, $r=2$, and 
$\sE\simeq \sO_{\p^1}(2)\oplus \sO_{\p^1}(a)^{\oplus m}$
for some integer $a \ge 1$.
We set $\sV:=\sO_{\p^1}(-2)$ and fix an isomorphism  $j:\sV^*\simeq T_{\p^1}$.
Then $Z\simeq\p^1\times\p^{m-1}$, and $q^*j$ lifts to a foliation 
$q^*T_{\p^1}\subset T_Z$.

\medskip

Case (3). Suppose that $l=1$, $r= 3$, and
$\sE\simeq \sO_{\p^1}(1)^{\oplus 2}\oplus \sO_{\p^1}(a)^{\oplus m-1}$
for some integer $a \ge 1$.
We set $\sV:=\sO_{\p^1}(-1)^{\oplus 2}$ and  fix an isomorphism  
$j:\det(\sV^*)\simeq T_{\p^1}$.
Then $Z\simeq\p^1\times\p^{m-2}$, and $q^*j$ lifts to a foliation 
$q^*T_{\p^1}\subset T_Z$.

\medskip

Case (4). Suppose that $l \ge 2$, $r=2$, and 
there exists a $V$-equivariant vector bundle $\sK$ on $\p^l$
for some
$V\in H^0(\p^l,T_{\p^l}\otimes\sO_{\p^l}(-1))\setminus \{0\}$
and an exact sequence 
$0\to\sK
\to\sE\to\sO_{\p^l}(1)
\to 0$. 
We set $\sV:=\sO_{\p^l}(-1)$ and 
let $j:\sV^*\simeq \sO_{\p^l}(1)\hookrightarrow T_{\p^l}$ be the inclusion associated to 
$V$. 
Since $\sK$ is $V$-equivariant, there exists a map $q^*\sO_{\p^l}(1)\into T_Z$ lifting $q^*j$.
It remains to show that $q^*\sO_{\p^l}(1)$ is saturated in $T_Z$.
To prove this, by Lemma~\ \ref{lemma:saturation}, it is enough to show that
$q^*\sO_{\p^l}(1)$ is a subbundle of $T_Z$ in codimension $1$. 
Suppose to the contrary that the map  $q^*\sO_{\p^l}(1)\to T_Z$ 
vanishes along an hypersurface $\Sigma$ in $Z$. Then the composed map 
$q^*\sO_{\p^l}(1)\to T_Z\to q^*T_{\p^l}$ vanishes along $\Sigma$, and
$q(\Sigma)$ has codimension $1$ in $\p^l$.
This is saying that 
$\sO_{\p^l}(1)\into T_{\p^l}$ vanishes in codimension 1 in $\p^l$, which is impossible since 
$l\ge 2$.
\end{proof}

\begin{lemma}\label{lemma:saturation}
Let $X$ be a normal variety, and
$\sE\subset\sF$ coherent sheaves of $\sO_X$-modules,  with
$\sE$ locally free and $\sF$ torsion-free. Then $\sE$ is saturated in $\sF$ 
if and only if $\sE$ is a subbundle of $\sF$ in codimension $1$.
\end{lemma}

\begin{proof}
To say that $\sE$ is saturated in $\sF$ is equivalent to saying that $\sF/\sE$ is torsion-free.
To say that $\sE$ is a subbundle of $\sF$ in codimension $1$ is equivalent to saying that 
$\sF/\sE$ is locally free in codimension $1$.
Since $X$ is normal, if $\sF/\sE$ is torsion-free, then it is locally free in codimension $1$.
Conversely, suppose  that $\sF/\sE$ is locally free in codimension $1$, and let us show that 
$\sF/\sE$ is torsion-free.
Let $f$ be a nonzero local section of $\sO_X$ and $s$ a local section of $\sF$
such that $fs$ is a local section of $\sE$. 
Since $\sF$ is torsion-free, $s$ is a rational section of $\sE$.
By assumption, $s$ is regular in codimension $1$.
Since $X$ is normal and $\sE$ is locally free, 
it follows that $s$ is a regular local section of $\sE$.
\end{proof}

\begin{rem}
Lemma \ref{lemma:saturation} fails to be true if $\sF$ is not torsion-free. 
Let $\sF:=\sE\oplus\sT$ where $\sT$ is a torsion sheaf whose support has codimension $\ge 2$ in $X$. 
Then $\sE$ is a subbundle of $\sF$ in codimension $1$ but $\sF/\sE=\sT$ is a torsion sheaf.
\end{rem}

\begin{lemma}\label{lemma:local_freeness}
Let $0\to \sH \to \sF \to \sQ \to 0$ be an exact sequence of coherent sheaves
on a noetherian integral Cohen-Macaulay scheme $X$. 
Suppose that $\sF$ is reflexive,
$\sH$ and $\sQ^{**}$ are locally free, and
$\sQ$ is 
locally free in codimension $2$.
Then $\sF$ and $\sQ$ are locally free.
\end{lemma}

\begin{proof}
By hypothesis, there is an open subset $U\subset X$, with $\codim_X(X\setminus U)\geq 3$,
such that $\sH|_{U}$ and $\sQ|_{U}$ are locally free.
Moreover $\sH\otimes\sQ^{*}$ is locally free, and hence a Cohen-Macaulay sheaf. 
So we have $H^1(X,\sH\otimes\sQ^{*})\simeq H^1\big(U,({\sH\otimes\sQ^{*}})|_{U}\big)$. 
Therefore, the extension class of
the exact sequence $0\to \sH|_{U} \to \sF|_{U} \to \sQ|_{U} \to 0$ on $U$ yields
an exact sequence $0\to \sH \to \sF' \to \sQ^{**} \to 0$ on $X$ such that 
$\sF'|_{U}\simeq\sF|_{U}$. 
Since $X$ is reduced and both $\sH$ and $\sQ^{**}$ are locally free,
so is $\sF'$.
Since $\sF$ is reflexive, $\sF\simeq\sF'$ by \cite[Proposition 1.6]{hartshorne80}.
\end{proof}

Our next goal is to construct del Pezzo foliations $\sF$ on 
projective space bundles $X\to T$ such that $\sF\subsetneq T_{X/T}$.
These can be viewed as families of del Pezzo foliations on 
projective spaces.

\begin{const}
Let $T$ be a positive dimensional smooth projective variety, 
$\sE$ an ample locally free sheaf of rank $m+1\geq 2$ on $T$, and set $X:=\p_T(\sE)$.
Denote by $\sO_X(1)$ the tautological line bundle on $X$, and
by $\pi:X\to T$ the natural projection.
Let $r$ be an integer such that $2 \le r \le m-1$.

We explain how to construct del Pezzo foliation $\sF$ of rank $r$ on $X$ such that 
$\sF\subsetneq T_{X/T}$.

Suppose that there are  locally free sheaves $\sQ$ and $\sK$ on $T$, of 
rank $r-1\ge 1$ and $m-r+2 \ge 3$, respectively, fitting into 
an exact sequence 
$$
0 \  \to \ \sK \ \to \ \sE \ \to \ \sQ \ \to \ 0.
$$
In particular, $\sQ$ is ample.
Denote by $e\in H^1(T,\sK\otimes\sQ^*)$ the class of this extension.

Set $Z:=\p_T(\sK)\to T$, denote by $\sO_Z(1)$ the tautological line bundle on $Z$, and
by $q:Z\to T$ the natural projection.
Recall the pushed-forwarded Euler's sequence: 
\begin{equation}\label{eq:euler}
0 \ \to \ \sO_T \ \to \ \sE\hspace{-0.05cm}\textit{nd}_{\sO_T}(\sK) \ \to \ q_*T_{Z/T} \ \to \ 0.
\end{equation}

Let $c:Y\to X$ be the blow up of
$X$ along $L:=\p_T(\sQ)$ where the inclusion $L\subset X$ is induced by the surjection
$\sE\twoheadrightarrow \sQ$.
Let $E\subset Y$ be the exceptional divisor of $c$. By Leray's spectral sequence,
there is a natural isomorphism
$H^1(Z,\sO_Z(1)\otimes q^*\sQ^*)\simeq H^1(T,\sK\otimes \sQ^*)$. 
Let $\sH$ be the 
vector bundle on $Z$ associated to the image of $e$ in $H^1(Z,\sO_Z(1)\otimes q^*\sQ^*)$.
Then $\sH$ has rank $r$, $q_*\sH=\sE$, and 
there is a commutative diagram of exact sequences:

\centerline{
\xymatrix{
0 \ar[r] 
 &  q^*\sK  \ar[r]\ar[d] 
 &  q^*\sE \ar[r]\ar[d] 
 &  q^*\sQ \ar@{=}[d]\ar[r] 
 & 0 \\
0 \ar[r] 
 &  \sO_Z(1) \ar[r] 
 &  \sH \ar[r] 
 &  q^*\sQ \ar[r] & 0.
}
}

Denote by $g:\p_Z(\sH) \to Z$ the natural projection, and by 
$\sO_{\p_Z(\sH)}(1)$ the the tautological line bundle. 
Observe that there 
is an isomorphism $Y\simeq \p_Z(\sH)$ 
that fits into the commutative diagram

\centerline{
\xymatrix{
& Y\simeq\p_Z(\sH)\ar[dl]_{c}\ar[dr]^{g}\ar[dd] & \\
X=\p_T(\sE) \ar[dr]_{\pi} & & Z=\p_T(\sK), \ar[dl]^{q}\\
& T &
}
}
\noindent 
where $c$ is induced by the surjection
$$
c^*\sE=g^*(q^*\sE) \twoheadrightarrow g^*\sH\twoheadrightarrow \sO_{\p_Z(\sH)}(1).
$$

In order to construct a del Pezzo foliation $\sF$ on $X$ we make the following assumptions.
\begin{enumerate}
	\item There is an injective map 
	$\phi:\det(\sQ)\hookrightarrow \sE\hspace{-0.05cm}\textit{nd}_{\sO_T}(\sK)$ 
	(equivalently, 
	$h^0\big(T,\det(\sQ)^*\otimes \sE\hspace{-0.05cm}\textit{nd}_{\sO_T}(\sK)\big)\neq 0$).
	\item The inclusion $q^*\det(\sQ)\subset T_{Z/T}\subset T_Z$ induced by $\varphi$ 
	via \eqref{eq:euler} defines a foliation on $Z$.
	By Lemma~\ref{lemma:saturation} this is equivalent to requiring that the map 
	$q^*\det(\sQ)\into T_{Z/T}$ is nonzero in codimension $1$.		
\end{enumerate}

We then set 
$\sF_Y:=dg^{-1}(q^*\det(\sQ))\subset T_Y$, and $\sF:=df(\sF_Y)\subset T_{X/T}\subset T_X$. 

For any $t\in T$, let $v_t\in End_{\mathbb C}(\sK_t)$ be an endomorphism induced by $\varphi$
at $t\in T$.
Then the foliation on $X_t\simeq\p^m$ induced by $\sF$ is
the linear pullback of a foliation $\sG_t$ on $Z_t\simeq \p^{m-r+1}$ induced by the global holomorphic vector field $\vec v_t$ associated to $v_t$. 
One can prove that
the closure of a general leaf of $\sG_t$ 
is a rational curve $C$ meeting the singular locus of $\sG_t$ at a single point.
Moreover, either this point is a cusp on $C$, or $\vec v_t$ 
viewed as a local vector field on $C$ vanishes with multiplicity at least $2$.

\medskip 

Once assumption (1) above is fulfilled, we investigate when assumption (2) holds.

First we claim that  $v_t$ is a nilpotent endomorphism for any $t\in T$. 
Indeed,
the composite map
$$
\begin{array}{ccccc}
\det(\sQ)^{\otimes k} & \longrightarrow & 
\sE\hspace{-0.05cm}\textit{nd}_{\sO_T}(\sK)^{\otimes k} & \longrightarrow 
& \sO_T \\
 \alpha_1\otimes\cdots\otimes\alpha_k & \longmapsto & 
 \phi(\alpha_1)\otimes\cdots\otimes\phi(\alpha_k) & \longmapsto & 
\textup{Tr}(\phi(\alpha_1)\circ\cdots\circ\phi(\alpha_k)) 
\end{array}
$$
is zero since $\det(\sQ)$ is ample. Thus
$\textup{Tr}(\underbrace{v_t\circ\cdots\circ v_t}_{k \text{ times}})=0$ for any $k\ge 1$,
showing that $v_t$ is  nilpotent.

Notice that $q^*\det(\sQ)$ is saturated in $T_{Z/T}$  if and only if the following holds.
For a general point $t\in T$, $v_t$ has rank $\ge 2$, and 
there exists an open subset $T_0\subset T$, with $\codim_T(T\setminus T_0)\ge 2$, such that
$v_t$ has rank $\ge 1$ for any point $t\in T_0$.

Finally, observe that the assumptions are fulfilled 
if $\sK$ contains $\det(\sQ)\oplus\det(\sQ)^{\otimes 2}\oplus\det(\sQ)^{\otimes 3}$ as a direct summand, and $\det(\sQ)\hookrightarrow
\sE\hspace{-0.05cm}\textit{nd}_{\sO_T}\big(\det(\sQ)\oplus\det(\sQ)^{\otimes 2}\oplus\det(\sQ)^{\otimes 3}\big)$ is associated to 

$$
\begin{pmatrix}
0 & 0 & 0\\
1 & 0 & 0\\
0 & 1 & 0
\end{pmatrix}
.$$

\medskip

\end{const}

We end this section by addressing Fano Pfaff fields on projective space bundles.

\begin{prop}\label{lemma:projective_space_bundle3}
Let $T$ be a smooth projective variety, 
$\sE$ a locally free sheaf of rank $m+1\geq 2$ on $T$, and set $X:=\p_T(\sE)$.
Denote by $\sO_X(1)$ the tautological line bundle on $X$, and
by $\pi:X\to T$ the natural projection.
Let $r\ge 2$ be an integer. 
\begin{enumerate}
\item If $r\ge m+3$ then $h^0\big(X,\wedge^{r}T_{X}(-r+1)\big)=0$.
\item If $h^0\big(X,\wedge^{r}T_{X}(-r+1)\big)\neq 0$, then 
$h^0\big(X,\wedge^{r-s}T_{X/T}(-r+1)\otimes \pi^*(\wedge^s T_T)\big)\neq 0$
for some $s\in\{0,1,2\}$. If $h^0\big(X,\wedge^{r-2}T_{X/T}(-r+1)\otimes \pi^*(\wedge^2 T_T)\big)\neq 0$ then $r=m+2\ge 3$.
\item If $\sE$ is an ample vector bundle and 
$h^0\big(X,\wedge^{r-2}T_{X/T}(-r+1)\otimes \pi^*(\wedge^2 T_T)\big)\neq 0$,
then either
$T\simeq\p^1\times\p^1$ and $r-2=m=1$, or
$T\simeq \p^2$ and $r-2=m=2$, or
$T\simeq \p^l$ ($l\ge 2$) and $r-2=m=1$. 
\item If $\sE$ is an ample vector bundle, $l=\dim(T)\ge 1$, 
$h^0\big(X,\wedge^{r-1}T_{X/T}(-r+1)\otimes \pi^*T_T\big)\neq 0$ and
$\rho(T)=1$, 
then $T\simeq\p^{l}$ and $r\le 3$.
\end{enumerate}
\end{prop}

\begin{proof}
The short exact sequence
$$0\to T_{X/T} \to T_X \to \pi^*T_T \to 0$$
yields a filtration 
$$\wedge^rT_X=F_0\supset F_1\supset\cdots\supset F_{r+1}=0$$
such that
$$F_i/F_{i+1}\simeq\wedge^iT_{X/T}\otimes \pi^* (\wedge^{r-i}T_T).$$
By Bott's formulae, $h^0\big(\p^m,\wedge^{i} T_{\p^m}(-r+1)\big)=0$
if either $0\le i \le r-3$ and $r-2 \le m$, or $i=r-2$ and $r-2<m$. 
So in these cases we have 
$h^0\big(X,(F_i/F_{i+1})(-r+1)\big)=0$,
proving (1) and (2). 

\medskip

From now on suppose that $\sE$ is an ample vector bundle.
If
$h^0\big(X,\wedge^{r-2}T_{X/T}(-r+1)\otimes \pi^*(\wedge^2 T_T)\big)\neq 0$,
then $r-2=m \ge 1$, $r\ge 3$, and $\wedge^{r-2}T_{X/T}(-r+1)\simeq \pi^*\det(\sE^*)$.
Hence $h^0\big(T,\wedge^2 T_T\otimes \det(\sE^*)\big)\neq 0$. 
By \cite{druel_paris}, 
either
$T\simeq\p^1\times\p^1$ and $r-2=m=1$, or
$T\simeq \p^2$ and $r-2=m=2$, or
$T\simeq \p^l$ ($l\ge 2$) and $r-2=m=1$, proving (3).

\medskip

Now suppose that $\rho(T)=1$, and $h^0\big(X,\wedge^{r-1}T_{X/T}(-r+1)\otimes \pi^*T_T\big)\neq 0$.
Euler's sequence
$$
0 \ \to \ \sO_X \ \to \ \pi^*\sE^*(1) \ \to \ T_{X/T} \ \to \ 0
$$
induces an exact sequence 
$$
0 \ \to \ \wedge^{r-2}T_{X/T}(-r+1) \ \to \
\wedge^{r-1}\pi^*(\sE^*)
\ \to \ \wedge^{r-1}T_{X/T}(-r+1) 
\ \to \ 0.
$$
By Bott's formulae, $h^0\big(\p^m,\wedge^{r-2} T_{\p^m}(-r+1)\big)=0$ and 
$h^1\big(\p^m,\wedge^{r-2} T_{\p^m}(-r+1)\big)=0$. 
Hence $\pi_*\big(\wedge^{r-2}T_{X/T}(-r+1)\big)=0$
and
$R^1\pi_*\big(\wedge^{r-2}T_{X/T}(-r+1)\big)=0$. 
Thus, by pushing forward by $\pi$ the above exact sequence, we conclude that
$\wedge^{r-1}\sE^*\simeq \pi_*\big(\wedge^{r-1}T_{X/T}(-r+1)\big)$. 
The projection formula then yields an isomorphism 
$$
H^0(T,\wedge^{r-1}\sE^*\otimes T_T)\simeq
H^0\big(X,\wedge^{r-1}T_{X/T}(-r+1)\otimes p^*T_T\big)\neq \{0\}.
$$
By \cite[Corollary 4.3]{aprodu_kebekus_peternell08}, we must have
$T\simeq\p^{l}$.

A nonzero section of $\wedge^{r-1}\sE^*\otimes T_{\p^l}$ yields a nonzero map 
$\alpha:\wedge^{r-1}\sE\to T_{\p^l}$.
Let $\ell\subset \p^l$ be a general line, and write 
$\sE|_{\ell}\simeq \sO_{\p^1}(a_0)\oplus\cdots\oplus\sO_{\p^1}(a_m)$, with
$1\le a_0\le \cdots\le a_m$.
Then $\alpha$ induces a nonzero map 
$$
\wedge^{r-1}\big(\sO_{\p^1}(a_0)\oplus\cdots\oplus\sO_{\p^1}(a_m)\big)
\to \sO_{\p^1}(2)\oplus\sO_{\p^1}(1)^{\oplus l-1}.
$$ 
Thus $r-1\le a_0+\cdots +a_{r-2}\le 2$, proving $(4)$.
\end{proof}

\bibliographystyle{amsalpha}
\bibliography{foliation}

\end{document}